\documentclass[11pt]{article}

\RequirePackage{amsthm,amsmath,amsfonts,amssymb}
\RequirePackage[numbers,sort&compress]{natbib}
\RequirePackage{hyperref}
\RequirePackage{graphicx}
\usepackage{color}
\usepackage{newlfont}
\usepackage{bbm}
\usepackage{mathtools}

\usepackage{physics}
\usepackage{enumerate}
\usepackage[title]{appendix}
\usepackage{hyperref}
\usepackage{dsfont}
\usepackage{comment}
\usepackage{array}
\usepackage{etoolbox}
\usepackage{enumitem}
\usepackage{cases}
\usepackage{subfigure}
\usepackage{subcaption}

\numberwithin{equation}{section}
\newcounter{countD}

\theoremstyle{plain}

\newtheorem{claim}{Claim}
\newtheorem{claim2}[countD]{Claim}
\newtheorem{Theorem}{Theorem}[section]
\newtheorem{Lemma}[Theorem]{Lemma}
\newtheorem{Proposition}[Theorem]{Proposition}

\theoremstyle{definition}
\newtheorem{Definition}[Theorem]{Definition}
\newtheorem{remark}[Theorem]{Remark}
\begin{document}

\title{Asymptotic analysis on narrow tubes: narrow escape problems and diffusion processes}
\author{Wen-Tai Hsu\footnote{Department of Mathematics, University of Maryland, College Park, MD 20742, wthsu@umd.edu}}

\date{\vspace{-5ex}}

\maketitle

\begin{abstract}
This paper investigates a diffusion process in a narrow tubular domain with reflecting boundary conditions, where the geometry serves as a singular perturbation of an underlying graph in $\mathbb{R}^2$ or $\mathbb{R}^3$.
The construction incorporates distinct scaling regimes in the neighborhoods of the graph's vertices and edges.
We show that, in the limit, the projected process converges weakly to a diffusion process on the graph, with gluing conditions at the vertices that depend on the relative scales of the neighborhoods.
Our analysis relies on a detailed understanding of the narrow escape problem in domains with bottlenecks.
In particular, we rigorously derive the asymptotic behavior of the expected escape time, establish the asymptotic exponential distribution of escape times, and obtain exit place estimates, results that may be of independent interest.
\end{abstract}

\setcounter{tocdepth}{1}
\tableofcontents

\section{Introduction}
Narrow domain problems arise in various areas of physics, such as the propagation of electrical signals along nearly one-dimensional neuronal structures, the diffusion and transport of proteins in cellular networks, and ocean dynamics where the vertical scale is small compared to the horizontal extent; see, e.g., \cite{DPT}, \cite{QGbook}, \cite{NET}, \cite{NETnetwork}, \cite{DST} for further examples.
These phenomena frequently occur in geometrically constrained environments, leading to complex dynamics that are often intractable to analyze directly within the full domain.
A common approach to overcome this difficulty is to approximate the narrow region by a lower dimensional object, such as an interval or a graph.
However, rigorously justifying such graph approximations is highly nontrivial, especially in settings involving multiple interacting scales, and addressing this challenge is a central goal of this paper.

In the present paper, we consider a simple, finite, connected graph $\Gamma = (V,E) \subset \mathbb{R}^d$, with vertices $O_1,\cdots,O_{\abs{V}} \in V$ and edges $I_1,\cdots,I_{\abs{E}} \in E$, where $d=2$ or $3$.
Here, $V$ and $E$ denote the sets of vertices and edges, respectively, and $\abs{V}$ and $\abs{E}$ their cardinalities.
The associated narrow tube domain $G_\epsilon$ consisting of the union of $\lambda_k \epsilon$-neighborhoods of the edges $I_k$ and $\rho_j r_j(\epsilon)$-neighborhoods of the vertices $O_j$, where $\rho_j, \lambda_k$ are positive constants and $r_j(\epsilon) \downarrow 0$ as $\epsilon \downarrow 0$.

We aim to study the diffusion process with reflecting boundary conditions on the narrow tube $G_\epsilon$
\begin{equation}\label{narrow diffusion}
    dZ^\epsilon(t) = \sqrt{2}dB(t) + \nu_\epsilon (Z^\epsilon(t)) d\phi^\epsilon(t), \ Z^\epsilon(0) = z \in G_\epsilon.
\end{equation}
Here $B(t)$ is a $d$-dimensional Brownian motion, $\nu_\epsilon(z)$ is the unit inward normal vector at the point $z \in \partial G_\epsilon$, and $\phi^\epsilon(t)$ is the local time of $Z^\epsilon(t)$ on the boundary $\partial G_\epsilon$.

From a mathematical perspective, this setting corresponds to a domain singular perturbation problem, involving diffusion in a domain that degenerates toward a lower-dimensional structure.
Specifically, we start with a one-dimensional graph $\Gamma$ and perturb it into a $d$-dimensional tubular domain $G_\epsilon$.
A natural question is how the geometry of $G_\epsilon$ influences the limiting dynamics, in particular, how the scaling of $r_j(\epsilon)$ affects the behavior near the vertices. 
We explore this question in detail in the present work.
(For related PDE results in this spirit for simpler domains, typically involving only a single small scale $\epsilon$, or domains that can be decomposed into several such simple pieces but possibly with more complex differential operators, see, e.g., the lecture note \cite{DST}, recent developments \cite{nlpde}, and the references therein.)

Another perspective, drawn from the averaging principle as in \cite{SDE2}, reveals a local three-scale structure near each vertex. 
More precisely, the diffusion process has a slow component corresponding to motion along the edge direction outside the vertex neighborhood, and two fast components: one arising from motion in directions perpendicular to the edge outside the vertex neighborhood, due to the narrow cross sectional width of order $\epsilon$; 
and another from rapid motion within the vertex neighborhood itself, modeled as shrinking balls with size controlled by $r_j(\epsilon)$. 
This separation of scales plays a central role in determining the limiting behavior. 
The effective dynamics depends critically on the interplay between $\epsilon$, $r_j(\epsilon)$, and the dimension $d$.

In \cite{SDE2}, the asymptotic behavior of \eqref{narrow diffusion} is analyzed in the regime where $r_j(\epsilon)$ and $\epsilon$ are of the same order.
The authors show that $\Pi^\epsilon \circ Z^\epsilon$ converges weakly in $C([0,\infty); \Gamma)$ to a diffusion process $\bar{Z}$ governed by the generator $\bar{L}$. 
Here, $\Pi^\epsilon$ denotes a continuous modification of the nearest projection $\Pi$ from $G_\epsilon$ onto the graph $\Gamma$.
The limiting operator $\bar{L}$ acts as the second derivative in the interior of each edge and satisfies the following Neumann-type gluing condition at each vertex
\begin{equation*}
	\sum_{k:I_k \sim O_j}   p_{j,k} \frac{d f}{d x_k}(O_j) 
    = 0,
\end{equation*}
where $I_k \sim O_j$ indicates that $O_j$ is a vertex of the edge $I_k$, the constants $p_{j,k} > 0$ and $\frac{d }{d x_k}$ denotes the directional derivative along edge $I_k$. See Section \ref{sec diff graph prelim} and \ref{sec diff tube prelim} for further details.

In \cite{SDE1}, a more general class of $r_j(\epsilon)$ is considered, and the author conjectures that the limiting process remains a diffusion on $\Gamma$, governed by gluing conditions of the form 
\begin{equation}\label{intro glu}
	\sum_{k:I_k \sim O_j}   p_{j,k} \frac{d f}{d x_k}(O_j) 
    = \alpha_j \bar{L}f(O_j),
\end{equation}
where $\bar{L}f(O_j)$ denotes the common second derivative value of $f$ at the vertex $O_j$ along each edge $I_k$ and $p_{j,k}, \alpha_j$ are non-negative constants depending on the scaling behavior of $r_j(\epsilon)$.
Three distinct scaling regimes are identified:
\begin{itemize}
    \item Small balls: $\lim_{\epsilon \to 0}  \frac{r_j(\epsilon)^d}{\epsilon^{d-1}} = 0
    \ \Rightarrow \ p_{j,k} =\frac{\lambda_k^{d-1}}{\sum_{l:I_l \sim O_j} \lambda_l^{d-1}}, \quad \alpha_j =0 $,
    \item Intermediate balls: $\lim_{\epsilon \to 0}  \frac{r_j(\epsilon)^d}{\epsilon^{d-1}} \in (0,\infty)
    \\ \qquad \qquad \Rightarrow \ p_{j,k}  =\frac{\lambda_k^{d-1}}{\sum_{l:I_l \sim O_j} \lambda_l^{d-1}}, \quad \alpha_j = \lim_{\epsilon \to 0} \frac{ \rho_j^d r_j(\epsilon)^d V_d}{\sum_{k:I_k \sim O_j} \lambda_k^{d-1} \epsilon^{d-1} V_{d-1}}.$
    \item Large balls: $\lim_{\epsilon \to 0}  \frac{r_j(\epsilon)^d}{\epsilon^{d-1}} = \infty
    \ \Rightarrow \ p_{j,k} = 0, \quad \alpha_j =1 $.
\end{itemize}
Here, $V_d = \frac{\pi^{d/2}}{\Gamma(d/2+1)}$ denotes the volume of the unit ball in $\mathbb{R}^d$, where $\Gamma$ is the Gamma function.
Intuitively, larger vertex neighborhoods lead to longer residence times for particles near the vertices, introducing delay effects that change the gluing conditions.
In the small ball regime, particles pass through vertices rapidly, yielding behavior similar to that described in \cite{SDE2}.
In the large ball regime, vertices effectively behave like absorbing states\footnote[2]{In some literature, the term absorbing boundary refers to a Dirichlet boundary condition, i.e., $f(O_j)=0$. Here, we follow the terminology of Feller (see, e.g., \cite{Feller54}, \cite{Ito}) and refer to the boundary corresponding to the Dirichlet condition as an escape region.}: once the process reaches a vertex, it remains there indefinitely.
The intermediate regime corresponds to finite, non-negligible delays at the vertices, similar to sticky boundary behavior in diffusion processes.

The main result of this paper establishes the weak convergence of $\Pi^\epsilon \circ Z^\epsilon$ converges weakly to $\bar{Z}$ in $C([0,\infty);\Gamma)$; see Theorem \ref{prevconv}.
Our approach, following \cite{SDE2}, is based on the martingale problem framework: 
we first establish tightness of the laws of $\Pi^\epsilon \circ Z^\epsilon$, and then show that every limit point solves the martingale problem associated with the generator $\bar{L}$.
The uniqueness of the solution to this martingale problem then implies convergence in distribution.
The main difficulty arises in the second step, where we must analyze the time the diffusion spends near the vertices.
When $r_j(\epsilon)$ and $\epsilon$ are comparable, this can be handled using a change of variables combined with Brownian scaling arguments.
In the general case, however, such methods fail, and we instead draw on techniques from narrow escape problems to estimate the relevant exit times.

The narrow escape problem, which concerns Brownian motion in domains with reflecting boundaries except for small escape regions, has been extensively studied in the applied mathematics and physics literature (see, e.g., the review article \cite{NET} and the monograph \cite{NETbook} for an overview of the narrow escape problem, and \cite{NETthird}, \cite{Li14}, \cite{LL24}, \cite{NETBN1} and also \cite{D+03} for domains featuring bottlenecks).
By contrast, relatively fewer works have addressed the mathematical aspects:
in \cite{LPNET} and \cite{NET4}, the authors study the mean first exit time from general smooth domains without bottlenecks in $\mathbb{R}^2$ and $\mathbb{R}^3$.
The work \cite{NETsecond} examines the exit behavior from a disk through an exit arc, while \cite{NETsp}, \cite{CLNV26}, \cite{C+26} investigates the first exit time and exit point distribution under quasi-stationary initial data.

Our exit problem closely resembles the narrow escape problem in bottleneck domains, as examined in \cite{NETthird}, \cite{NET}, and \cite{NETBN1}.
(Strictly speaking, in the first two regimes, the escape is not truly “narrow.”)
More precisely, let $C_{\epsilon,j}(\delta)$ denote the boundary of $\delta$-neighborhood of $O_j$ in $G_\epsilon$, and consider the first hitting time $\sigma^{\epsilon,\delta}$ of $\bigcup_{j=1}^{\abs{V}} C_{\epsilon,j}(\delta)$ for initial points near $O_j$ (see Section \ref{sec prelim} for the precise definition).
Our main interest lies in understanding the asymptotic behavior of the expected hitting time $\mathbb{E}_z \sigma^{\epsilon,\delta} $ as $\epsilon \downarrow 0$. 
One of the main technical challenges in proving the convergence of the diffusion processes is to establish the following asymptotic estimate:
for each $j=1,\cdots,\abs{V}$,
\begin{equation}\label{intro est}
	\lim_{\epsilon \to 0} \sup_{z \in C_{\epsilon,j}(\rho_j r_j(\epsilon)+3\epsilon)}
	\abs{
	\mathbb{E}_z \sigma^{\epsilon,\delta}  \left( \frac{\delta^2}{2} + \alpha_j(\epsilon) \delta \right)^{-1} 
	-1
	}
	=0,
\end{equation}
where 
\begin{equation*}
	\alpha_j (\epsilon) \coloneqq  \frac{ \rho_j^d r_j(\epsilon)^d V_d}{\sum_{k:I_k \sim O_j} \lambda_k^{d-1} \epsilon^{d-1} V_{d-1}}.
\end{equation*}

To the best of our knowledge, a rigorous derivation of such estimates was open in the current literature.
The key challenge in domains with bottlenecks, as opposed to general smooth domains, lies in the presence of multiple spatial scales, specifically, the disparity between the neck and the ball regions, which introduces asymptotic nonsmoothness. 
As a result, standard approaches based on Neumann or Green's functions, such as those used in \cite{LPNET} and \cite{NET4}, do not apply directly. Indeed, as predicted in \cite{NET}, the presence of bottlenecks intensifies the singular behavior of the solution. 
In particular, in two dimensions, the singularity strengthens from logarithmic to algebraic.

The main argument in the aforementioned works (\cite{NETthird}, \cite{Li14}, \cite{LL24} and \cite{NETBN1}) proceeds as follows: first solve the exit problem in the neck region using a suitable compatibility condition, and then use a form of continuity for $\mathbb{E}_z \sigma^{\epsilon,\delta} = v_{\epsilon,\delta}(z)$ to reduce the problem to an escape problem within the ball.
In the first step, the key assumption is that the PDE on the neck can be effectively approximated by a one dimensional problem, allowing for an explicit solution in this region.
In the second step, while the continuity of $v_{\epsilon,\delta}(z)$ follows from classical PDE theory, connecting the neck and the ball requires continuity uniform in $\epsilon$, which is far from trivial due to the $\epsilon$-dependence of the domain.
Establishing this uniform continuity is a central and delicate part of our analysis and plays a critical role in both stages.

In addition to estimating the exit time, we also provide an estimate on the exit place:
for each $j=1,\cdots,\abs{V}$ and $k$ such that $I_k \sim O_j$,
    \begin{equation}\label{intro place est}
        \displaystyle
        \lim_{\delta \to 0}  \lim_{\epsilon \to 0} \sup_{z \in C_{\epsilon,j}(\rho_j r_j(\epsilon)+3\epsilon)} 
        \abs{ \mathbb{P}_z \left( Z^\epsilon(\sigma^{\epsilon,\delta}) \in C_{\epsilon,j}^k(\delta) \right)  - \frac{\lambda_k^{d-1}}{\sum_{l:I_l \sim O_j} \lambda_l^{d-1}}} = 0,
    \end{equation}
    where $C_{\epsilon,j}^k(\delta)  = C_{\epsilon,j}(\delta) \cap \Pi^{-1}(I_k)$.
	While \eqref{intro place est} provides only a coarse approximation and does not fully capture the exit distribution, we are not aware of a comparable result in the existing literature (see, however, \cite{NETsecond} and \cite{NETsp} for the unit disk case under different restrictions on the initial data).
	It is important to note that the estimate in \eqref{intro place est} holds only after taking the limit $\delta \to 0$, i.e., when the bottlenecks are sufficiently short.
Although it may be possible to relax this condition, the current form suffices for our purposes.

The derivation of \eqref{intro est} and \eqref{intro place est} relies on their respective uniform continuity estimates:
for each $j=1,\cdots,\abs{V}$ and $k$ such that $I_k \sim O_j$,
\begin{equation*}
	\lim_{\epsilon \to 0} \sup_{z_1,z_2 \in C_{\epsilon,j}(\rho_j r_j(\epsilon)+3\epsilon)} 
           \abs{\mathbb{E}_{z_1} \sigma^{\epsilon,\delta} - \mathbb{E}_{z_2} \sigma^{\epsilon,\delta} } \left( \frac{r_j(\epsilon)^d}{\epsilon^{d-1}} \vee 1 \right)^{-1} =0,
\end{equation*}
and
\begin{equation*}
	\lim_{\epsilon \to 0} \sup_{z_1,z_2 \in C_{\epsilon,j}(\rho_j r_j(\epsilon)+3\epsilon)} \abs{ \mathbb{P}_{z_1} \left( Z^\epsilon(\sigma^{\epsilon,\delta}) \in  C_{\epsilon,j}^k(\delta) \right) 
         - \mathbb{P}_{z_2} \left( Z^\epsilon(\sigma^{\epsilon,\delta}) \in  C_{\epsilon,j}^k(\delta) \right)} = 0.
\end{equation*}
The key idea is that, before escape, the particle has sufficient time to mix within the domain. 
As a result, the exit behavior becomes asymptotically independent of the initial point within $C_{\epsilon,j}(\rho_j r_j(\epsilon)+3\epsilon)$.

To make this intuition rigorous, we employ the framework of quasi-stationary distributions: distributions that characterize the long time behavior of a process conditioned on survival (i.e., not having exited the domain).
More precisely, we discretize the process $Z^\epsilon(t)$ into a discrete time Markov process $X_n^{\epsilon,\delta',\delta}$ with state space $C_{\epsilon,j}(\delta') \cup \{ \partial \}$, where $ \partial $ denotes an absorbing state.
Once the process enters $ \partial $, it remains there permanently. 
We then study the convergence of $X_n^{\epsilon,\delta',\delta}$ to its quasi-stationary distribution, conditioned on survival.
To control the convergence rate uniformly in $\epsilon,\delta'$, we adapt the Dobrushin's type condition introduced in \cite{QSD} to a uniform setting suitable for our multiscale context.
Justifying this condition requires a priori estimates on both the exit time and exit place. 
The exit time bound ensures that the mixing time of the process is negligible relative to the expected exit time.
The exit place bound ensures that each neck region is accessible with positive probability, which is essential for ensuring sufficient mixing across the domain.

In the large ball regime, the expected exit time in \eqref{intro est} diverges, reflecting the presence of an effective absorbing state.
We also provide a rigorous justification for key claims in the narrow escape literature: specifically, the connection formula derived in \cite{NETthird}, equation 1, for domains with a single bottleneck, and the statement in \cite{NET}, Section 8.1.1, concerning domains with multiple bottlenecks.

Moreover, we show the asymptotic exponential distribution of first exit times from large balls with multiple bottlenecks:
for each $j$ such that $\lim_{\epsilon \to 0} \frac{r_j(\epsilon)^d}{\epsilon^{d-1}} = \infty$,
\begin{equation}\label{intro exp}
	\lim_{\epsilon \to 0} \sup_{z \in C_{\epsilon,j}(\rho_j r_j(\epsilon)+3\epsilon)}\abs{ \mathbb{P}_z \left( \left( \frac{\rho_j^d r_j(\epsilon)^d V_d}{\sum_{k:I_k \sim O_j} \lambda_k^{d-1} \epsilon^{d-1}V_{d-1}} \delta \right)^{-1} \sigma^{\epsilon,\delta} \geq t \right) - e^{-t}}=0.
\end{equation}
As noted in \cite{NETsp}, if the initial distribution is quasi-stationary, the exit time is exponentially distributed, with the principal eigenvalue as the rate parameter.
Furthermore, as discussed in \cite{NETconj}, at the intuitive level, examining higher moments of the exit time $\sigma^{\epsilon,\delta}$ reveals that $\mathbb{E}_z \left( \sigma^{\epsilon,\delta} \right)^m \sim m! \left( \mathbb{E}_z \sigma^{\epsilon,\delta} \right)^m$, suggesting that $\sigma^{\epsilon,\delta}$ is approximately exponentially distributed.
Section 1.7 of \cite{NETconj} offers an intuitive justification: the escape process resembles a Poisson process, naturally leading to exponential exit times.

We rigorously justify the asymptotic exponential distribution for the narrow escape problem from balls with multiple bottlenecks and explicitly characterize its rate parameter, without assuming that the process starts from the quasi-stationary distribution.
Our argument relies on an abstract result from \cite{SMP}, which provides continuous analogue of the classical Poisson limit theorem. 
Intuitively, \eqref{intro exp} states that the exit time $\sigma^{\epsilon,\delta}$ from $z \in C_{\epsilon,j}(\rho_j r_j(\epsilon)+3\epsilon)$ behaves like an exponential random variable with rate
\begin{equation*}
\left( \frac{\rho_j^d r_j(\epsilon)^d V_d}{\sum_{k: I_k \sim O_j} \lambda_k^{d-1} \epsilon^{d-1} V_{d-1}} \delta \right)^{-1}.
\end{equation*}
As a consequence, one expects that the moment generating function satisfies 
\begin{equation*}
	\lim_{\epsilon \to 0} \sup_{z \in C_{\epsilon,j}(\rho_j r_j(\epsilon)+3\epsilon)} \mathbb{E}_z e^{-\lambda \sigma^{\epsilon,\delta}} = 0, \text{ for all } \lambda > 0,
\end{equation*}
a property that plays a crucial role in our convergence analysis.

Lastly, we note that diffusion processes in various types of narrow domains have been studied extensively; see, for example, \cite{DPT}, \cite{FH13}, \cite{FW12}, \cite{Spi09} as well as the interesting homogenization result in \cite{C+20}. 
In addition, the spectral analysis of differential operators on graphs has been extensively studied; see, for example, \cite{KZ01}, \cite{spectral}, \cite{Convbook}, and \cite{QGbook}.
In particular, \cite{spectral} (see also \cite{Convbook}) analyzes differential operators on graphs subject to gluing conditions of the form (\cite{spectral}, equation 7.4)
\begin{equation}\label{gl spec}
	\sum_{k:I_k \sim O_j}   p_{j,k} \frac{d f}{d x_k}(O_j) 
    = -\alpha_j f(O_j).
\end{equation}
They prove convergence of the spectrum of the Laplacian, viewed as an operator on a suitably constructed Hilbert space over the domain
\begin{equation*}
	\mathcal{G}_\epsilon \coloneqq \left( \bigsqcup_{k=1}^{\abs{E}} I_k + B_k(0,\lambda_k \epsilon)  \right)
\bigsqcup \left( \bigsqcup_{j=1}^{\abs{V}}  B(O_j,\rho_j r_j(\epsilon)) \right),
\end{equation*}
where $\bigsqcup$ denotes disjoint union, to that of the corresponding limiting operator on the graph satisfying \eqref{gl spec}. 
To achieve this, they introduce extended Hilbert spaces that incorporate both edge and vertex neighborhoods, which are initially treated as separate geometric components and then glued together.
We briefly comment on the relation between the gluing condition obtained here and those in \cite{spectral}.
At an intuitive level, the spectral problem studied in \cite{spectral} is related to
\begin{equation}\label{gr spec}
    -\frac{d^2 f}{dx^2}(x) = \lambda f(x) \quad x \in \Gamma \setminus V.
\end{equation}
Formally, if the second derivative extends continuously to the vertices, then combining \eqref{gl spec} with \eqref{gr spec} yields
\begin{equation*}
    \sum_{k:I_k \sim O_j}   p_{j,k} \frac{d f}{d x_k}(O_j) 
    = \lambda^{-1} \alpha_j \bar{L} f(O_j),
\end{equation*}
which has the same general form as our gluing condition \eqref{intro glu}. 
We do not make this heuristic correspondence precise, since the two
conditions arise from fundamentally different underlying function spaces.
The analysis in \cite{spectral} is formulated in a Hilbert space containing separate edge and vertex components. 
Consequently, the compatibility between the edge and vertex components is formulated according to the underlying Hilbert-space structure. 
In particular, equation (7.2) in \cite{spectral} relates the vertex component to the endpoint values of the edge components by a volume dependent normalization factor. 
In the present work, the limiting objects are continuous functions on the graph itself, so that $f(O_j)$ is directly the common value of the function at the endpoints of all edges incident to $O_j$. 
Consequently, the vertex value appearing in the gluing condition has a different meaning in the two formulations: in \cite{spectral}, it is associated with the separate vertex component of the limiting Hilbert space, whereas here it is the common continuous value of the edge components at $O_j$.

The paper is organized as follows.
In Section \ref{sec prelim}, we construct the narrow tube $G_\epsilon$ from the graph $\Gamma$ and provide a brief overview of previous results on diffusion processes on graphs and narrow tubes.
Section \ref{sec exit} develops estimates for exit times and exit place. 
Assuming a uniform continuity estimate, we derive the exit behavior of the process.
Specifically, in Section \ref{sec NET bottle neck} we deal with the large ball regime, along with a discussion of the narrow escape problem.
In Section \ref{sec apriori}, we obtain a priori bounds for exit time and exit place.
Section \ref{sec qsd} shows that the discretized  Markov chain $X_n^{\epsilon,\delta',\delta}$ satisfies a uniform Dobrushin condition, ensuring convergence to its quasi-stationary distribution conditioned on survival.
Section \ref{sec a prior est} establishes the key uniform continuity estimates.
Section \ref{sec exit exp} is devoted to the large ball regime and proves the asymptotic exponential distribution for the narrow escape time.
The convergence of the diffusion process is proved in Section \ref{sec conv diff}.
Section \ref{sec rem} concludes with remarks and possible extensions.

The appendices contain supporting materials:
Appendix \ref{sec NET} reviews background on the (narrow) escape problem.
Appendix \ref{App reg per} justifies technical estimates used in Section \ref{sec NET connect}.
Appendix \ref{app conv qsd} proves that the Dobrushin's condition implies exponential convergence to the quasi-stationary distribution.
Appendix \ref{sec pf exp} establishes an exponential limiting result for the exit times.
Finally, Appendix \ref{app list dom} summarizes the key domains frequently used in Sections \ref{sec exit}--\ref{sec exit exp}.

\section{Notation and preliminaries}\label{sec prelim}
In this section, we present a brief overview of the notation and recall important findings from \cite{SDE1} and \cite{SDE2}.

\subsection{Asymptotic notation}
Here, we define some asymptotic notation, which we will use very often throughout the paper:
\begin{enumerate}
	\item $f(\epsilon) \ll g(\epsilon)$ or $f(\epsilon)=o(g(\epsilon))$ if $\lim_{\epsilon \to 0} \frac{f(\epsilon)}{g(\epsilon)} = 0 $.
	\item $f(\epsilon) \sim g(\epsilon)$ if $\lim_{\epsilon \to 0} \frac{f(\epsilon)}{g(\epsilon)} =1$.
	\item $f(\epsilon)  \asymp g(\epsilon)$ if there exist $c,c'>0$ and $\epsilon_0>0$ such that $c g(\epsilon) \leq f(\epsilon) \leq c' g(\epsilon)$, for all $\epsilon \in (0,\epsilon_0)$.
	\item $f(\epsilon)=O(g(\epsilon))$ or $f(\epsilon) \lesssim g(\epsilon)$ if there exist $c>0$ and $\epsilon_0>0$ such that $\abs{f(\epsilon)} \leq c g(\epsilon)$, for all $\epsilon \in (0,\epsilon_0)$.
\end{enumerate}
In situations involving multiple parameters (e.g., $\epsilon$, $\delta'$, $\delta$), we use subscripts to indicate the relevant asymptotic dependencies.
If we only consider the limit $\epsilon \to 0$, we omit the subscripts.
For instance, we write $f(\epsilon,\delta) \ll_{\epsilon,\delta} g(\epsilon,\delta)$ or $f(\epsilon,\delta) = o_{\epsilon,\delta} (g(\epsilon,\delta))$ to mean $\lim_{\delta \to 0} \lim_{\epsilon \to 0} \frac{f(\epsilon,\delta)}{g(\epsilon,\delta)}=0$.
In particular:
\begin{itemize}
\item In Section \ref{sec proof lem place} and \ref{sec conv diff}, we adopt the convention of taking the limit $\epsilon \to 0$ first, followed by $\delta \to 0$.

\item In Section \ref{sec exit exp}, we take the limit $\epsilon \to 0$ first, followed by $\delta' \to 0$.
\end{itemize}

\subsection{The graph \texorpdfstring{$\Gamma$}{Lg}}\label{sec graph}
Throughout this paper, we consider only simple, finite, connected graph $\Gamma = (V,E) \subset \mathbb{R}^d$ with $d=2$ or $3$.
The graph consists of vertices $O_1, \cdots, O_{\abs{V}}$ and edges $I_1,\cdots, I_{\abs{E}}$, where $V$ and $E$ denote the sets of vertices and edges, respectively, and $\abs{V}$ and $\abs{E}$ their cardinalities.
Each edge is a straight line segment connecting a pair of distinct vertices. 
See Figure \ref{fig:graph}.
For each $k=1,\cdots,\abs{E}$, we write $I_k \sim O_j$ to indicate that $O_j$ is an endpoint of edge $I_k$.
\begin{figure}[ht]
    \centering
    \subfigure[An example of a simple, finite, connected graph $\Gamma \subset \mathbb{R}^2$.]{\includegraphics[width=0.65\textwidth]{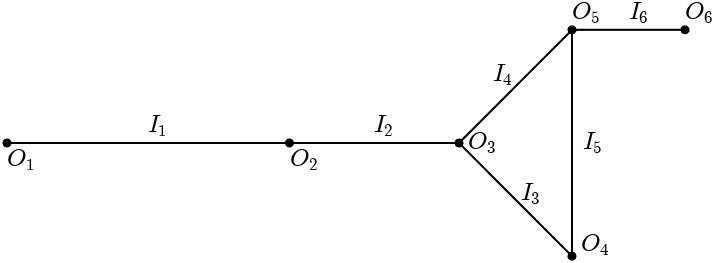}  
    \label{fig:graph}} 
    \subfigure[The corresponding narrow tube domain $G_\epsilon \subset \mathbb{R}^2$.]{\includegraphics[width=0.65\textwidth]{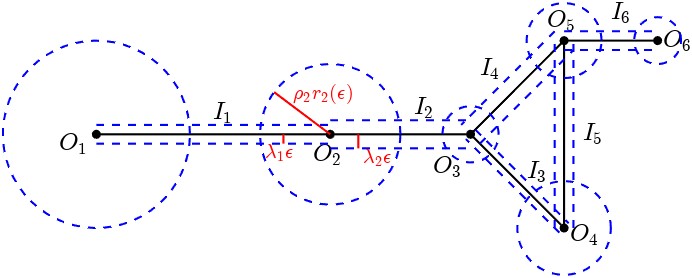}   
    \label{fig:tube}}     
    \caption{Illustration of the graph $\Gamma$ and its associated narrow tube domain $G_\epsilon$.}
\end{figure}

We define a metric $d_\Gamma$ on the graph $\Gamma$ as follows:
if $x,x'$ lie on the same edge, then $d_\Gamma(x,x') \coloneqq d(x,x')$, where $d(\cdot,\cdot) $ denotes the Euclidean distance.
If $x,x'$ lie on different edges, then 
\begin{equation*}
	d_\Gamma(x,x') \coloneqq 
		\min \{ d(x, O_{j_1})+d(O_{j_1}, O_{j_2})+\cdots+d(O_{j_n}, x') \}, 
\end{equation*}
where the minimum is taken over all possible paths from $x$ to $x'$ through a sequence of vertices $O_{j_1}, \cdots, O_{j_n}$ connecting them.
For each $j =1,\cdots,\abs{V}$ and each $k$ such that $I_k \sim O_j$, we define the direction $e_{j,k}$ as the unit vector pointing outward from the vertex $O_j$ along the edge $I_k$.

\subsection{The narrow tube \texorpdfstring{$G_\epsilon$}{Lg}}\label{sec narrow tube}
We consider a domain $G_\epsilon$ consisting of narrow tubes $\Gamma^\epsilon_k$ surrounding the edges $I_k \subset \Gamma$, along with small neighborhoods $\mathcal{E}_j^\epsilon$ around the vertices $O_j \subset \Gamma$. See Figure \ref{fig:tube}. 
More precisely, for each $j =1,\cdots,\abs{V}$, let $\mathcal{E}_j^\epsilon \coloneqq B(O_j, \rho_j r_j(\epsilon))$ denote the Euclidean ball in $\mathbb{R}^d$ centered at $O_j$ with radius $\rho_j r_j(\epsilon)$, where $\rho_j >0$ and $r_j(\epsilon) > 0$ for all $\epsilon >0$, with $ \lim_{\epsilon \to 0} r_j(\epsilon) = 0$.
For each $k=1,\cdots,\abs{E}$, we consider the edge $I_k$ and define its orthogonal complement $I_k^\perp$ by
\begin{equation*}
	I_k^\perp \coloneqq \{ y \in \mathbb{R}^d : \langle y, e_{j,k} \rangle_{\mathbb{R}^d} = 0 \},
\end{equation*} 
where $e_{j,k}$ is the unit vector pointing outward from a vertex $O_j$ (with $I_k \sim O_j$) along the edge $I_k$.
This definition is well-posed, as the choice of $j$ does not affect $I_k^\perp$; the vectors $e_{j,k}$ corresponding to different endpoints of $I_k$ differ only by a sign.
Additionally, we define the ball $B_k(y,r)$ as
\begin{equation*}
	B_k(y,r) \coloneqq \{ y' \in \mathbb{R}^d: y'-y \in I_k^\perp, \abs{y'-y} \leq r \}.
\end{equation*}
Now let $\Gamma_k^\epsilon \coloneqq I_k + B_k(0,\lambda_k \epsilon)$, where $\lambda_k > 0$ is a constant.
We then define 
\begin{equation*}
    G_\epsilon \coloneqq \bigcup_{j=1}^{\abs{V}} \mathcal{E}_j^\epsilon \cup \bigcup_{k=1}^{\abs{E}} \Gamma_k^\epsilon.
\end{equation*}

\begin{Definition}\label{def ball}
For every $j=1,\cdots,\abs{V}$
\begin{enumerate}
	\item A vertex $j$ is a small vertex if $\epsilon \ll r_j(\epsilon) \ll \epsilon^{(d-1)/d}$ and we say it belongs to $\mathfrak{S}$.
	\item A vertex $j$ is an intermediate vertex if $r_j(\epsilon) \sim \epsilon^{(d-1)/d}$ and we say it belongs to $\mathfrak{M}$.
	\item A vertex $j$ is a large vertex if $r_j(\epsilon) \gg \epsilon^{(d-1)/d}$ and we say it belongs to $ \mathfrak{L}$.
\end{enumerate}
Moreover, the ball corresponding to a small, intermediate or large vertex is called a small, intermediate or large ball, respectively.
\end{Definition}

\begin{remark}
	 For the definition of small vertices/balls, one could also include the case $r_j(\epsilon) \lesssim \epsilon$. This situation has been addressed in \cite{SDE1}; where the case $r_j(\epsilon) \asymp \epsilon$ is treated in detail. For $ \rho_j r_j(\epsilon) \leq \max_{k:I_k \sim O_j} \lambda_k \epsilon$, the same arguments applies, although the neighborhood near each vertex no longer has a ball-like structure. 
	To avoid delving into these technicalities, we assume throughout that $\epsilon \ll r_j(\epsilon)$. However, it is not hard to see how to incorporate these cases.

	Regarding the definition of intermediate vertices/balls, assuming $r_j(\epsilon) \sim \epsilon^{(d-1)/d}$ doesn't imply any loss of generality. 
    If the limit differs by a constant factor, one can simply adjust $\rho_j$ accordingly.
	Indeed, this is precisely the motivation for introducing the parameter $\rho_j$.
	
\end{remark}

We first define the projection $\Pi: G_\epsilon \to \Gamma$, which maps $z \in G_\epsilon$ to the point in $\Gamma$ closest to $z$. 
While this point may not be unique, we can choose it in a unique way. For example, we already have labels on the edges $I_k$. We then assign the projection to the point with the least edge index. 
To be more precise, for every $x \in \Gamma \setminus \bigcup_{j=1}^{\abs{V}} O_j$, let $i(x)$ be the index map such that $i(x)=k$ if and only if $x \in I_k$. 
For each $j=1,\cdots,\abs{V}$, set $i(O_j) = k$ if and only if $O_j \in I_k$ and $k < l$ for all $l \neq k$ such that $O_j \in I_l$.
Then, for any $z \in G_\epsilon$, we define $\Pi(z) = x $ if and only if $d(z,x) < d(z,x')$ for all $x' \in \Gamma$, or if $d(z,x) = d(z,x')$ for some $x' \in \Gamma$, but $i(x)<i(x')$ for all such $x' \in \Gamma$.
Note that the projection doesn't depend on $\epsilon$; in fact, one can define the projection in the whole $\mathbb{R}^d$.

We denote by $\Pi^\epsilon: G_\epsilon \to \Gamma$ any continuous function satisfying the following properties
\begin{equation}\label{proj 1}
	\Pi^\epsilon(z)
		= \Pi(z), \ \text{ if } \ d_\Gamma(\Pi(z), O_j) \geq \rho_j r_j(\epsilon) + 2\epsilon \quad \text{ for all } j = 1,\cdots,\abs{V},
\end{equation}
and
\begin{equation}\label{proj 2}
	d_\Gamma (\Pi^\epsilon(z), O_j) \leq \rho_j r_j(\epsilon) + 2\epsilon, \ \text{ if } \ d_\Gamma(\Pi(z),O_j) \leq \rho_j r_j(\epsilon) + 2\epsilon \quad \text{ for some } j \in \{1,\cdots,\abs{V}\} .
\end{equation}
For concreteness, one may consider the following explicit choice for $\Pi^\epsilon$,
	\begin{equation*}
	\begin{aligned}
\Pi^\epsilon(z) \coloneqq 
	\begin{dcases}
	\displaystyle
		\Pi(z), \text{ if } \ d_\Gamma(\Pi(z), O_j) \geq \rho_j r_j(\epsilon) + 2\epsilon \text{ for all } j=1,\cdots,\abs{V} , \\
		O_j + \left( \frac{\rho_j r_j(\epsilon)+2\epsilon}{4 \epsilon} d_{\Gamma}(\Pi(z),O_j)-\frac{\rho^2_j r_j(\epsilon)^2-4 \epsilon^2}{4\epsilon} \right) e_{j,k}, \\      
		\ \ \ \ \ \ \ \text{ if } \ d_\Gamma(\Pi(z),O_j) \in [ \rho_j r_j(\epsilon) - 2\epsilon, \rho_j r_j(\epsilon) + 2\epsilon] \text{ for some } j \in \{1,\cdots,\abs{V}\} \\
        \ \ \ \ \ \ \ \qquad \text{ and } \ \Pi(z) \in I_k, \\
		O_j,  \ \   \text{ if } \ d_\Gamma(\Pi(z),O_j) \leq \rho_j r_j(\epsilon) - 2\epsilon \text{ for some } j \in \{1,\cdots,\abs{V}\} .
	\end{dcases}
\end{aligned}
\end{equation*}

It is important to note that $G_\epsilon$ is not initially a smooth domain due to the presence of corners where the tubular regions meet the vertex neighborhoods. 
Since we will use the reflected Brownian motion in $G_\epsilon$, we smooth the boundary in a neighborhood of each such corner. 
The smoothing is performed at the $\epsilon$-scale and is chosen so that
\begin{equation*}
    B(O_j,\rho_jr_j(\epsilon)),\ \Gamma_k^\epsilon\subset G_\epsilon, \quad \text{ for all } j = 1,\cdots,\abs{V}, \ \ k  = 1,\cdots,\abs{E}.
\end{equation*}
In addition, the local graph representations of the resulting boundary in $\epsilon$-scale neighborhoods have $W^{1,\infty}$ norms bounded uniformly in $\epsilon$ and satisfy a certain monotonicity condition.
See Figure \ref{fig:smooth dom} below.

More precisely, for each $j=1,\cdots,\abs{V}$ and each $k$ such that $I_k \sim O_j$, let 
\begin{equation*}
z^*_{j,k} \coloneqq \partial B(O_j, \rho_j r_j(\epsilon)) \cap I_k
\end{equation*}
and consider the change of variables 
\begin{equation*}
	z \mapsto \Psi^{j,k}_\epsilon(z):=\frac{z-z^*_{j,k}}{\epsilon}.
\end{equation*}
After a rotation, we assume that the positive $x$-axis coincides with $e_{j,k}$ and write $z = (x,y)$, where $y \in \mathbb{R}^{d-1}$ denotes coordinates orthogonal to $e_{j,k}$.
Under $\Psi^{j,k}_\epsilon$, the vertex neighborhood near the corner $B(O_j,\rho_jr_j(\epsilon)) \cap B(z_{j,k}^*, (5\lambda_k+3)\epsilon )$ is mapped to
\begin{equation*}
\begin{aligned}
    & \Psi^{j,k}_\epsilon(B(O_j,\rho_jr_j(\epsilon)) \cap B(z_{j,k}^*, (5\lambda_k+3)\epsilon )) \\
    & = \{ z = (x,y): (x+\rho_j r_j(\epsilon)/\epsilon)^2+y^2 \leq \rho_j^2 r_j(\epsilon)^2 /\epsilon^2\} \cap \{ z=(x,y):x^2+y^2 \leq (5\lambda_k+3)^2 \} \\
    & = \{ x < \sigma_\epsilon(\abs{y}) \} \cap \{ x^2+y^2 \leq (5\lambda_k+3)^2 \},
\end{aligned}
\end{equation*}
where
\begin{equation*}
    \sigma_\epsilon(\abs{y}) \coloneqq \frac{ \sqrt{\rho_j^2 r_j(\epsilon)^2-\epsilon^2 \abs{y}^2} - \rho_j r_j(\epsilon)}{ \epsilon}
\end{equation*}
Since $r_j(\epsilon) \gg \epsilon$, $\Psi^{j,k}_\epsilon(\partial B(O_j,\rho_jr_j(\epsilon)) \cap B(z_{j,k}^*, (5\lambda_k+3)\epsilon ) )$ converges in $C^\infty$ to the flat surface $\{x=0\} \cap B(0, 5\lambda_k+3 )$, as $\epsilon \to 0$.
After the $\epsilon$-rescaling, the only non smoothness occurs at the intersection of the flat surface $\{x=0\} \cap B(0, 5\lambda_k+3 )$ with the tube boundary $\{ (x,y): x \geq 0, \abs{y} = \lambda_k \}$, namely,
\begin{equation*}
    \{ (x,y) : x = 0, \ \abs{y} = \lambda_k \}.
\end{equation*}
We then mollify the corner within 
\begin{equation*}
    \{ (x,y) : 0 \leq x \leq \frac{1}{4}, \ \abs{y} = \lambda_k \}
    \cup \{ (x,y) : x=0, \ \lambda_k \leq  \abs{y} \leq \frac{3}{2} \lambda_k \}
\end{equation*}
by a parametrized hypersurface $\gamma^{j,k}$, independent of $\epsilon$, invariant under rotation in the $y$-variables, satisfying
\begin{equation}\label{d eq 2}
    \gamma^{j,k} \subset \{ x\geq 0 \} \cap \{ \abs{y} \geq \lambda_k \},
\end{equation}
and in the parametrization, all derivatives at the endpoint corresponding to $x = \frac{1}{4}$ agree with those of the tube wall, and likewise at $\abs{y}=\frac{3}{2}\lambda_k$ with those of the surface $\{x = 0\}$.
We further require that every line parallel to the $x$-axis meet $\gamma^{j,k} \cap \{ \abs{y} > \lambda_k \}$ in at most one point.

We then connect $\gamma^{j,k}$ to the original boundary $ \{ x = \sigma_\epsilon(\abs{y})\}$, which converges locally uniformly for $\abs{y} \geq \frac{3}{2} \lambda_k$.
Fix a smooth cutoff function $\chi(\abs{y})$ such that $\chi = 0$ near $\frac{3}{2} \lambda_k$ and $\chi = 1$ near $\frac{5}{2} \lambda_k$ with all derivatives vanishing at both ends.
The full mollified boundary is
\begin{equation*}
\begin{aligned}
    &\tilde{\Sigma}_\epsilon^{j,k} \coloneqq \left( \{ x \geq \frac{1}{4}, \abs{y} =\lambda_k \} \cup \gamma^{j,k} \cup \{ x = \chi(\abs{y}) \sigma_\epsilon(\abs{y}), \frac{3}{2}\lambda_k \leq \abs{y} \leq \frac{5}{2}\lambda_k\} \cup \{ x = \sigma_\epsilon(\abs{y}), \abs{y} \geq \frac{5}{2}\lambda_k\} \right) \\
    & \qquad \qquad  \cap B(0,5\lambda_k+3)
\end{aligned}
\end{equation*}
We then replace the boundary $\partial G_\epsilon \cap B(z_{j,k}^*,(5\lambda_k+3) \epsilon)$ by $(\Psi_\epsilon^{j,k})^{-1} (\tilde{\Sigma}_\epsilon^{j,k})$.
Throughout the rest of the paper, we work with this modified smooth domain and continue to denote it by $G_\epsilon$.

By construction, this replacement satisfies
\begin{itemize}
    \item \textit{Containment:} By \eqref{d eq 2}, we in fact enlarge the domain and thus $B(O_j, \rho_j r_j(\epsilon)), \Gamma_k^\epsilon \subset G_\epsilon$ for every $j = 1,\cdots,\abs{V}$, $k = 1,\cdots,\abs{E} $.
    \item \textit{Locality:} The modification is confined to $B(z_{j,k}^*,(5\lambda_k+3) \epsilon)$. In fact,
    \begin{equation*}
    \{ z \in \partial G_\epsilon: d_\Gamma(\Pi(z),O_j) \geq \rho_j r_j(\epsilon)+\frac{1}{4} \epsilon \} \text{ and } \{ z \in \partial G_\epsilon \cap B(O_j,\rho_j r_j(\epsilon)): d(z,z_{j,k}^*) \geq 3 \lambda_k \epsilon\}
\end{equation*}
remains unchanged.
\item \textit{Regularity:} The transformed boundary $\tilde{\Sigma}_\epsilon^{j,k}$ is a $C^\infty$ hypersurface whose local graph representations have uniformly bounded $C^\infty$ norms.
In particular, every point $z \in \tilde{\Sigma}_\epsilon^{j,k}$ admits a local graph with Lipschitz constant at most $c$, uniformly in $\epsilon$.
\item  \textit{Monotonicity:} The transformed boundary $\tilde{\Sigma}_\epsilon^{j,k} \subset \{ \abs{y} \geq \lambda_k \}$ and every line parallel to the $x$-axis meets $\tilde{\Sigma}_\epsilon^{j,k} \cap \{\abs{y} > \lambda_k\}$ in at most one point.
\end{itemize}

\begin{remark}

    The construction above gives a uniformly smooth family after the $\epsilon$-rescaling. 
    For the PDE estimates in the subsequent analysis, however, we only use the uniform local Lipschitz control together with the containment, locality, and monotonicity properties described above.
    These geometric properties are natural in the present setting: the original domains already satisfy the required containment, locality, uniform Lipschitz, and monotonicity properties, including at the corners. 
    The smoothing is introduced only to remove the non-smoothness at the corners while preserving these properties.
    
    We introduce the smoothing primarily to avoid technical issues in the analysis of the reflected Brownian motion; see, however, \cite{LS84}, \cite{BH} and the references therein for constructions of reflected Brownian motion in certain non-smooth domains.

    Finally, the monotonicity condition is imposed for technical reasons, as it allows us to establish a Poincar\'e inequality in the rescaled domain. 
    This condition may be weakened; see Remark \ref{rem poin} for further discussion.

\end{remark}

\begin{figure}[ht]
  \centering
  \includegraphics[width=0.75\linewidth]{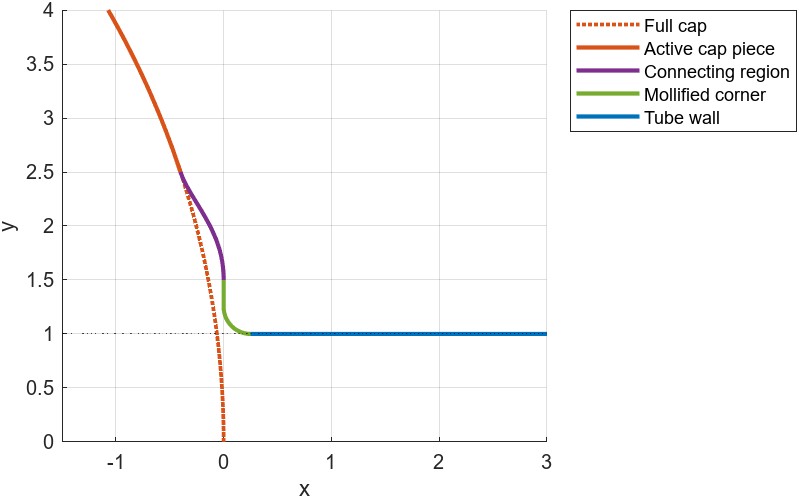}
  \caption{The mollified boundary $\tilde{\Sigma}_\epsilon^{j,k}$ with $\rho_j r_j(\epsilon)=0.8$, $\lambda_k \epsilon = \epsilon = 0.1$ in dimension $d=2$.
  The portion lying in $\{ y \leq 0 \}$ is symmetric with respect to the $x$-axis.}\label{fig:smooth dom}
\end{figure}

\subsection{Diffusion processes on graphs}\label{sec diff graph prelim}
In \cite{SDE2}, a continuous Markov process $\bar{Z}(t)$ on the graph $\Gamma$ is introduced. Its generator $\bar{L}$ is given by suitable differential operators $\mathcal{L}_k$ within each edge of the graph and is subject to certain gluing conditions at the vertices $O_j$. 
To be more precise, for any point $x$ in the interior of the edge $I_k$, 
\begin{equation*}
    \bar{L} f(x) = \mathcal{L}_k f(x),
\end{equation*}
where $\mathcal{L}_k$ is a second order differential operator on $I_k$.
The domain of $D(\bar{L})$ consists of continuous functions on the graph that are twice continuously differentiable in the interior of each edge. 
Furthermore, for each vertex $O_j$, the limit 
\begin{equation*}
	\lim_{x \in I_k, x \to O_j} \mathcal{L}_k f(x)
\end{equation*}
exists and is independent of the choice of edge $I_k \sim O_j$.
	We therefore define
\begin{equation*}
	\bar{L}f(O_j) \coloneqq \lim_{x \in I_k, x \to O_j} \mathcal{L}_k f(x).
\end{equation*}
In addition, the following one sided limit exist
\begin{equation*}
	\frac{d f}{d x_k}(O_j) \coloneqq \lim_{x \in I_k, x \to O_j} \lim_{h \to 0^+} \frac{f(x+h e_{j,k} ) - f(x)}{h},
\end{equation*}
where $e_{j,k}$ is the unit vector pointing outward from $O_j$ along $I_k$, as defined in Section \ref{sec graph}.
Finally, at each vertex $O_j$ the following gluing condition is imposed
\begin{equation}\label{abs glu}
    \sum_{k:I_k \sim O_j}   p_{j,k} \frac{d f}{d x_k}(O_j) 
    = \alpha_j \bar{L}f(O_j),
\end{equation}
for some nonnegative constants $p_{j,k}, \alpha_j$ with $\alpha_j+\sum_{k:I_k \sim O_j}p_{j,k}>0$.

The following theorem is proved in \cite{SDE2}, Theorem 3.1.
\begin{Theorem}\label{op on graph}
    The operator $\bar{L}$ is the infinitesimal generator of a strongly continuous semigroup of linear operators on $C(\Gamma)$, corresponding to a conservative Markov process $(\bar{Z}(t),\bar{\mathbb{P}})$ on $\Gamma$ with continuous paths. 
\end{Theorem}
We denote the expectation associate with the process $\bar{Z}(t)$ starting from $x \in \Gamma$ by $\bar{\mathbb{E}}_{x}$.

\subsection{Diffusion processes on narrow tubes}\label{sec diff tube prelim}
We consider the diffusion process confined to live on the domain $G_\epsilon$,

\begin{equation*}
    dZ^\epsilon(t) = \sqrt{2}dB(t) + \nu_\epsilon (Z^\epsilon(t)) d\phi^\epsilon(t), \ Z^\epsilon(0) = z,
\end{equation*}
where $z \in G_\epsilon$.
Here $B(t)$ is a $d$-dimensional Brownian motion defined on a stochastic basis $(\Omega, \mathcal{F}, \{\mathcal{F}_t\}_{t \geq 0}, \mathbb{P})$, $\nu_\epsilon(z)$ is the unit inward normal vector at the point $z \in \partial G_\epsilon$, and $\phi^\epsilon(t)$ is the local time of the process $Z^\epsilon(t)$ on the boundary $\partial G_\epsilon$. 
That is, $\phi^\epsilon(t)$ is an adapted, continuous with probability $1$, non-decreasing process that increases only when $Z^\epsilon(t) \in \partial G_\epsilon$.
Formally, one can write
\begin{equation*}
    \phi^\epsilon(t) = \int_0^t  \mathbbm{1}_{\{ Z^\epsilon(s) \in \partial G_\epsilon \} } d\phi^\epsilon(s).
\end{equation*}
It can be shown that the generator of $Z^\epsilon(t)$ is the Laplacian $\Delta$, with Neumann boundary conditions on $\partial G_\epsilon$.
We refer to this process as the scaled reflected Brownian motion on $G_\epsilon$.

Our main goal is to prove the following theorem
\begin{Theorem}\label{prevconv}
	The process $\Pi^\epsilon \circ Z^\epsilon$ converges weakly to $\bar{Z}$ in $C([0,\infty);\Gamma)$, uniformly with respect to the initial condition, where the generator $\bar{L}$ of $\bar{Z}$ is given as follows:  $\mathcal{L}_k f = f''$, for all $k=1,\cdots,\abs{E}$, and the gluing condition at $O_j$ differs depending on the type of $O_j$ (small, intermediate, or large, see Definition \ref{def ball}):
    \begin{subnumcases}{}
        \displaystyle \sum_{k:I_k \sim O_j}  \frac{\lambda_k^{d-1}}{\sum_{l :I_l \sim O_j} \lambda_l^{d-1} } \frac{d f}{d x_k}(O_j) =0, 
        & if $j \in \mathfrak{S}$, \label{smG} \\
        \displaystyle \sum_{k:I_k \sim O_j}    \frac{\lambda_k^{d-1} }{\sum_{l :I_l \sim O_j} \lambda_l^{d-1} } \frac{d f}{d x_k}(O_j) 
    = \frac{\rho_j^{d} V_d}{\sum_{l :I_l \sim O_j} \lambda_l^{d-1} V_{d-1}} \bar{L}f(O_j), 
     & if $j \in \mathfrak{M}$, \label{meG} \\
        \displaystyle \bar{L}f(O_j) =0, 
       &  if $j \in \mathfrak{L}$. \label{lG}
    \end{subnumcases}
    
\end{Theorem}

When $r_j(\epsilon) \asymp \epsilon$ for all $j=1,\cdots,\abs{V}$, this result is proved in \cite{SDE2}, Theorem 6.1. 
The general case is conjectured in \cite{SDE1}, Chapter 7.

Before proceeding to the details, we briefly discuss the interpretation of the associated gluing conditions.
In the first case, \eqref{smG} corresponds to a classical boundary condition, if, for a given vertex $O_j$, there is only one edge $I_k$ such that $I_k \sim O_j$, then \eqref{smG} reduces to the usual Neumann boundary condition, meaning that the Brownian particle is instantaneously reflected upon reaching $O_j$.
If multiple edges $I_k$ are incident to $O_j$, then the resulting motion resembles a skew Brownian motion at the vertex.
More precisely, when the Brownian particle hits $O_j$, it leaves along edge $I_k$ with probability
\begin{equation}\label{p}
	p_{j,k} = \frac{\lambda_k^{d-1}}{\sum_{l :I_l \sim O_j} \lambda_l^{d-1} }.
\end{equation}

The second and third types of gluing conditions correspond to general boundary conditions for diffusion introduced by Feller \cite{Feller54} (see also \cite{Ito}).
Condition \eqref{meG} describes a situation where, upon reaching the vertex $O_j$, the process exhibits a delay at the vertex before leaving along edge $I_k$ with probability $p_{j,k}$.
The parameter 
\begin{equation*}
	\alpha_j = \frac{\rho_j^{d} V_d}{\sum_{l :I_l \sim O_j} \lambda_l^{d-1} V_{d-1}}
\end{equation*}
quantifies the strength of the delay effect at the vertex $O_j$.
In fact, one can show that (see, e.g., \cite{sticky},  \cite{MPD})
\begin{equation*}
	\lim_{\delta \to 0} \frac{1}{\delta} \bar{\mathbb{E}}_{O_j} \bar{\sigma}^{\delta} = \alpha_j,
\end{equation*}
where $\bar{\sigma}^{\delta}$ is the first hitting time of the set $\{x \in \Gamma: d_\Gamma(x,O_j) = \delta \}$.
In contrast, condition \eqref{lG} represents an absorbing state: once the Brownian particle reaches $O_j$, it remains there forever.

Formally, conditions \eqref{smG} and \eqref{lG} can be viewed as limits of \eqref{meG} by letting $\alpha_j \to 0$ and $\alpha_j \to \infty$, respectively.

\section{Exit time and exit place estimates}\label{sec exit}

In this section, our goal is to prove the exit time and exit place estimates that will be used in the subsequent analysis.

First, we define the domains
\begin{equation*}
\begin{aligned}
    C_{\epsilon,j}(\delta) \coloneqq  \{ z \in G_\epsilon : d_{\Gamma}(\Pi(z),O_j) = \delta \}, \quad
    C_{\epsilon,j}^k(\delta) \coloneqq  C_{\epsilon,j}(\delta) \cap \Pi^{-1}(I_k), 
\end{aligned}
\end{equation*}
and
\begin{equation*}
\begin{aligned}
    G_{\epsilon,j}(\delta) \coloneqq  \{ z \in G_\epsilon : d_{\Gamma}(\Pi(z),O_j) \geq \delta \}, \quad 
    B_{\epsilon,j}(\delta) \coloneqq  \{ z \in G_\epsilon : d_{\Gamma}(\Pi(z),O_j) \leq \delta \}.
\end{aligned}
\end{equation*}
We also define the stopping time
\begin{equation*}
\begin{aligned}
\sigma^{\epsilon,\delta} &\coloneqq \inf  \{t \geq  0 : Z^\epsilon(t) \in \bigcup_{j=1}^{\abs{V}} C_{\epsilon,j}(\delta) \}.
\end{aligned}
\end{equation*}

\begin{figure}[ht]
  \centering
  \includegraphics[width=0.65\linewidth]{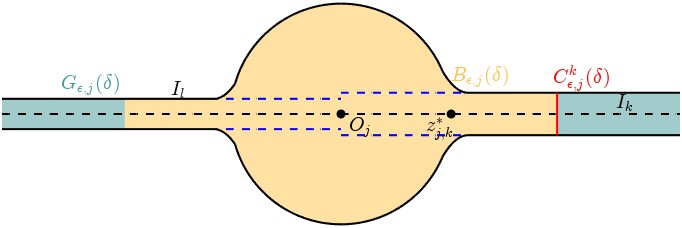}
  \caption{Illustration of the sets $C_{\epsilon,j}^k(\delta)$, $G_{\epsilon,j}(\delta)$, and $B_{\epsilon,j}(\delta)$.}\label{fig:def}
\end{figure}

The following two propositions are the main ingredients to prove the convergence of the diffusion process $\Pi^\epsilon \circ Z^\epsilon$. 

\begin{Proposition}\label{exit place est}
For every $j=1,\cdots,\abs{V}$ and $k$ such that $I_k \sim O_j$,we have
    \begin{equation*}
        \displaystyle
        \lim_{\delta \to 0}  \lim_{\epsilon \to 0} \sup_{z \in C_{\epsilon,j}(\rho_j r_j(\epsilon)+3\epsilon)} 
        \abs{ \mathbb{P}_z \left( Z^\epsilon(\sigma^{\epsilon,\delta}) \in C_{\epsilon,j}^k(\delta) \right)  - p_{j,k}} = 0,
    \end{equation*}
    where $p_{j,k}$ is the constant introduced in \eqref{p}.
\end{Proposition}

\begin{Proposition}\label{small exit est 2}
    For every $j=1,\cdots,\abs{V}$, 
    \begin{equation*}
        \displaystyle
           \lim_{\epsilon \to 0} \sup_{z \in C_{\epsilon,j}(\rho_j r_j(\epsilon)+3\epsilon)} \abs{\mathbb{E}_z \sigma^{\epsilon,\delta} \left( \alpha_j(\epsilon) \delta + \frac{\delta^2}{2} \right)^{-1} -1 }=0,
    \end{equation*} 
    where 
    \begin{equation*}
\alpha_j(\epsilon) = 
 \frac{\rho_j^d r_j(\epsilon)^d V_{d}}{\sum_{k:I_k \sim O_j} \lambda_k^{d-1} \epsilon^{d-1} V_{d-1}}
.
\end{equation*}
\end{Proposition}

\begin{remark}
	In the regime where $r_j(\epsilon)^d \gg \epsilon^{d-1}$, Proposition \ref{small exit est 2} provides an asymptotic estimate for the narrow escape time, consistent with the prediction in the literature (see, e.g., \cite{NET}), while Proposition \ref{exit place est} offers a corresponding estimate for the exit place.
	We will return to a more detailed discussion of the exit time in Section \ref{sec NET bottle neck}.
\end{remark}

Propositions \ref{exit place est} and \ref{small exit est 2} are significantly more challenging to prove.
To proceed, we require the following two lemmas, which, intuitively, assert that the asymptotic behavior of all points in $C_{\epsilon,j}(\rho_j r_j(\epsilon)+3\epsilon)$ is essentially identical, for each $j=1,\cdots,\abs{V}$.
This uniformity is crucial because, in the limiting process, these points collapse to a single vertex.
Thus, to ensure the limit is well-defined, the behavior of the process at these points must converge uniformly.

\begin{Lemma}\label{exit place conti}
For every $j=1,\cdots,\abs{V}$ and $k$ such that $I_k \sim O_j$, we have
    \begin{equation*}
        \displaystyle
        \lim_{\epsilon \to 0} \sup_{z_1,z_2 \in C_{\epsilon,j}(\rho_j r_j(\epsilon)+3\epsilon)} \abs{ \mathbb{P}_{z_1} \left( Z^\epsilon(\sigma^{\epsilon,\delta}) \in  C_{\epsilon,j}^k(\delta) \right) 
         - \mathbb{P}_{z_2} \left( Z^\epsilon(\sigma^{\epsilon,\delta}) \in  C_{\epsilon,j}^k(\delta) \right)} = 0.
    \end{equation*}
\end{Lemma}

\begin{Lemma}\label{exit est conti}
    For every $j=1,\cdots,\abs{V}$, we have
    \begin{equation*}
        \displaystyle
           \lim_{\epsilon \to 0} \sup_{z_1,z_2 \in C_{\epsilon,j}(\rho_j r_j(\epsilon)+3\epsilon)} 
           \abs{\mathbb{E}_{z_1} \sigma^{\epsilon,\delta} - \mathbb{E}_{z_2} \sigma^{\epsilon,\delta} } \left( \frac{r_j(\epsilon)^d}{\epsilon^{d-1}} \vee 1 \right)^{-1} =0.
    \end{equation*}
\end{Lemma}

We defer the proofs of Lemmas \ref{exit place conti} and \ref{exit est conti} to Section \ref{sec a prior est}. 
The ingredients needed for these proofs are developed separately in Sections \ref{sec apriori} and \ref{sec qsd}. 
The remaining subsections are devoted to the proofs of Propositions \ref{exit place est} and \ref{small exit est 2}, assuming that Lemmas \ref{exit place conti} and \ref{exit est conti} hold.

\subsection{Proof of Proposition \ref{exit place est}}\label{sec proof lem place}
The proof closely follows the argument of Theorem 6.1 in \cite{SDE2}.
As we will need some of these arguments later, we include the proof here for clarity and completeness.

Recall that $Z^\epsilon(t)$ is the scaled reflected Brownian motion on $G_\epsilon$.
Define the sequence of stopping times
\begin{equation*}
\begin{aligned}
\sigma_{n}^{\epsilon,\delta} &\coloneqq \inf  \{t \geq  \tau_{n}^{\epsilon,\delta} : Z^\epsilon(t) \in \bigcap_{j=1}^{\abs{V}} G_{\epsilon,j}(\delta) \}, \\
\tau_{n}^{\epsilon,\delta} &\coloneqq \inf  \{t \geq  \sigma_{n-1}^{\epsilon,\delta} : Z^\epsilon(t) \in \bigcup_{j=1}^{\abs{V}} C_{\epsilon,j}(\rho_j r_j(\epsilon)+ 3\epsilon ) \},
\end{aligned}
\end{equation*}
with $\tau_{0}^{\epsilon,\delta} \coloneqq 0$.
Also, define
\begin{equation*}
\begin{aligned}
\tau^{\epsilon} &\coloneqq \inf  \{t \geq  0 : Z^\epsilon(t) \in \bigcup_{j=1}^{\abs{V}} C_{\epsilon,j}(\rho_j r_j(\epsilon) + 3\epsilon ) \}.
\end{aligned}
\end{equation*}
and recall that above we defined
\begin{equation*}
\begin{aligned}
\sigma^{\epsilon,\delta} &= \inf  \{t \geq  0 : Z^\epsilon(t) \in \bigcup_{j=1}^{\abs{V}} C_{\epsilon,j}(\delta) \}.
\end{aligned}
\end{equation*}
Note that if $Z^\epsilon(t)$ starts at $z \in \bigcup_{j=1}^{\abs{V}} B_{\epsilon,j}(\delta)$ (in particular, if $z \in \bigcup_{j=1}^{\abs{V}} C_{\epsilon,j}(\rho_j r_j(\epsilon) + 3\epsilon )$), then
\begin{equation*}
	\sigma^{\epsilon,\delta} = \sigma_{0}^{\epsilon,\delta}.
\end{equation*}
On the other hand, if $Z^\epsilon(t)$ starts at $z \in \bigcup_{j=1}^{\abs{V}} C_{\epsilon,j}(\delta )$, then
\begin{equation*}
	\sigma^{\epsilon,\delta} = \sigma_{0}^{\epsilon,\delta} = 0, \ \ \ \ \ 
	\tau^{\epsilon} = \tau_{1}^{\epsilon,\delta}.
\end{equation*}

It is known that $Z^\epsilon(t)$ has the Lebesgue measure $Leb$ on $G_\epsilon$ as its invariant measure.
Thus if we define $\mu_\epsilon$ as the invariant measure of the process $Z^\epsilon(\tau_{n}^{\epsilon,\delta})$ on $\bigcup_{j=1}^{\abs{V}} C_{\epsilon,j}(\rho_j r_j(\epsilon) + 3\epsilon)$, we have the following relation (see \cite{EPD}, Theorem 2.1, for a proof; also see \cite{SDE2}), for every Borel set $A_\epsilon \subset G_\epsilon$ ,
\begin{equation}\label{exit place int}
	Leb(A_\epsilon) = \int_{\bigcup_{j=1}^{\abs{V}} C_{\epsilon,j}(\rho_j r_j(\epsilon) + 3\epsilon)} \left( \mathbb{E}_z \int_0^{\tau_{1}^{\epsilon,\delta}} \mathbbm{1}_{A_\epsilon} \left( Z^\epsilon(t) \right) dt \right) \mu_\epsilon(dz).
\end{equation}

First, for each $j=1,\cdots,\abs{V}$ and $k$ such that $I_k \sim O_j$, let $i=i(k)$ be the other endpoint of $I_k$.
We define $A_{j,\epsilon,\delta}^k$ as the region enclosed by $C_{\epsilon,j}^k(\delta)$, $C_{\epsilon,j}^k(2\delta)$ and $\partial G_\epsilon$.
Note that only $C_{\epsilon,i}(\rho_i r_i(\epsilon) + 3\epsilon)$ and $C_{\epsilon,j}(\rho_j r_j(\epsilon) + 3\epsilon)$ contribute to the integral \eqref{exit place int}.
Hence, by the strong Markov property, 
\begin{equation}\label{eq leb 1}
\begin{aligned}
	& Leb(A^k_{j,\epsilon,\delta})  \\
	&= \int_{C_{\epsilon,j}(\rho_j r_j(\epsilon)+3\epsilon )} 
	\mathbb{E}_z \left( \mathbbm{1}_{\{Z^\epsilon(\sigma^{\epsilon,\delta}) \in C_{\epsilon,j}^k(\delta) \}} 
	\mathbb{E}_{Z^\epsilon(\sigma^{\epsilon,\delta})} \left( \int_0^{\tau^{\epsilon}} \mathbbm{1}_{A^k_{j,\epsilon,\delta}} \left( Z^\epsilon(t) \right) dt  \right)  \right) \mu_\epsilon(dz) \\
	& \ \ \  + \int_{C_{\epsilon,i}(\rho_i r_i(\epsilon) + 3\epsilon)}  
	\mathbb{E}_z \left( \mathbbm{1}_{\{Z^\epsilon(\sigma^{\epsilon,\delta}) \in C_{\epsilon,i}^k(\delta) \}} 
	\mathbb{E}_{Z^\epsilon(\sigma^{\epsilon,\delta})} \left( \int_0^{\tau^{\epsilon}} \mathbbm{1}_{A^k_{j,\epsilon,\delta}} \left( Z^\epsilon(t) \right) dt \right)  \right) \mu_\epsilon(dz).
\end{aligned}
\end{equation}
By the properties of one dimensional Brownian motion, we obtain, for $z' \in C_{\epsilon,j}^k(\delta)$, the occupation time 
\begin{equation}\label{eq leb 2}
\mathbb{E}_{z'} \int_0^{\tau^{\epsilon}} \mathbbm{1}_{A^k_{j,\epsilon,\delta}} \left( Z^\epsilon(t) \right) dt 
= \delta^2+o_{\epsilon,\delta}(\delta^2).
\end{equation}
and for $z'' \in C_{\epsilon,i}^k(\delta)$, the occupation time 
\begin{equation}\label{eq leb 3}
\mathbb{E}_{z''} \int_0^{\tau^{\epsilon}} \mathbbm{1}_{A^k_{j,\epsilon,\delta}} \left( Z^\epsilon(t) \right) dt = O_{\epsilon,\delta}(\delta^3).
\end{equation}
In fact, the occupation time can be obtain by solving the ODE on $[0,L_{k}^\epsilon]$
\begin{equation*}
\begin{dcases}
	f_{j,k}''(x) = -\mathbbm{1}_{[\delta-\rho_j r_j(\epsilon)-3\epsilon,2\delta-\rho_j r_j(\epsilon)-3\epsilon] }(x), \ \ \ x \in [0,L_{k}^\epsilon] ,\\
	f_{j,k}(0) =  f_{j,k}(L_{k}^\epsilon) = 0,
\end{dcases}
\end{equation*}
where $L_{k}^\epsilon=\abs{I_{k}}-\rho_j r_j(\epsilon)-3\epsilon-\rho_i r_i(\epsilon)-3\epsilon$.

For all $j = 1,\cdots, \abs{V}$ and $k$ such that $I_k \sim O_j$, we define
\begin{equation}\label{exit place qj}
	q_{j,\epsilon,\delta}^k \coloneqq \int_{C_{\epsilon,j}(\rho_j r_j(\epsilon) + 3\epsilon )} 
	\mathbb{P}_z \left( Z^\epsilon(\sigma^{\epsilon,\delta}) \in C_{\epsilon,j}^k(\delta) \right) \mu_\epsilon(dz),
\end{equation}
then combining \eqref{eq leb 1}, \eqref{eq leb 2}, \eqref{eq leb 3} and \eqref{exit place qj}, we obtain
\begin{equation}\label{exit place eq 1}
	\delta \lambda_k^{d-1} \epsilon^{d-1} V_{d-1} = Leb(A^k_{j,\epsilon,\delta}) 
	= q_{j,\epsilon,\delta}^k (\delta^2+o_{\epsilon,\delta}(\delta^2)) +  q_{i,\epsilon,\delta}^k (O_{\epsilon,\delta}(\delta^3)).
\end{equation}
Summing over all $k$ such that $I_k \sim O_j$ and then over all $j=1,\cdots,\abs{V}$ (noting that $i$ depends on $k$), we have
\begin{equation*}
	\delta  \epsilon^{d-1} V_{d-1} \sum_{j=1}^{\abs{V}} \sum_{k:I_k \sim O_j} \lambda_k^{d-1} 
	= \sum_{j=1}^{\abs{V}} \mu_\epsilon(C_{\epsilon,j}(\rho_j r_j(\epsilon)+ 3\epsilon)) (\delta^2+o_{\epsilon,\delta}(\delta^2) + O_{\epsilon,\delta}(\delta^3)).
\end{equation*}
Therefore, for all $j=1,\cdots,\abs{V}$, and $k$ such that $I_k \sim O_j$,
\begin{equation*}
	0 \leq q_{j,\epsilon,\delta}^k \leq \mu_\epsilon(C_{\epsilon,j}(\rho_j r_j(\epsilon)+ 3\epsilon)) =O_{\epsilon,\delta}(\epsilon^{d-1} \delta^{-1}).
\end{equation*}
Plugging back in \eqref{exit place eq 1}, we obtain
\begin{equation}\label{exit place eq 2}
	\mu_\epsilon(C_{\epsilon,j}( \rho_j r_j(\epsilon)+3\epsilon )) \sim_{\epsilon,\delta}  \delta^{-1} \sum_{k:I_k \sim O_j} \epsilon^{d-1} \lambda_k^{d-1} V_{d-1}.
\end{equation}
In addition, by \eqref{exit place qj}--\eqref{exit place eq 2},
\begin{equation*}
\begin{aligned}
\displaystyle
	&\inf_{z \in C_{\epsilon,j}( \rho_j r_j(\epsilon) + 3\epsilon)}  
	\mathbb{P}_z \left( Z^\epsilon(\sigma^{\epsilon,\delta}) \in C_{\epsilon,j}^k(\delta) \right)  \mu_\epsilon(C_{\epsilon,j}( \rho_j r_j(\epsilon)+3\epsilon ))  \left( \delta^2 + o_{\epsilon,\delta}(\delta^2) \right) + O_{\epsilon,\delta}(\epsilon^{d-1}\delta^2) \\
	\displaystyle
	& \leq \delta \lambda_k^{d-1} \epsilon^{d-1} V_{d-1} \\
	\displaystyle
	& \leq  \sup_{z \in C_{\epsilon,j}( \rho_j r_j(\epsilon) + 3\epsilon)}  
	\mathbb{P}_z \left( Z^\epsilon(\sigma^{\epsilon,\delta}) \in C_{\epsilon,j}^k(\delta) \right) 	\mu_\epsilon(C_{\epsilon,j}( \rho_j r_j(\epsilon)+3\epsilon )) \left( \delta^2 + o_{\epsilon,\delta}(\delta^2) \right) + O_{\epsilon,\delta}(\epsilon^{d-1}\delta^2)
\end{aligned}
\end{equation*}
and thus
\begin{equation}\label{exit place eq 3}
\begin{aligned}
\displaystyle
	&\inf_{z \in C_{\epsilon,j}( \rho_j r_j(\epsilon) + 3\epsilon)}  
	\mathbb{P}_z \left( Z^\epsilon(\sigma^{\epsilon,\delta}) \in C_{\epsilon,j}^k(\delta) \right) 
	  + o_{\epsilon,\delta}(1) \\
	& \leq \frac{\lambda_k^{d-1} }{\sum_{l:I_l \sim O_j} \lambda_l^{d-1} }  \leq  \sup_{z \in C_{\epsilon,j}( \rho_j r_j(\epsilon) + 3\epsilon)}  
	\mathbb{P}_z \left( Z^\epsilon(\sigma^{\epsilon,\delta}) \in C_{\epsilon,j}^k(\delta) \right) 
	 + o_{\epsilon,\delta}(1).
\end{aligned}
\end{equation}

Finally, thanks to Lemma \ref{exit place conti}, we can conclude that
\begin{equation*}
	\lim_{\delta \to 0} \lim_{\epsilon \to 0} 
	\sup_{z \in C_{\epsilon,j}( \rho_j r_j(\epsilon) + 3\epsilon)} \abs{ \mathbb{P}_z \left( Z^\epsilon(\sigma^{\epsilon,\delta}) \in C_{\epsilon,j}^k(\delta) \right)   - \frac{\lambda_k^{d-1} }{\sum_{l:I_l \sim O_j} \lambda_l^{d-1} } } = 0.
\end{equation*}

\begin{remark}
We only use Lemma \ref{exit place conti} at the conclusion of the proof.
In particular, even without assuming Lemma \ref{exit place conti}, equation \eqref{exit place eq 3} implies that there exists a constant $c>0$ such that
\begin{equation}\label{exit place sup lb}
	\lim_{\delta \to 0} \lim_{\epsilon \to 0} \sup_{z \in C_{\epsilon,j}( \rho_j r_j(\epsilon) + 3\epsilon)}  
	\mathbb{P}_z \left( Z^\epsilon(\sigma^{\epsilon,\delta}) \in C_{\epsilon,j}^k(\delta) \right) \geq c.
\end{equation}
We will use this observation in Section \ref{sec exit place}.
\end{remark}

\subsection{Proof of Proposition \ref{small exit est 2} for small and intermediate balls}\label{sec proof lem}
In this subsection, we prove Proposition \ref{small exit est 2} for small and intermediate balls.
The case of large balls will be discussed in the next subsection.
These results concern the local behavior near a vertex, so for simplicity, we will omit the subscript $j$ throughout the remaining section.

In particular, we will denote
\begin{equation*}
\begin{aligned}
C_\epsilon(\delta) & = \{ z \in G_\epsilon: d_{\Gamma} (\Pi(z),O) = \delta \}, \quad
C_\epsilon^k(\delta)  = C_\epsilon(\delta) \cap \Pi^{-1}(I_k), \\
B_\epsilon(\delta)  &= \{ z \in G_\epsilon: d_{\Gamma} (\Pi(z),O) \leq \delta \}.
\end{aligned}
\end{equation*}
Next, we define
\begin{equation*}
\begin{aligned}
	\Omega_{1,\epsilon} & \coloneqq B(O, \rho r(\epsilon)), \\
	\Omega_{2,\epsilon,\delta}^k &\coloneqq \text{the region enclosed by } C_{\epsilon}^k (\rho r(\epsilon)+3\epsilon), C_{\epsilon}^k(\delta) \text{ and } \partial G_\epsilon, \quad
	\Omega_{2,\epsilon,\delta} \coloneqq \bigcup_{k:I_k \sim O} \Omega_{2,\epsilon,\delta}^k, \\
	\Omega_{3,\epsilon,\delta} & \coloneqq B_\epsilon(\delta) \setminus \left( \Omega_{1,\epsilon} \cup \Omega_{2,\epsilon,\delta} \right).
\end{aligned}
\end{equation*}

Before turning to the proof, let us first explain the main idea. 
Our approach is closely related to \cite{NETthird}. 
We introduce an averaging operator on each edge neighborhood, thereby reducing the problem to an interval. 
Lemma \ref{exit est conti} shows that this averaging provides a good approximation at points in the edge neighborhood. 
When there is only one edge, this one dimensional approximation, together with the flux balance condition (the compatibility condition) associated with the corresponding PDE, allows us to conclude the proof. 
When multiple edges are present, we apply Lemma \ref{exit est conti} again to control the averaged fluxes across the different edge cross sections. 
Combining these estimates with the flux balance condition then yields the desired result.

We now give the details of the argument. 
It is known that (see Appendix \ref{sec NET} for more details), $v_{\epsilon,\delta}(z) = \mathbb{E}_z \sigma^{\epsilon,\delta}$ solves the PDE
\begin{equation*}
\displaystyle
\begin{dcases}
	\Delta v_{\epsilon,\delta}(z) = -1, \ \ z \in B_\epsilon(\delta) \\
	v_{\epsilon,\delta}(z) = 0, \ \ z \in C_\epsilon(\delta), \\
	\frac{\partial v_{\epsilon,\delta}(z)}{\partial \nu_{\epsilon,\delta}} = 0, \ \ z \in \partial  B_\epsilon(\delta) \setminus  C_\epsilon(\delta).
\end{dcases}
\end{equation*}
Using integration by parts, we have
\begin{equation}\label{compat sum}
	\sum_{k:I_k \sim O} \int_{C_\epsilon^k(\rho r(\epsilon)+ 3 \epsilon)} \frac{\partial v_{\epsilon,\delta}}{\partial \nu_\epsilon^{out}}(z) d\sigma(z) 
	= \int_{\Omega_{1,\epsilon} \cup \Omega_{3,\epsilon,\delta}} \Delta v_{\epsilon,\delta}(z) dz
	= -  \left( \abs{\Omega_{1,\epsilon}} + \abs{\Omega_{3,\epsilon,\delta}} \right),
\end{equation}
where $\nu_\epsilon^{out}$ is the unit normal vector pointing outward from $\Omega_{1,\epsilon} \cup \Omega_{3,\epsilon,\delta}$ into the neck region.
For $k$ such that $I_k \sim O$, we define
\begin{equation}\label{compat graph}
	\kappa_{\epsilon,\delta}^k \coloneqq \frac{1}{\abs{B_k(0,\lambda_k \epsilon)}} \int_{C_\epsilon^k(\rho r(\epsilon)+ 3 \epsilon)} \frac{\partial v_{\epsilon,\delta}}{\partial \nu_\epsilon^{out}} (z) d \sigma(z).
\end{equation}

In the following, we fix an index $k$ such that $I_k \sim O$ and focus on the cylinder domain $ \Omega_{2,\epsilon,\delta}^k$.
Therefore, by rotation and translation, we may, without loss of generality, assume that
\begin{equation*}
	\Pi(\Omega_{2,\epsilon,\delta}^k) = [0,\delta-\rho r(\epsilon)-3\epsilon],
\end{equation*}
where $\Pi(C_\epsilon^k(\rho r(\epsilon)+ 3 \epsilon)) = \delta-\rho r(\epsilon)-3\epsilon$ and $\Pi(C_\epsilon^k(\delta)) = 0$.
We thus have 
\begin{equation*}
	\Omega_{2,\epsilon,\delta}^k=[0,\delta-\rho r(\epsilon)-3\epsilon] \times B^{d-1}(0,\lambda_k \epsilon),
\end{equation*}
where $B^{d-1}(0,r)$ is a $(d-1)$-dimension ball center at $0$ with radius $r$. We then write $z=(x,y)$.

On $\Omega_{2,\epsilon,\delta}^k$, we consider 
\begin{equation*}
	v_{\epsilon,\delta}^\wedge(x)  
	\coloneqq \frac{1}{\abs{B^{d-1}(0,\lambda_k \epsilon)}} \int_{B^{d-1}(0,\lambda_k \epsilon)} v_{\epsilon,\delta}(x,y) dy.
\end{equation*}
First, observe that for $x >0$, the function $v_{\epsilon,\delta}^\wedge(x)$ is smooth, as we are away from the point where the mixed boundary conditions intersect.
Although the smoothness may depend on $x,\epsilon,\delta$, it is sufficient to justify the interchange of differentiation and integration for $x>0$.
Using the notation introduced in \eqref{compat graph}, we have
\begin{equation*}
\begin{aligned}
	& \frac{d}{d x} v_{\epsilon,\delta}^\wedge(\delta-\rho r(\epsilon)-3\epsilon) = \frac{d}{d x} \frac{1}{\abs{B^{d-1}(0,\lambda_k \epsilon)}} \int_{B^{d-1}(0,\lambda_k \epsilon)} v_{\epsilon,\delta}(x,y) dy \Big|_{x = \delta-\rho r(\epsilon)-3\epsilon}	\\
	&= \frac{1}{\abs{B^{d-1}(0,\lambda_k \epsilon)}} \int_{B^{d-1}(0,\lambda_k \epsilon)} \frac{\partial}{\partial x}  v_{\epsilon,\delta}(\delta-\rho r(\epsilon)-3\epsilon,y) dy \\
	&= \frac{1}{\abs{B^{d-1}(0,\lambda_k \epsilon)}} \int_{C_\epsilon^k(\rho r(\epsilon)+ 3 \epsilon)} \frac{\partial}{\partial \nu_\epsilon^{in}} v_{\epsilon,\delta}(z) d\sigma(z) = -\kappa_{\epsilon,\delta}^k
\end{aligned}
\end{equation*}
since $\nu_\epsilon^{in} = -\nu_\epsilon^{out}$. 
Also, using integration by parts and the Neumann boundary condition, for $x>0$,
\begin{equation*}
	\int_{B^{d-1}(0,\lambda_k \epsilon)} \Delta_y  v_{\epsilon,\delta}(x,y) dy
	=\int_{\partial B^{d-1}(0,\lambda_k \epsilon)} \frac{\partial v_{\epsilon,\delta}(x,y)}{\partial \nu_y}  d\sigma(y) = 0,
\end{equation*}
which implies
\begin{equation*}
\begin{aligned}
	&\frac{d^2}{d x^2} v_{\epsilon,\delta}^\wedge(x)
	=\frac{d^2}{d x^2} \frac{1}{\abs{B^{d-1}(0,\lambda_k \epsilon)}} 
	\int_{B^{d-1}(0,\lambda_k \epsilon)} v_{\epsilon,\delta}(x,y) dy \\
	& =  \frac{1}{\abs{B^{d-1}(0,\lambda_k \epsilon)}} \int_{B^{d-1}(0,\lambda_k \epsilon)} \frac{\partial^2}{\partial x^2} v_{\epsilon,\delta}(x,y) dy \\
	&=  \frac{1}{\abs{B^{d-1}(0,\lambda_k \epsilon)}} \int_{B^{d-1}(0,\lambda_k \epsilon)} \left( \frac{\partial^2}{\partial x^2} v_{\epsilon,\delta}(x,y) + \Delta_y  v_{\epsilon,\delta}(x,y) \right)  dy
	=-1.
\end{aligned}
\end{equation*}
We conclude that $v_{\epsilon,\delta}^\wedge(x)$ satisfies the following ODE on $\Pi(\Omega_{2,\epsilon,\delta}^k)=[0,\delta-\rho r(\epsilon)-3\epsilon]$
\begin{equation*}
\displaystyle
\begin{dcases}
	\frac{d^2}{dx^2}v_{\epsilon,\delta}^\wedge(x) = -1, \ \ x \in (0,\delta-\rho r(\epsilon)-3\epsilon) \\
	v_{\epsilon,\delta}^\wedge(0) = 0, \\
	\frac{d}{dx} v_{\epsilon,\delta}^\wedge (\delta-\rho_j r_j(\epsilon)-3\epsilon) = -\kappa_{\epsilon,\delta}^k.
\end{dcases}
\end{equation*}
We can solve it and obtain
\begin{equation*}
	v_{\epsilon,\delta}^\wedge(x) = -\frac{x^2}{2}+\left( (\delta-\rho r(\epsilon)-3\epsilon) -\kappa_{\epsilon,\delta}^k \right)x.
\end{equation*}
In particular,
\begin{equation}\label{graph sol boundary}
	v_{\epsilon,\delta}^\wedge(\delta-\rho r(\epsilon)-3\epsilon) = \frac{(\delta - \rho r(\epsilon)-3\epsilon)^2}{2} -\kappa_{\epsilon,\delta}^k (\delta - \rho r(\epsilon)-3\epsilon).
\end{equation}

Note that, for any $y \in B^{d-1}(0, \lambda_k \epsilon)$,
\begin{equation*}
\begin{aligned}
	&\abs{v_{\epsilon,\delta}^\wedge(\delta-\rho r(\epsilon)-3\epsilon) - v_{\epsilon,\delta}(\delta-\rho r(\epsilon)-3\epsilon,y)} \\
	& \leq  \sup_{y',y'' \in B^{d-1}(0,\lambda_k \epsilon )} \abs{v_{\epsilon,\delta}(\delta-\rho r(\epsilon)-3\epsilon,y')-v_{\epsilon,\delta}(\delta-\rho r(\epsilon)-3\epsilon,y'')}.
\end{aligned}
\end{equation*}
Thus, by Lemma \ref{exit est conti} and equation \eqref{graph sol boundary}, since we are dealing with small and intermediate balls, we have
\begin{equation*}
	v_{\epsilon,\delta}(\delta-\rho r(\epsilon)-3\epsilon,y) =  \frac{(\delta - \rho r(\epsilon)-3\epsilon)^2}{2} -\kappa_{\epsilon,\delta}^k (\delta - \rho r(\epsilon)-3\epsilon)+o(1),
\end{equation*}
uniformly for $y \in B^{d-1}(0,\lambda_k \epsilon )$.
In addition, the above equation holds for every $l$ such that $I_l \sim O$, we obtain
\begin{equation*}
	v_{\epsilon,\delta}(z) =  \frac{(\delta - \rho r(\epsilon)-3\epsilon)^2}{2} -\kappa_{\epsilon,\delta}^l (\delta - \rho r(\epsilon)-3\epsilon)+o(1)
\end{equation*}
uniformly in $C_{\epsilon}^l(\rho r(\epsilon)+3\epsilon)$, so that, for every $k,l$ such that $I_k,I_l \sim O$,
\begin{equation*}
\begin{aligned}
	\sup_{ z \in C_{\epsilon}^k(\rho r(\epsilon)+3\epsilon), z' \in C_{\epsilon}^l(\rho r(\epsilon)+3\epsilon)}\abs{ v_{\epsilon,\delta}(z) - v_{\epsilon,\delta}(z')}  = \abs{(\kappa_{\epsilon,\delta}^k - \kappa_{\epsilon,\delta}^l )(\delta - \rho r(\epsilon)-3\epsilon)}+o(1).
\end{aligned}
\end{equation*}
In particular, due to Lemma \ref{exit est conti}, this implies that $\abs{ \kappa_{\epsilon,\delta}^k -\kappa_{\epsilon,\delta}^l} \to 0$ as $\epsilon \to 0$.
Together with \eqref{compat sum} and \eqref{compat graph}, this implies, for all $k$ such that $I_k \sim O$,
\begin{equation*}
	-\kappa^k_{\epsilon,\delta} = \frac{\rho^d r(\epsilon)^d V_d}{\sum_{k:I_k \sim O} \lambda_k^{d-1} \epsilon^{d-1} V_{d-1}} +o(1),
\end{equation*}
and thus we conclude that
\begin{equation*}
	 v_{\epsilon,\delta}(z) = \frac{\delta^2}{2}+\alpha(0) \delta + o(1)
	 =\left( \frac{\delta^2}{2}+\alpha(0) \delta \right) (1+o(1)),
\end{equation*}
uniformly in $C_\epsilon(\rho r(\epsilon)+3\epsilon)$, where 
\begin{equation*}
	\alpha(0) = \lim_{\epsilon \to 0} \frac{\rho^d r(\epsilon)^d V_d}{\sum_{k:I_k \sim O}\lambda_k^{d-1} \epsilon^{d-1} V_{d-1}} < \infty.
\end{equation*}

\subsection{Proof of Proposition \ref{small exit est 2} for large balls}\label{sec NET bottle neck}
In this subsection, we consider the case of large balls; that is, $r(\epsilon)^d \gg \epsilon^{d-1}$. 
The only difference is that Lemma \ref{exit est conti} in this case becomes
\begin{equation*}
	\sup_{z, z' \in C_{\epsilon}(\rho r(\epsilon)+3\epsilon)} \abs{v_{\epsilon,\delta}(z) - v_{\epsilon,\delta}(z')} 
	= o\left( \frac{r(\epsilon)^d}{\epsilon^{d-1}} \right).
\end{equation*}
Arguing similarly as above, we have for every $k$ such that $I_k \sim O$
\begin{equation*}
v_{\epsilon,\delta}(z) =  \frac{(\delta -  \rho r(\epsilon)-3\epsilon)^2}{2} -\kappa_{\epsilon,\delta}^k (\delta - \rho r(\epsilon)-3\epsilon)+o\left( \frac{r(\epsilon)^d}{\epsilon^{d-1}} \right),
\end{equation*}
uniformly in $C_{\epsilon}^k(\rho r(\epsilon)+3\epsilon)$
and
\begin{equation*}
	-\kappa^k_{\epsilon,\delta} = (1+o(1))\frac{\rho^d r(\epsilon)^d V_d}{\sum_{k:I_k \sim O} \lambda_k^{d-1} \epsilon^{d-1} V_{d-1}}.
\end{equation*}
We conclude that
\begin{equation*}
	v_{\epsilon,\delta}(z) = \left( \frac{\delta^2}{2} 
		+ \frac{ \rho^d r(\epsilon)^d V_d}{\sum_{k:I_k \sim O} \lambda_k^{d-1} \epsilon^{d-1}  V_{d-1}} \delta \right) (1+o(1))
\end{equation*}
uniformly in $C_{\epsilon}(\rho r(\epsilon)+3\epsilon)$.

\begin{remark}
	For $z \in \Omega_{1,\epsilon}$, let $\mathbb{E}_z \sigma^{\epsilon}$ be the expected hitting time of $C_\epsilon(\rho r(\epsilon)+3\epsilon)$ from $z$.
	By the strong Markov property, we have
\begin{equation*}
\begin{aligned}
	v_{\epsilon,\delta}(z) &= \mathbb{E}_z \sigma^{\epsilon} + \left( \frac{\delta^2}{2} 
		+ \frac{ \rho^d r_\epsilon^d V_d}{ \sum_{k:I_k \sim O} \lambda_k^{d-1} \epsilon^{d-1}  V_{d-1}} \delta \right) (1+o(1)) \\
		 &= \mathbb{E}_z \sigma^{\epsilon} + \left( \frac{\delta^2}{2} 
		+ \frac{ \abs{\Omega_{1,\epsilon}}}{  \abs{C_\epsilon(\rho r(\epsilon)+3\epsilon)}} \delta \right) (1+o(1)),
\end{aligned}
\end{equation*}
which provides a rigorous justification of the connection formula derived in \cite{NETthird}, equation 1.

In Section \ref{sec exit time apriori}, we will show that the dominant term of $v_{\epsilon,\delta}(z)$ is given by
\begin{equation*}
	\frac{ \rho^d r(\epsilon)^d V_d}{\sum_{k:I_k \sim O} \lambda_k^{d-1} \epsilon^{d-1}  V_{d-1}} \delta
	= \left( \sum_{k:I_k \sim O}  \frac{ \lambda_k^{d-1} \epsilon^{d-1}  V_{d-1}}{\rho^d r(\epsilon)^d V_d \delta } \right)^{-1}.
\end{equation*}
This rigorously confirms the statement from \cite{NET}, Section 8.1.1:
\begin{quote}
The reciprocal of the expected exit time from a domain with several bottlenecks is the sum of the reciprocal of the expected exit times from a domain with a single bottleneck.
\end{quote}

\end{remark}

\section{A priori estimates}\label{sec apriori}
This section is devoted to prove a priori estimates for the exit time and exit place, which is crucial in our further analysis.
As our focus remains on the local behavior near a vertex, for simplicity of notations, we omit the subscript $j$ throughout this section.

\subsection{A priori estimates: exit time}\label{sec exit time apriori}

The main goal in this subsection is to prove the following lemmas:
\begin{Lemma}\label{small exit est 1.5}
   We have
    \begin{equation*}
        \displaystyle
        \lim_{\epsilon \to 0} \sup_{z \in B_{\epsilon}(\delta)} \mathbb{E}_z \sigma^{\epsilon,\delta} \left( \frac{r(\epsilon)^d }{ \epsilon^{d-1}} \delta + \delta^2 \right)^{-1} < \infty .
    \end{equation*}
\end{Lemma}

\begin{Lemma}\label{inter exit est 1}
    We have
    \begin{equation*}
         \lim_{\epsilon \to 0} \sup_{z \in B_{\epsilon}(\delta)}  \mathbb{E}_z  (\sigma^{\epsilon,\delta})^2  \left( \left( \frac{r(\epsilon)^d}{\epsilon^{d-1}} \delta \right)^2 +\delta^4 \right)^{-1} < \infty.
    \end{equation*}
\end{Lemma}

Before establishing the a priori estimates we need, we begin by decomposing the domain into several pieces.
We recall that we already defined
\begin{equation*}
\begin{aligned}
C_\epsilon(\delta) & \coloneqq \{ z \in G_\epsilon: d_{\Gamma} (\Pi(z),O) = \delta \}, &\quad
C_\epsilon^k(\delta)  &\coloneqq C_\epsilon(\delta) \cap \Pi^{-1}(I_k), \\
B_\epsilon(\delta) & \coloneqq \{ z \in G_\epsilon: d_{\Gamma} (\Pi(z),O) \leq \delta \}, \\
z_{k}^* &\coloneqq B(O,\rho r(\epsilon)) \cap I_k, &\quad
\Omega_{1,\epsilon}  &\coloneqq B(O, \rho r(\epsilon)).
\end{aligned}
\end{equation*}
Now we define 
\begin{equation*}
\begin{aligned}
\Gamma_{1,\epsilon} &\coloneqq  \partial B(O,\rho r(\epsilon)) \setminus \partial G_\epsilon , &\quad
\Gamma_{1,\epsilon}^k  &\coloneqq \Gamma_{1,\epsilon} \cap \Pi^{-1}(I_k), \\
\Gamma_{2,\epsilon}^k &\coloneqq \left\{ z \in \Omega_{1,\epsilon}  : d(z,z_{k}^*)=4 \lambda_k \epsilon \right\}, &\quad
\Gamma_{2,\epsilon} &\coloneqq \bigcup_{k:I_k \sim O} \Gamma_{2,\epsilon}^k, \\
\Gamma_{3,\epsilon}^k &\coloneqq C_\epsilon^k(\rho r(\epsilon) + 3 \epsilon ), &\quad
\Gamma_{3,\epsilon} &\coloneqq \bigcup_{k:I_k \sim O} \Gamma_{3,\epsilon}^k ,\\
\Omega_{4,\epsilon}^k & \coloneqq \text{the region enclosed by $\Gamma_{2,\epsilon}^k$, $\Gamma_{3,\epsilon}^k$ and $\partial G_\epsilon,$ } &\quad
\Omega_{4,\epsilon} &\coloneqq \bigcup_{k:I_k \sim O} \Omega_{4,\epsilon}^k.
\end{aligned}
\end{equation*}

\begin{figure}[ht]
  \centering
  \includegraphics[width=0.65\linewidth]{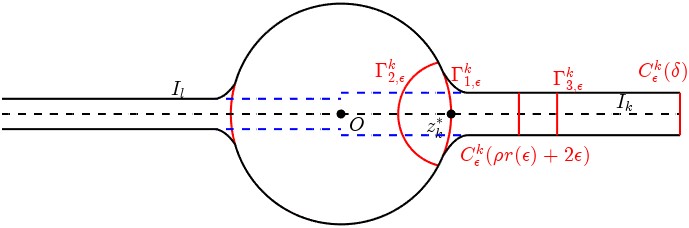}
  \caption{Domains associated with the edge $I_k$.}\label{fig:smooth dom time}
\end{figure}

In Section \ref{sec NET ball to connect}, we will study the escape problem from $\Gamma_{2,\epsilon}$ to $\Gamma_{1,\epsilon}$ inside $\Omega_{1,\epsilon}$.
In Section \ref{sec NET connect}, we will study the escape problem from either $\Gamma_{1,\epsilon}$ or $C_\epsilon(\rho r(\epsilon)+ 2 \epsilon)$ to $\Gamma_{3,\epsilon}$ with particular emphasis on the domain $\Omega_{4,\epsilon}$.
Finally, we define
\begin{equation*}
\begin{aligned}
\Omega_{5,\epsilon,\delta}^k  \coloneqq \text{ the region enclosed by $C_{\epsilon}^k(\rho r(\epsilon)+ 2 \epsilon)$, $C_{\epsilon}^k(\delta)$ and $\partial G_\epsilon$, } \quad
\Omega_{5,\epsilon,\delta} \coloneqq \bigcup_{k:I_k \sim O} \Omega_{5,\epsilon,\delta}^k ,
\end{aligned}
\end{equation*}
for $\delta \geq \rho r(\epsilon)+3\epsilon$.
In Section \ref{sec NET neck}, we will study the escape problem from $\Gamma_{3,\epsilon}$ to $C_{\epsilon}(\delta)$ with particular emphasis on the domain $\Omega_{5,\epsilon,\delta}$.

\subsubsection{Transitions in the ball}\label{sec NET ball to connect}
\begin{Lemma}\label{lem NET ball}
	Let $\sigma^{1,\epsilon}$ be the first hitting time of $\Gamma_{1,\epsilon}$.
	Then, we have,
	\begin{equation*}
		\sup_{z \in \Gamma_{2,\epsilon}} \mathbb{E}_{z} \sigma^{1,\epsilon} \lesssim \frac{r(\epsilon)^d}{\epsilon^{d-2}}
	\end{equation*}
	and
	\begin{equation*}
		\sup \mathbb{E}_{z} \sigma^{1,\epsilon} \lesssim
		\begin{dcases} 
		r(\epsilon)^2 \ln \frac{r(\epsilon)}{\epsilon}, \ \ \ d=2, \\
		 \frac{r(\epsilon)^3}{\epsilon} , \ \ \ d = 3,
		\end{dcases}
	\end{equation*}
	where supremum is taken over $z \in \Omega_{1,\epsilon}, d(z, z_k^*) \geq 2 \lambda_k \epsilon$ for all $k $ such that $I_k \sim O$.
\end{Lemma}
\begin{proof}
	Consider the process 
\begin{equation*}
	\tilde{Z}^\epsilon(t) =  \frac{1}{\rho r (\epsilon)} \left( Z^\epsilon(\rho^2 r(\epsilon)^2 t) - O \right),
\end{equation*}	
	which reduces the problem to a narrow escape problem in the unit ball. 
	Actually, we have
	\begin{equation*}
		\mathbb{E}_{z} \sigma^{1,\epsilon} = \rho^2 r(\epsilon)^2 \tilde{\mathbb{E}}_{z'} \tilde{\sigma}^{\epsilon'},
	\end{equation*}
	where $z' = \frac{z-O}{\rho r(\epsilon)}$ and $\tilde{\mathbb{E}}_{z'} \tilde{\sigma}^{\epsilon'}$ denotes the expected escape time from the unit ball through several small escape regions whose radii are of order $\epsilon' = \frac{\epsilon}{\rho r(\epsilon)}$.
The narrow escape problem in the unit ball has been extensively studied in \cite{NET4}.
	For $d=3$, using the result in \cite{NET4}, Section 3, we have
	\begin{equation*}
		\sup \tilde{\mathbb{E}}_{z'} \tilde{\sigma}^{\epsilon'}
		\lesssim \frac{1}{\epsilon'} \lesssim \frac{r(\epsilon)}{\epsilon},
	\end{equation*}		
	where the supremum is taken over $z' \in B(0,1), d(z',(z_k^*)') \geq \frac{2 \lambda_k}{\rho} \epsilon'$ for all $k$ such that $I_k \sim O$.
	
	For $d=2$, if there is only one escape region, \cite{NET4}, Section 5.1, shows \begin{equation*}
		\sup \tilde{\mathbb{E}}_{z'} \tilde{\sigma}^{\epsilon'}
		\lesssim 1,
	\end{equation*}
	where the supremum is taken over $z' \in B(0,1), d(z',(z_k^*)') \asymp \epsilon'$ for some $k$ such that $I_k \sim O$,
	and
	\begin{equation*}
		\sup \tilde{\mathbb{E}}_{z'} \tilde{\sigma}^{\epsilon'}
		\lesssim \ln \frac{1}{\epsilon'} \lesssim \ln \frac{r(\epsilon)}{\epsilon},
	\end{equation*}		
	where the supremum is taken over $z' \in B(0,1), d(z',(z_k^*)') \geq \frac{2 \lambda_k}{\rho} \epsilon'$ for all $k$ such that $I_k \sim O$.
	When multiple escape windows are present, one can obtain the same result by solving the algebraic system in \cite{NET4}, Section 5.3.
	In our setting, since we only require a rough estimate and all escape window radii are of the same order, verifying the result is straightforward.
\end{proof}

\subsubsection{Transitions in connecting parts}\label{sec NET connect}
\begin{Lemma}\label{lem NET connect}
	Let $\sigma_s^{2,\epsilon}$ be the first hitting time of $\Gamma_{3,\epsilon}$.
	Then we have
	\begin{equation*}
		\sup_{z \in \Gamma_{1,\epsilon} \cup C_\epsilon(\rho r(\epsilon) + 2 \epsilon)} \mathbb{E}_{z} \sigma_s^{2,\epsilon} \lesssim \frac{r(\epsilon)^d}{\epsilon^{d-2}}.
	\end{equation*}
\end{Lemma}
\begin{proof}
We first consider a sequence of stopping times 
\begin{equation*}
\begin{aligned}
	\sigma_n^{2,\epsilon} &\coloneqq \inf \{t \geq \tau_n^{2,\epsilon} : Z^\epsilon(t) \in \Gamma_{2,\epsilon} \cup \Gamma_{3,\epsilon} \}, \\
	\tau_n^{2,\epsilon} &\coloneqq \inf \{ t \geq \sigma_{n-1}^{2,\epsilon} : Z^\epsilon(t) \in \Gamma_{1,\epsilon} \cup C_\epsilon(\rho r(\epsilon) + 2 \epsilon) \},
\end{aligned}
\end{equation*}
with $\tau_0^{2,\epsilon} \coloneqq 0$.
Let $N^{1,\epsilon}$ be the random number such that $\sigma_{N^{1,\epsilon}}^{2,\epsilon} = \sigma_s^{2,\epsilon}$; that is, the numbers of trials needed to successfully escape to $\Gamma_{3,\epsilon}$ (instead of $\Gamma_{2,\epsilon}$).
We claim that the following two estimates hold
\begin{equation}\label{reg per 1}
	\sup_{z \in \Gamma_{1,\epsilon} \cup C_\epsilon(\rho r(\epsilon) + 2 \epsilon) } \mathbb{E}_z N^{1,\epsilon} \lesssim 1
\end{equation}
and
\begin{equation}\label{reg per 2}
	\sup_{z \in \Gamma_{1,\epsilon} \cup C_\epsilon(\rho r(\epsilon) + 2 \epsilon)} \mathbb{E}_z \sigma_0^{2,\epsilon} \lesssim \epsilon^2.
\end{equation}

It is not difficult to see that equation \eqref{reg per 1} and \eqref{reg per 2} depend only on the exit problem in $\Omega_{4,\epsilon}^k$. 
Indeed, for \eqref{reg per 1} it suffices to show that for each $k$, there exists a constant $c_k>0$ such that
\begin{equation}\label{reg per 3}
	\inf_{z \in \Gamma_{1,\epsilon}^k \cup C_\epsilon^k(\rho r(\epsilon) + 2 \epsilon)} \mathbb{P}_z \left( Z^\epsilon(\sigma_0^{2,\epsilon}) \in \Gamma_{3,\epsilon} \right) \geq c_k,
\end{equation}
which implies
\begin{equation*}
	\sup_{z \in \Gamma_{1,\epsilon} \cup C_\epsilon(\rho r(\epsilon) + 2 \epsilon) } \mathbb{E}_z N^{1,\epsilon} 
	\leq  \frac{1}{\min_{k:I_k \sim O} c_k}.
\end{equation*}
Therefore, it is enough to analyze the exit problem within each $\Omega_{4,\epsilon}^k$. 
Our strategy is to perform the change of variables $z \mapsto \frac{z-z_k^*}{\epsilon}$ within the domain $\Omega_{4,\epsilon}^k$ in order to reduce or eliminate its dependence on $\epsilon$. 
If the transformed domain were exactly independent of $\epsilon$, then \eqref{reg per 2} would follow directly from the Brownian scaling property, and \eqref{reg per 1} would follow from the fact that the success probability in each trial is independent of $\epsilon$; that is, \eqref{reg per 3} would hold with equality, implying that $N^{1,\epsilon}$ remains of constant order.
However, the curvature induced by $r(\epsilon)$ means that the rescaled domain still retains some mild $\epsilon$-dependence. 
Nevertheless, estimates \eqref{reg per 1} and \eqref{reg per 2} remain valid.
The key idea is to regard the rescaled domain as a regular perturbation of a fixed smooth domain.
A detailed justification of this procedure is given in Appendix \ref{App reg per}, where we prove \eqref{reg per 1} and \eqref{reg per 2}.

Finally, for all $z \in \Gamma_{1,\epsilon} \cup C_\epsilon(\rho r(\epsilon) + 2 \epsilon)$, we have
\begin{equation}\label{eq wald type identity}
\begin{aligned}
	&\mathbb{E}_{z} \sigma_s^{2,\epsilon} 
	=\sum_{l=0}^\infty \mathbb{E}_{z} \left (\sigma^{2,\epsilon}_{N^{1,\epsilon}} \mathbbm{1}_{ \{N^{1,\epsilon}=l \} } \right) \\
	& = \sum_{l=0}^\infty \mathbb{E}_z \left[  \sum _{i=0}^l  \left( \sigma_i^{2,\epsilon} - \tau_i^{2,\epsilon} \right) \mathbbm{1}_{ \{N^{1,\epsilon}=l \} } \right]
	+ \sum_{l=1}^\infty \mathbb{E}_z \left[  \sum _{i=1}^l \left( \tau_i^{2,\epsilon} - \sigma_{i-1}^{2,\epsilon} \right) \mathbbm{1}_{ \{N^{1,\epsilon}=l \} }  \right]  \\
	&= \sum_{i=0}^\infty \sum _{l=i}^\infty \mathbb{E}_z \left[    \left( \sigma_i^{2,\epsilon} - \tau_i^{2,\epsilon} \right) \mathbbm{1}_{ \{N^{1,\epsilon}=l \} } \right]
	+ \sum_{i=1}^\infty \sum _{l=i}^\infty \mathbb{E}_z \left[   \left( \tau_i^{2,\epsilon} - \sigma_{i-1}^{2,\epsilon} \right) \mathbbm{1}_{ \{N^{1,\epsilon}=l \} }  \right]  \\
	&=  \sum_{i=0}^\infty \mathbb{E}_z \left[    \left( \sigma_i^{2,\epsilon} - \tau_i^{2,\epsilon} \right) \mathbbm{1}_{ \{N^{1,\epsilon} \geq i \} } \right]
	+ \sum_{i=1}^\infty \mathbb{E}_z \left[   \left( \tau_i^{2,\epsilon} - \sigma_{i-1}^{2,\epsilon} \right) \mathbbm{1}_{ \{N^{1,\epsilon} \geq i \} }  \right] .
\end{aligned}
\end{equation}
Notice that $\{N^{1,\epsilon} \geq i \} = \{N^{1,\epsilon} \leq i-1\}^c \in \mathcal{F}_{\sigma_{i-1}^{2,\epsilon}} \subset \mathcal{F}_{\tau_{i}^{2,\epsilon}}$, so that, by the strong Markov property, 
\begin{equation}\label{eq wald type identity 2}
\begin{aligned}
	&\mathbb{E}_{z} \sigma_s^{2,\epsilon} 
	=  \sum_{i=0}^\infty \mathbb{E}_z \left[    \left( \sigma_i^{2,\epsilon} - \tau_i^{2,\epsilon} \right) \mathbbm{1}_{ \{N^{1,\epsilon} \geq i \} } \right]
	+ \sum_{i=1}^\infty \mathbb{E}_z \left[   \left( \tau_i^{2,\epsilon} - \sigma_{i-1}^{2,\epsilon} \right) \mathbbm{1}_{ \{N^{1,\epsilon} \geq i \} }  \right]  \\
	&=  \sum_{i=0}^\infty  \mathbb{E}_z \left[ \mathbbm{1}_{ \{N^{1,\epsilon} \geq i \} }  \mathbb{E}_{Z^\epsilon(\tau_i^{2,\epsilon})}  \sigma_0^{2,\epsilon}  \right]
	+ \sum_{i=1}^\infty \mathbb{E}_z \left[ \mathbbm{1}_{ \{N^{1,\epsilon} \geq i \} } \mathbb{E}_{Z^\epsilon(\sigma_{i-1}^{2,\epsilon})}  \tau_1^{2,\epsilon}   \right] \\
	&\lesssim  \sum_{i=0}^\infty \mathbb{P}_z (N^{1,\epsilon} \geq i) \epsilon^{2} + \sum_{i=1}^\infty \mathbb{P}_z (N^{1,\epsilon} \geq i) \frac{r(\epsilon)^d}{\epsilon^{d-2}} 
	\asymp \mathbb{E}_z N^{1,\epsilon} \left( 
	\epsilon^2 + \frac{r(\epsilon)^d}{\epsilon^{d-2}}
	\right) \asymp  \frac{r(\epsilon)^d}{\epsilon^{d-2}},
\end{aligned}
\end{equation}
since, by Lemma \ref{lem NET ball}, for all $i =1,2,\cdots, N^{1,\epsilon}$, we have  $ Z^\epsilon(\sigma_{i-1}^{2,\epsilon}) \in \Gamma_{2,\epsilon}$ and thus
\begin{equation*}
	\mathbb{E}_{Z^\epsilon(\sigma_{i-1}^{2,\epsilon})} \tau_1^{2,\epsilon}  	
	\leq \sup_{z'' \in \Gamma_{2,\epsilon}} \mathbb{E}_{z''} \sigma^{1,\epsilon} 
	\lesssim \frac{r(\epsilon)^d}{\epsilon^{d-2}}, \ \ \ \mathbb{P} \ a.s.
\end{equation*}
\end{proof}

\subsubsection{Transitions in cylinders}\label{sec NET neck}
\begin{Lemma}\label{lem NET neck}
	Recall that $\sigma^{\epsilon,\delta}$ is the first hitting time of $C_\epsilon(\delta)$.
	We have
	\begin{equation*}
		\sup_{z \in \Gamma_{3,\epsilon}} \mathbb{E}_{z} \sigma^{\epsilon,\delta} \lesssim \delta^2 + \frac{r(\epsilon)^d}{\epsilon^{d-1}} \delta.
	\end{equation*}
\end{Lemma}
\begin{proof}
We introduce the sequence of stopping times 
\begin{equation*}
\begin{aligned}
	\sigma_n^{3,\epsilon,\delta} &\coloneqq \inf \{t \geq \tau_n^{3,\epsilon,\delta} : Z^\epsilon(t) \in C_\epsilon(\rho r(\epsilon) + 2 \epsilon) \cup C_\epsilon(\delta) \}, \\
	\tau_n^{3,\epsilon,\delta} &\coloneqq \inf \{ t \geq \sigma_{n-1}^{3,\epsilon,\delta} : Z^\epsilon(t) \in \Gamma_{3,\epsilon} \},
\end{aligned}
\end{equation*}
with $\tau_0^{3,\epsilon,\delta} \coloneqq 0$.
Let $N^{2,\epsilon,\delta}$ be the random number such that $\sigma_{N^{2,\epsilon,\delta}}^{3,\epsilon,\delta} = \sigma^{\epsilon,\delta}$; that is, the number of trials needed to successfully escape to $C_\epsilon(\delta)$ (instead of $C_\epsilon(\rho r(\epsilon) + 2 \epsilon)$) from $\Gamma_{3,\epsilon}$.

Note that by construction, $\Omega_{5,\epsilon,\delta}$ is composed of several cylinder domains. 
We can use the behavior of one-dimensional Brownian motion to derive effective bounds for the exit problem within $\Omega_{5,\epsilon,\delta}$.  
Specifically, we have
\begin{equation*}
	\sup_{z \in \Gamma_{3,\epsilon}} \mathbb{E}_z N^{2,\epsilon,\delta} \lesssim \frac{\delta}{\epsilon}.
\end{equation*}
In fact, whenever the process hits $\Gamma_{3,\epsilon}$, it hits $\Gamma_{3,\epsilon}^k$ for some $k$.
The escape rate $p^{2,\epsilon,\delta}(z)$ is related only to the one dimension Brownian motion, and we have
\begin{equation*}
	\inf_{z \in \Gamma_{3,\epsilon}^k}  p^{2,\epsilon,\delta}(z) \gtrsim \frac{\epsilon}{\delta}.
\end{equation*}
Thus
\begin{equation*}
\sup_{z \in \Gamma_{3,\epsilon}} \mathbb{E}_z N^{2,\epsilon,\delta} \leq \left( \min_{k:I_k \sim O} \inf_{z \in \Gamma_{3,\epsilon}^k}  p^{2,\epsilon,\delta}(z) \right)^{-1}  \lesssim \frac{\delta}{\epsilon}.
\end{equation*}
In addition,
\begin{equation*}
	\sup_{z \in \Gamma_{3,\epsilon}} \mathbb{E}_{z}\sigma_0^{3,\epsilon,\delta} \lesssim \epsilon \delta.
\end{equation*}

Then, for all $z \in \Gamma_{3,\epsilon}$, using arguments analogous to those in \eqref{eq wald type identity} and \eqref{eq wald type identity 2}, we have
\begin{equation*}
\begin{aligned}
	&\mathbb{E}_{z} \sigma^{\epsilon,\delta} 
	= \sum_{i=0}^\infty  \mathbb{E}_z \left[ \mathbbm{1}_{ \{N^{2,\epsilon,\delta} \geq i \} } \mathbb{E}_{Z^\epsilon(\tau_i^{3,\epsilon,\delta})} \sigma_0^{3,\epsilon,\delta}  \right] 
	 + \sum_{i=1}^\infty \mathbb{E}_z \left[ \mathbbm{1}_{ \{N^{2,\epsilon,\delta} \geq i \} } \mathbb{E}_{Z^\epsilon(\sigma_{i-1}^{3,\epsilon,\delta} )}  \tau_1^{3,\epsilon,\delta} \right] \\
	& \lesssim  \mathbb{E}_z N^{2,\epsilon,\delta} \left( \epsilon \delta + \frac{r(\epsilon)^d}{\epsilon^{d-2}} \right)
	\lesssim \delta^2 + \frac{r(\epsilon)^d}{\epsilon^{d-1}} \delta
\end{aligned}
\end{equation*}
since, by Lemma \ref{lem NET connect}, for every $i =1,2,\cdots,N^{2,\epsilon,\delta}$, we have $Z^\epsilon(\sigma_{i-1}^{3,\epsilon,\delta} ) \in C_\epsilon(\rho r(\epsilon)+2\epsilon)$ and thus
\begin{equation*}
	\mathbb{E}_{Z^\epsilon(\sigma_{i-1}^{3,\epsilon,\delta} )} \tau_1^{3,\epsilon,\delta}
	\leq \sup_{z'' \in C_\epsilon(\rho r(\epsilon) + 2\epsilon)} \mathbb{E}_{z''} \sigma^{2,\epsilon}
	\lesssim \frac{r(\epsilon)^d}{\epsilon^{d-2}}, \ \ \ \mathbb{P} \ a.s.
\end{equation*}

\end{proof}

\subsubsection{Proof of Lemma \ref{small exit est 1.5} and \ref{inter exit est 1}}
\begin{proof}[Proof of Lemma \ref{small exit est 1.5}]
	For $d=2$, we first consider $z \in \Omega_{1,\epsilon}$ such that $d(z,z_k^*) \geq 2 \lambda_k \epsilon$.
	By the strong Markov property, Lemma \ref{lem NET connect} and Lemma \ref{lem NET neck},
	\begin{equation*}
	\begin{aligned}
		&\sup_{z \in \Omega_{1,\epsilon}, d(z,z_k^*) \geq 2 \lambda_k \epsilon} \mathbb{E}_z \sigma^{\epsilon,\delta} 
		 = \sup_{z \in \Omega_{1,\epsilon}, d(z,z_k^*) \geq 2 \lambda_k \epsilon} \mathbb{E}_z \left[ \sigma^{1,\epsilon} 
		+ \mathbb{E}_{Z^\epsilon(\sigma^{1,\epsilon} )} \left( \sigma^{2,\epsilon} 
		+ \mathbb{E}_{Z^\epsilon(\sigma^{2,\epsilon})} \sigma^{3,\epsilon,\delta} \right) \right]  \\
		& \lesssim  \sup_{z \in \Omega_{1,\epsilon}, d(z,z_k^*) \geq 2 \lambda_k \epsilon} \mathbb{E}_z \sigma^{1,\epsilon}  
		+r(\epsilon)^2 
		+ \delta^2 + \frac{r(\epsilon)^2}{\epsilon} \delta 
	\end{aligned}
	\end{equation*}
	Thanks to Lemma \ref{lem NET ball}, we know that $\sup_{z \in \Omega_{1,\epsilon}, d(z,z_k^*) \geq 2 \lambda_k \epsilon} \mathbb{E}_z \sigma^{1,\epsilon} $ is at most $r(\epsilon)^2 \ln \frac{r(\epsilon)}{\epsilon}.$ 
	Given the assumption that $r(\epsilon) = o(1)$, and taking the limit $\epsilon \to 0$ first, we have
\begin{equation*}
	\sup_{z \in \Omega_{1,\epsilon}, d(z,z_k^*) \geq 2 \lambda_k \epsilon} \mathbb{E}_z^\epsilon \sigma^{1,\epsilon} \lesssim  r(\epsilon)^2 \ln \frac{r(\epsilon)}{\epsilon} \lesssim  r(\epsilon)^2 \ln \frac{1}{\epsilon} \lesssim \frac{r(\epsilon)^2}{\epsilon} \delta
\end{equation*}	
	and $r(\epsilon)^2$ is dominated by $\delta^2$ and the proof is done.
    
	Next, for each $k$ such that $I_k \sim O$, consider $z$ lying in the region enclosed by $\{z \in \Omega_{1,\epsilon}: d(z,z_k^*) = 2 \lambda_k \epsilon\}$, $C_\epsilon^k(\rho r(\epsilon)+2\epsilon)$ and $\partial G_\epsilon$.
	By the strong Markov property, an analogue of \eqref{reg per 2}, Lemma \ref{lem NET connect}, and Lemma \ref{lem NET neck}, we have
	\begin{equation*}
	\begin{aligned}
		&\sup_{z} \mathbb{E}_z \sigma^{\epsilon,\delta}
		=\sup_{z} \mathbb{E}_{z} \left[ \sigma^{2,\epsilon}_0 
		+ \mathbb{E}_{Z^\epsilon(\sigma^{2,\epsilon} _0)} \left( \sigma_s^{2,\epsilon}  
		+ \mathbb{E}_{Z^\epsilon(\sigma_s^{2,\epsilon} )}  \sigma^{3,\epsilon,\delta} \right) \right] 
		 \lesssim \epsilon^2 + r(\epsilon)^2 + \delta^2 + \frac{r(\epsilon)^2}{\epsilon} \delta
		\lesssim \delta^2 + \frac{r(\epsilon)^2}{\epsilon} \delta,
	\end{aligned}
	\end{equation*}
	where the supremum is taken over the $z$ in the described region.
	More precisely, equation \eqref{reg per 2} remains valid for this region as well; its justification follows the same argument as in Appendix \ref{App reg per}.
	
	Finally, for each $k$ such that $I_k \sim O$, consider $z$ lying in the region enclosed by $C_\epsilon^k(\rho r(\epsilon)+2\epsilon)$, $C_\epsilon^k(\delta)$ and $\partial G_\epsilon$, for some $k$.
	Define $\sigma^{4,\epsilon,\delta}$ be the first hitting time to $C_\epsilon(\rho r(\epsilon)+2\epsilon) \cup C_\epsilon(\delta)$.
	By the preceding estimates, we have
	\begin{equation*}
	\begin{aligned}
		&\sup_z \mathbb{E}_z \sigma^{\epsilon,\delta}
		= \sup_z \left( \mathbb{E}_z \left[ \sigma^{\epsilon,\delta} \mathbbm{1}_{ \{ \sigma^{4,\epsilon,\delta} < \sigma^{\epsilon,\delta} \} } \right]
		+\mathbb{E}_z \left[ \sigma^{\epsilon,\delta} \mathbbm{1}_{ \{\sigma^{4,\epsilon,\delta} = \sigma^{\epsilon,\delta} \} } \right] \right) \\
		& \leq \sup_z \mathbb{E}_z \left[ \mathbbm{1}_{ \{ \sigma^{4,\epsilon,\delta} < \sigma^{\epsilon,\delta} \} } \left( \sigma^{4,\epsilon,\delta} + \mathbb{E}_{Z^\epsilon(\sigma^{4,\epsilon,\delta})}  \sigma^{\epsilon,\delta}  \right) \right]
		+\sup_z \mathbb{E}_z \left[ \sigma^{4,\epsilon,\delta} \mathbbm{1}_{ \{\sigma^{\epsilon,\delta} = \sigma^{4,\epsilon,\delta} \} } \right] \\
		& \leq \sup_{z \in C_\epsilon(\rho r(\epsilon)+2\epsilon)} \mathbb{E}_z \sigma^{\epsilon,\delta} + \sup_z \mathbb{E}_z  \sigma^{4,\epsilon,\delta} 
		\lesssim \delta^2 + \frac{r(\epsilon)^2}{\epsilon} \delta + \delta^2
		\lesssim \delta^2 + \frac{r(\epsilon)^2}{\epsilon} \delta ,
	\end{aligned}
	\end{equation*}
	where the supremum is taken over the $z$ in the specified region.
	We then completes the proof for $d=2$.
	The case $d = 3$ can be treated with similar arguments.
\end{proof}

\begin{proof}[Proof of Lemma \ref{inter exit est 1}]
	By Theorem \ref{SMFPT Main} and Lemma \ref{small exit est 1.5}, for $\epsilon>0$ small enough,
	\begin{equation*}
        \sup_{z \in B_{\epsilon}(\delta)}  \mathbb{E}_z  (\sigma^{\epsilon,\delta})^2 
        \leq c  \left( \sup_{z \in B_{\epsilon}(\delta)}   \mathbb{E}_z \sigma^{\epsilon,\delta} \right)^2 
        \leq c \left(\left( \frac{r(\epsilon)^d}{\epsilon^{d-1}} \delta \right)^2 +\delta^4 \right).
    \end{equation*}
\end{proof}

\subsection{A priori estimates: exit place}\label{sec exit place}
In this subsection, we aim to prove the following lemma
\begin{Lemma}\label{exit place apriori}
	There exists a constant $c>0$ such that for all $l$ such that $I_l \sim O$, we have
	\begin{equation*}
		\lim_{\delta' \to 0} \lim_{\epsilon \to 0} \inf_{z \in C_\epsilon(\rho r(\epsilon)+3\epsilon)} 
\mathbb{P}_{z} \left( Z^\epsilon(\tilde{\tau}^{\epsilon,\delta'}) \in C_\epsilon^l(\delta') \right)
\geq c,
	\end{equation*}
	where $\tilde{\tau}^{\epsilon,\delta'}$ is the first hitting time of $C_\epsilon(\delta')$.
	In particular, for all $k,l$ such that $I_k, I_l \sim O$,
	\begin{equation*}
		\lim_{\delta' \to 0} \lim_{\epsilon \to 0} \inf_{z \in C_\epsilon^k(\rho r(\epsilon)+3\epsilon)} 
\mathbb{P}_{z} \left( Z^\epsilon(\tilde{\tau}^{\epsilon,\delta'}) \in C_\epsilon^l(\delta') \right)
\geq c.
	\end{equation*}
\end{Lemma}

We recall that
\begin{equation*}
\begin{aligned}
C_\epsilon(\delta) & \coloneqq \{ z \in G_\epsilon: d_{\Gamma} (\Pi(z),O) = \delta \}, &\quad
C_\epsilon^k(\delta)  &\coloneqq C_\epsilon(\delta) \cap \Pi^{-1}(I_k), \\
B_\epsilon(\delta) & \coloneqq \{ z \in G_\epsilon: d_{\Gamma} (\Pi(z),O) \leq \delta \}, \\
z_{k}^* &\coloneqq B(O,\rho r(\epsilon)) \cap I_k, &\quad
\Omega_{1,\epsilon}  &\coloneqq B(O, \rho r(\epsilon)), \\
\Gamma_{1,\epsilon} & \coloneqq \partial B(O,\rho r(\epsilon)) \setminus \partial G_\epsilon , & \quad
\Gamma_{1,\epsilon}^k  &\coloneqq \Gamma_{1,\epsilon} \cap \Pi^{-1}(I_k).
\end{aligned}
\end{equation*}
Note that, since our bottlenecks are well-separated, we can find a constant $\tilde{\rho} \in (0,\rho)$ such that for every $k \neq l$, $B(z_k^*, \tilde{\rho}  r(\epsilon)) \cap B(z_l^*, \tilde{\rho} r(\epsilon)) = \emptyset$ when $\epsilon$ is small enough.
Once fixed $\tilde{\rho}$, we define
\begin{equation*}
\begin{aligned}
\Gamma_{4,\epsilon}^k &\coloneqq \left\{ z \in \Omega_{1,\epsilon} : d(z,z_k^*)= \tilde{\rho}  r(\epsilon) \right\} , & 
\Gamma_{4,\epsilon} & \coloneqq \bigcup_{k:I_k \sim O} \Gamma_{4,\epsilon}^k, \\
\Omega_{6,\epsilon}^k & \coloneqq \left\{ z \in \Omega_{1,\epsilon} : d(z,z_k^*) \geq \tilde{\rho}  r(\epsilon) \right\} , & 
\Omega_{6,\epsilon} & \coloneqq \bigcap_{k:I_k \sim O} \Omega_{6,\epsilon}^k , \\
\Omega_{7,\epsilon,\delta}^k & \coloneqq \left( B_\epsilon(\delta+1) \cap \Pi^{-1}(I_k) \right) \cup \Omega_{1,\epsilon}, & \\
\Omega_{8,\epsilon,\delta}^k & \coloneqq \text{ the region enclosed by } \Gamma_{4,\epsilon}^k, C_\epsilon^k(\delta+1) \text{ and } \partial G_\epsilon.
\end{aligned}
\end{equation*}
In other words, the domain $\Omega_{7,\epsilon,\delta}^k$ is a ball with single bottleneck of length approximately $\delta+1$; see Figure \ref{fig:smooth dom place}. 
The following lemma is the only place where we consider this domain; afterward, we return to the general case with multiple bottlenecks.
\begin{figure}[ht]
  \centering
  \includegraphics[width=0.65\linewidth]{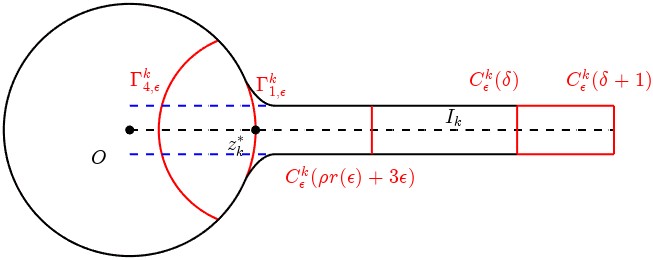}
  \caption{Illustration of the domain $\Omega_{7,\epsilon,\delta}^k$.}\label{fig:smooth dom place}
\end{figure}

\begin{Lemma}\label{exit place lem 1}
	Let $\sigma_k^{5,\epsilon,\delta}$ be the first hitting time of $\Gamma_{4,\epsilon}^k \cup C_\epsilon^k(\delta)$. Then
	\begin{equation*}
		\lim_{\epsilon \to 0} \sup_{z \in C_\epsilon^k(\rho r(\epsilon)+3\epsilon)} \mathbb{P}_z ( Z^\epsilon (\sigma^{5,\epsilon,\delta}_k) \in C_\epsilon^k(\delta) ) = 0.
	\end{equation*}
\end{Lemma}
\begin{proof}
The proof is similar to that of Proposition \ref{exit place est}.
Consider the scaled reflected Brownian motion $\hat{Z}^\epsilon_k(t)$ on $\Omega_{7,\epsilon,\delta}^k $ and the sequence of stopping times
\begin{equation*}
\begin{aligned}
\sigma_{k,n}^{5,\epsilon,\delta} &\coloneqq \inf  \{t \geq  \tau_{k,n}^{5,\epsilon,\delta} : \hat{Z}^\epsilon_k(t) \in \Gamma_{4,\epsilon}^k \cup C_\epsilon^k(\delta) \}, \\
\tau_{k,n}^{5,\epsilon,\delta} &\coloneqq \inf  \{t \geq  \sigma_{k,n-1}^{5,\epsilon,\delta} : \hat{Z}^\epsilon_k(t) \in C_\epsilon^k(\rho r(\epsilon)+3\epsilon) \},
\end{aligned}
\end{equation*}
with $\tau_{k,0}^{5,\epsilon,\delta} \coloneqq 0$.
We also define
\begin{equation*}
\begin{aligned}
\sigma_{k}^{5,\epsilon,\delta} &\coloneqq \inf  \{t \geq  0 : \hat{Z}^\epsilon_k(t) \in \Gamma_{4,\epsilon}^k \cup C_\epsilon^k(\delta) \}, \\
\tau_{k}^{5,\epsilon,\delta} &\coloneqq \inf  \{t \geq  0 : \hat{Z}^\epsilon_k(t) \in C_\epsilon^k(\rho r(\epsilon)+3\epsilon) \}.
\end{aligned}
\end{equation*}
It is known that $\hat{Z}^\epsilon_k(t)$ has Lebesgue measure $Leb$ on $\Omega_{7,\epsilon,\delta}^k $ as its invariant measure.
Define $\hat{\mu}_\epsilon^k$ be the invariant measure of the process $\hat{Z}^\epsilon_k(\tau_{k,n}^{5,\epsilon,\delta})$ on $C_\epsilon^k(\rho r(\epsilon)+3\epsilon)$.
We have the following relation (see \cite{EPD}, also \cite{SDE2}), for every Borel set $A_{\epsilon,\delta}^k \subset \Omega_{7,\epsilon,\delta}^k$ ,
\begin{equation*}
	Leb(A_{\epsilon,\delta}^k) = \int_{C_\epsilon^k(\rho r(\epsilon)+3\epsilon)} \left( \hat{\mathbb{E}}_z \int_0^{\tau_{k,1}^{5,\epsilon,\delta}} \mathbbm{1}_{A_{\epsilon,\delta}^k} \left( \hat{Z}^\epsilon_k(t) \right) dt \right) \hat{\mu}_\epsilon^k(dz).
\end{equation*}
First, let $A_{1,\epsilon,\delta}^k$ be the region enclosed by $C_\epsilon^k(\delta)$, $C_\epsilon^k(\delta+1)$ and $\partial G_\epsilon$.
By the strong Markov property, 
\begin{equation*}
\begin{aligned}
	 Leb(A^k_{1,\epsilon,\delta})  
	= \int_{C_\epsilon^k(\rho r(\epsilon)+3\epsilon)}  \hat{\mathbb{E}}_z \left( \mathbbm{1}_{\{\hat{Z}^\epsilon_k(\sigma_{k}^{5,\epsilon,\delta}) \in C_\epsilon^k(\delta) \}} \hat{\mathbb{E}}_{\hat{Z}^\epsilon_k(\sigma_{k}^{5,\epsilon,\delta})}  \int_0^{\tau_{k}^{5,\epsilon,\delta}} \mathbbm{1}_{A^k_{1,\epsilon,\delta}}  \left( \hat{Z}^\epsilon_k(t) \right) dt   \right) \hat{\mu}_\epsilon^k(dz) .
\end{aligned}
\end{equation*}
For $z' \in C_\epsilon^k(\delta)$, we claim that the occupation time 
\begin{equation*}
\hat{\mathbb{E}}_{z'} \int_0^{\tau_{k}^{5,\epsilon,\delta}} \mathbbm{1}_{A^k_{1,\epsilon,\delta}} \left( \hat{Z}^\epsilon_k(t) \right) dt = \delta-\rho r(\epsilon)-3\epsilon.
\end{equation*}
In fact, the occupation time only relates to a one dimensional Brownian motion, and thus it suffices to solve the ODE on $[0,1+\delta-\rho r(\epsilon)-3\epsilon]$
\begin{equation*}
\begin{dcases}
	f''(x) = -\mathbbm{1}_{[\delta-\rho r(\epsilon)-3\epsilon,1+\delta-\rho r(\epsilon)-3\epsilon] }(x), \ \ \ x \in [0,1+\delta-\rho r(\epsilon)-3\epsilon] ,\\
	f'(1+\delta-\rho r(\epsilon)-3\epsilon) = 0 ,  \ \ \ \
	f(0) = 0.
\end{dcases}
\end{equation*}
One then gets, for $z' \in C_\epsilon^k(\delta)$,
\begin{equation*}
	\hat{\mathbb{E}}_{z'} \int_0^{\tau_{k}^{5,\epsilon,\delta}} \mathbbm{1}_{A^k_{1,\epsilon,\delta}} \left( \hat{Z}^\epsilon_k(t) \right) dt
	= f(\delta-\rho r(\epsilon)-3\epsilon)
	= \delta-\rho r(\epsilon)-3\epsilon
\end{equation*}
since
\begin{equation*}
    f(x)
    =
    \begin{dcases}
	x, \quad x \in [0,\delta-\rho r(\epsilon)-3\epsilon] ,\\
	x-\frac{1}{2}(x-(\delta-\rho r(\epsilon)-3\epsilon))^2, \quad [\delta-\rho r(\epsilon)-3\epsilon,1+\delta-\rho r(\epsilon)-3\epsilon].
\end{dcases}
\end{equation*}
Therefore, if we define
\begin{equation}\label{exit place q}
	q_{\epsilon,\delta}^k \coloneqq \int_{C_\epsilon^k(\rho r(\epsilon)+3\epsilon)} \hat{\mathbb{P}}_z \left( \hat{Z}^\epsilon_k(\sigma_{k}^{5,\epsilon,\delta}) \in C_\epsilon^k(\delta) \right) \hat{\mu}_\epsilon^k(dz),
\end{equation}
we obtain
\begin{equation}\label{exit place ineq 1}
	q_{\epsilon,\delta}^k = \frac{Leb(A^k_{1,\epsilon,\delta})}{ \delta-\rho r(\epsilon)-3\epsilon } = \frac{\lambda_k^{d-1} \epsilon^{d-1}V_{d-1}}{\delta-\rho r(\epsilon)-3\epsilon}.
\end{equation}

Next, if we take $A^k_{2,\epsilon} = \Omega_{6,\epsilon}^k$, by the strong Markov property, we have
\begin{equation*}
\begin{aligned}
	Leb(A^k_{2,\epsilon}) 
	= \int_{C_\epsilon^k(\rho r(\epsilon)+3\epsilon)}  \hat{\mathbb{E}}_z \left( \mathbbm{1}_{\{\hat{Z}^\epsilon_k(\sigma_{k}^{5,\epsilon,\delta}) \in \Gamma_{4,\epsilon}^k\}} \hat{\mathbb{E}}_{\hat{Z}^\epsilon_k(\sigma_{k}^{5,\epsilon,\delta})}  \int_0^{\tau_{k}^{5,\epsilon,\delta}} \mathbbm{1}_{A^k_{2,\epsilon}} \left( \hat{Z}^\epsilon_k(t) \right) dt   \right) \hat{\mu}_\epsilon^k(dz) .
\end{aligned}
\end{equation*}
Consequently,
\begin{equation*}
Leb(A^k_{2,\epsilon}) 
\leq \left( \sup_{z' \in \Gamma_{4,\epsilon}^k} \hat{\mathbb{E}}_{z'}  \int_0^{\tau_{k}^{5,\epsilon,\delta}} \mathbbm{1}_{A^k_{2,\epsilon}} \left( \hat{Z}^\epsilon_k(t) \right) dt  \right)
\left(\hat{\mu}_\epsilon^k(C_\epsilon^k(\rho r(\epsilon)+3\epsilon)) -q_{\epsilon,\delta}^k \right).
\end{equation*}
Now we claim that
\begin{equation}\label{exit place occ time}
	\sup_{z' \in \Gamma_{4,\epsilon}^k} \hat{\mathbb{E}}_{z'}  \int_0^{\tau_{k}^{5,\epsilon,\delta}} \mathbbm{1}_{A^k_{2,\epsilon}} \left( \hat{Z}^\epsilon_k(t) \right) dt
	\lesssim 
	\begin{dcases}
	r(\epsilon)^2 \ln \frac{r(\epsilon)}{\epsilon}, \ \  d=2, \\
	\frac{r(\epsilon)^3}{\epsilon}, \ \ d =3.
	\end{dcases}
\end{equation}
To see this, let $\sigma^{\epsilon}_k$ be the first hitting time of $C_\epsilon^k(\rho r(\epsilon)+3\epsilon)$.
Since $\Omega_{7,\epsilon,\delta}^k$ is a single bottleneck domain, an argument similar to that in Section \ref{sec NET ball to connect} and \ref{sec NET connect} yields
\begin{equation*}
	\sup_{z' \in \Gamma_{4,\epsilon}^k } \hat{\mathbb{E}}_{z'}  \int_0^{\tau_{k}^{5,\epsilon,\delta}} \mathbbm{1}_{A^k_{2,\epsilon}} \left( \hat{Z}^\epsilon_k(t) \right) dt
	\leq \sup_{z' \in \Omega_{1,\epsilon} } \hat{\mathbb{E}}_{z'}  \sigma^{\epsilon}_k
	\lesssim 
	\begin{dcases}
	r(\epsilon)^2 \ln \frac{r(\epsilon)}{\epsilon}, \ \  d=2, \\
	\frac{r(\epsilon)^3}{\epsilon}, \ \ d =3.
	\end{dcases}
\end{equation*}

In the case $d=2$, \eqref{exit place occ time} gives
\begin{equation*}
	\hat{\mu}_\epsilon^k(C_\epsilon^k(\rho r(\epsilon)+3\epsilon)) -q_{\epsilon,\delta}^k
	\gtrsim Leb(A^k_{2,\epsilon}) \left(r(\epsilon)^2 \ln \frac{r(\epsilon)}{\epsilon} \right)^{-1}
	\geq c \left( \ln \frac{r(\epsilon)}{\epsilon} \right)^{-1}
\end{equation*}
and thus
\begin{equation*}
	\hat{\mu}_\epsilon^k(C_\epsilon^k(\rho r(\epsilon)+3\epsilon)) \gtrsim \left( \ln \frac{r(\epsilon)}{\epsilon} \right)^{-1}.
\end{equation*}
Moreover, by definition \eqref{exit place q},
\begin{equation*}
	q_{\epsilon,\delta}^k \geq \inf_{ z \in C_\epsilon^k(\rho r(\epsilon)+3\epsilon)} \hat{\mathbb{P}}_z \left( \hat{Z}^\epsilon_k(\sigma_{k}^{5,\epsilon,\delta}) \in C_\epsilon^k(\delta) \right) 
	\hat{\mu}_\epsilon^k(C_\epsilon^k(\rho r(\epsilon)+3\epsilon))
\end{equation*}
and hence using \eqref{exit place ineq 1}, we have
\begin{equation}\label{exit place ineq fianl}
	\inf_{ z \in C_\epsilon^k(\rho r(\epsilon)+3\epsilon)} \hat{\mathbb{P}}_z \left( \hat{Z}^\epsilon_k(\sigma_{k}^{5,\epsilon,\delta}) \in C_\epsilon^k(\delta) \right) 
	\leq q_{\epsilon,\delta}^k \left( \hat{\mu}_\epsilon^k(C_\epsilon^k(\rho r(\epsilon)+3\epsilon)) \right)^{-1}
	\lesssim \frac{\epsilon}{\delta}  \ln \frac{r(\epsilon)}{\epsilon} \to 0
\end{equation}
as $\epsilon \to 0$.
The case $d=3$ can be treated similarly.

Finally, we claim that, there exists a constant $c>0$ independent of $\epsilon$, such that
\begin{equation}\label{exit place har}
	\sup_{ z \in C_\epsilon^k(\rho r(\epsilon)+3\epsilon)} \hat{\mathbb{P}}_z \left( \hat{Z}^\epsilon_k(\sigma_{k}^{5,\epsilon,\delta}) \in C_\epsilon^k(\delta) \right) 
	 \leq c \inf_{ z \in C_\epsilon^k(\rho r(\epsilon)+3\epsilon)} \hat{\mathbb{P}}_z \left( \hat{Z}^\epsilon_k(\sigma_{k}^{5,\epsilon,\delta}) \in C_\epsilon^k(\delta) \right) .
\end{equation}
Once \eqref{exit place har} is established, combining it with \eqref{exit place ineq fianl}, it yields
\begin{equation*}
\begin{aligned}
	&\lim_{\epsilon \to 0} \sup_{ z \in C_\epsilon^k(\rho r(\epsilon)+3\epsilon)} \mathbb{P}_z \left( Z^\epsilon(\sigma_{k}^{5,\epsilon,\delta}) \in C_\epsilon^k(\delta) \right) 	
	= \lim_{\epsilon \to 0}\sup_{ z \in C_\epsilon^k(\rho r(\epsilon)+3\epsilon)} \hat{\mathbb{P}}_z \left( \hat{Z}^\epsilon_k(\sigma_{k}^{5,\epsilon,\delta}) \in C_\epsilon^k(\delta) \right) \\
	 &\leq c \lim_{\epsilon \to 0} \inf_{ z \in C_\epsilon^k(\rho r(\epsilon)+3\epsilon)} \hat{\mathbb{P}}_z \left( \hat{Z}^\epsilon_k(\sigma_{k}^{5,\epsilon,\delta}) \in C_\epsilon^k(\delta) \right)
	 =0,
\end{aligned}
\end{equation*}
completing the argument.

Equation \eqref{exit place har} is a Harnack type inequality.
The main challenge in applying it here is that Harnack inequality is typically valid only for interior points, away from the boundary.
However, due to the Neumann boundary condition, we can introduce an even extension of the solution and apply the Harnack inequality in the extended domain.
More precisely, define 
$\hat{p}_{\epsilon,\delta}^k(z) = \hat{\mathbb{P}}_z \left( \hat{Z}^\epsilon_k(\sigma_{k}^{5,\epsilon,\delta}) \in C_\epsilon^k(\delta) \right) $ and let $\Omega_{8,\epsilon,\delta,var}^k$ be the region enclosed by $\Gamma_{4,\epsilon}^k$, $C_\epsilon^k(\delta)$ and $\partial G_\epsilon$.
It is known that $\hat{p}_{\epsilon,\delta}^k(z)$ satisfies the following boundary value problem
	\begin{equation*}
\begin{aligned}
    \begin{dcases}
    \displaystyle
        \Delta \hat{p}_{\epsilon,\delta}^k(z) = 0 \ \ \ \text{for} \ \ z \in \Omega_{8,\epsilon,\delta,var}^k, \\
        \hat{p}_{\epsilon,\delta}^k(z)=1 \ \ \ \text{for} \ \ z \in C_\epsilon^k(\delta), \\
        \hat{p}_{\epsilon,\delta}^k(z)=0 \ \ \ \text{for} \ \ z \in \Gamma_{4,\epsilon}^k, \\
        \frac{\partial \hat{p}_{\epsilon,\delta}^k(z)}{\partial \nu_{\epsilon,\delta}}=0 \ \ \  \text{for} \ \ z \in \partial \Omega_{8,\epsilon,\delta,var}^k \setminus \left( C_\epsilon^k(\delta) \cup \Gamma_{4,\epsilon}^k \right).
    \end{dcases}
\end{aligned}
\end{equation*}
The proof follows a strategy similar to that of Lemma \ref{lem reg har}.
Since we are only concerned with the behavior on $C_\epsilon^k(\rho r(\epsilon)+3\epsilon)$, we perform a local change of variables via $z \mapsto \frac{z-z_k^{**}}{\epsilon}$, where $z_k^{**} = C_\epsilon^k(\rho r(\epsilon)+3\epsilon) \cap I_k$.
In the case $d=2$, the key distinction is that no flattening map is needed; the reflection can be carried out directly.
The rest of the argument proceeds in the same way as in Lemma \ref{lem reg har}.
\end{proof}

\begin{proof}[Proof of Lemma \ref{exit place apriori}]
By the strong Markov property, for $z \in C_\epsilon^k(\rho r(\epsilon)+3\epsilon)$, we have
	\begin{equation}\label{st eq}
	\begin{aligned}
		& 
		\mathbb{P}_{z} \left( Z^\epsilon(\tilde{\tau}^{\epsilon,\delta'}) \in C_\epsilon^l(\delta') \right) \\
		& =   \mathbb{P}_{z} \left( Z^\epsilon(\tilde{\tau}^{\epsilon,\delta'}) \in C_\epsilon^l(\delta'), Z^\epsilon(\sigma_k^{5,\epsilon,\delta}) \in \Gamma_{4,\epsilon}^k  \right)  + \mathbb{P}_{z} \left( Z^\epsilon(\tilde{\tau}^{\epsilon,\delta'}) \in C_\epsilon^l(\delta'), Z^\epsilon(\sigma_k^{5,\epsilon,\delta}) \in C_\epsilon^k(\delta)  \right)  \\
		& = \mathbb{E}_{z} \left [ \mathbbm{1}_{ \{ Z^\epsilon(\sigma_k^{5,\epsilon,\delta}) \in \Gamma_{4,\epsilon}^k\}} 
		\mathbb{P}_{Z^\epsilon(\sigma_k^{5,\epsilon,\delta})} \left( Z^\epsilon(\tilde{\tau}^{\epsilon,\delta'}) \in C_\epsilon^l(\delta') \right) \right] \\
		& \qquad + \mathbb{P}_{z} \left( Z^\epsilon(\tilde{\tau}^{\epsilon,\delta'}) \in C_\epsilon^l(\delta'), Z^\epsilon(\sigma_k^{5,\epsilon,\delta}) \in C_\epsilon^k(\delta)  \right).
	\end{aligned}
	\end{equation}
	By Lemma \ref{exit place lem 1}, for all $\eta >0$ there exists $\epsilon_{\eta} >0$ such that for all $\epsilon \in (0, \epsilon_\eta)$ and all $k$,
	\begin{equation}\label{inf est}
	\begin{aligned}
		& \inf_{z \in C_\epsilon^k( \rho r(\epsilon) + 3\epsilon)} 
		\mathbb{P}_{z} \left( Z^\epsilon(\tilde{\tau}^{\epsilon,\delta'}) \in C_\epsilon^l(\delta') \right) \\
		& \geq \inf_{z \in C_\epsilon^k(\rho r(\epsilon) + 3\epsilon)}  \mathbb{E}_{z} \left [ \mathbbm{1}_{ \{ Z^\epsilon(\sigma_k^{5,\epsilon,\delta}) \in \Gamma_{4,\epsilon}^k\}} 
		\mathbb{P}_{Z^\epsilon(\sigma_k^{5,\epsilon,\delta})} \left( Z^\epsilon(\tilde{\tau}^{\epsilon,\delta'}) \in C_\epsilon^l(\delta')  \right)  \right]\\
		& \geq \inf_{z_1 \in \Gamma_{4,\epsilon}^k} \mathbb{P}_{z_1} \left( Z^\epsilon(\tilde{\tau}^{\epsilon,\delta'}) \in C_\epsilon^l(\delta') \right) 
		\inf_{z \in C_\epsilon^k( \rho r(\epsilon)+3\epsilon)}  \mathbb{P}_{z}  \left( Z^\epsilon(\sigma_k^{5,\epsilon,\delta}) \in \Gamma_{4,\epsilon}^k  \right) \\
		&\geq \inf_{z_1 \in \Gamma_{4,\epsilon}^k} \mathbb{P}_{z_1} \left( Z^\epsilon(\tilde{\tau}^{\epsilon,\delta'}) \in C_\epsilon^l(\delta') \right) -\eta. 
	\end{aligned}
	\end{equation}
	Moreover, using \eqref{st eq}, we have for all $\epsilon \in (0, \epsilon_\eta)$ and all $k$,
	\begin{equation}\label{sup est}
	\begin{aligned}
		&\sup_{z \in C_\epsilon^k( \rho r(\epsilon)+3\epsilon )} \mathbb{P}_{z} \left( Z^\epsilon(\tilde{\tau}^{\epsilon,\delta'}) \in C_\epsilon^l(\delta') \right)\\
		&\leq \sup_{z_1 \in \Gamma_{4,\epsilon}^k} \mathbb{P}_{z_1} \left( Z^\epsilon(\tilde{\tau}^{\epsilon,\delta'}) \in C_\epsilon^l(\delta') \right)
		 + \sup_{z \in C_\epsilon^k( \rho r(\epsilon)+3\epsilon )} \mathbb{P}_{z} \left( Z^\epsilon(\sigma_k^{5,\epsilon,\delta}) \in C_\epsilon^k(\delta)  \right)\\
		&\leq \sup_{z_1 \in \Gamma_{4,\epsilon}^k} \mathbb{P}_{z_1} \left( Z^\epsilon(\tilde{\tau}^{\epsilon,\delta'}) \in C_\epsilon^l(\delta') \right) +\eta .
	\end{aligned}
	\end{equation}
	
	Now for $z \in \Gamma_{4,\epsilon}$, we can apply a similar argument as in the proof of Lemma \ref{lem reg har}, to obtain that there exists a constant $c>0$ such that
	\begin{equation}\label{har est}
		\sup_{z \in \Gamma_{4,\epsilon}} \mathbb{P}_{z} \left( Z^\epsilon(\tilde{\tau}^{\epsilon,\delta'}) \in C_\epsilon^l(\delta') \right)
		\leq c \inf_{z \in \Gamma_{4,\epsilon}} \mathbb{P}_{z} \left( Z^\epsilon(\tilde{\tau}^{\epsilon,\delta'}) \in C_\epsilon^l(\delta') \right).
	\end{equation}
	To be more precise, we perform the change of variable $z \mapsto \frac{z-O}{r(\epsilon)}$ on $\Omega_{1,\epsilon}$.
	Let $\tilde{\Omega}_{6,\epsilon}$ and $\tilde{\Gamma}_{1,\epsilon}'$ denote the images of $\Omega_{6,\epsilon}$ and $\Gamma_{1,\epsilon}$ under this transformation, respectively.
	By construction, there exists $\rho' < \tilde{\rho}$ such that $d(\tilde{\Omega}_{6,\epsilon},\tilde{\Gamma}_{1,\epsilon}') \geq \rho' $. 
	We then apply the same argument as in Lemma \ref{lem reg har}, using balls of radius less than $\frac{\rho'}{2}$ to cover $\tilde{\Omega}_{6,\epsilon}$.
	
	Combining \eqref{inf est}, \eqref{sup est}, and \eqref{har est}, we obtain
	\begin{equation*}
	\begin{aligned}
	&\inf_{z \in C_\epsilon(\rho r(\epsilon)+3\epsilon )} 
		\mathbb{P}_{z} \left( Z^\epsilon(\tilde{\tau}^{\epsilon,\delta'}) \in C_\epsilon^l(\delta') \right)  
		\geq \inf_{z \in \Gamma_{4,\epsilon}} \mathbb{P}_{z} \left( Z^\epsilon(\tilde{\tau}^{\epsilon,\delta'}) \in C_\epsilon^l(\delta') \right) -\eta \\
		& \geq c^{-1} \sup_{z \in \Gamma_{4,\epsilon}} \mathbb{P}_{z} \left( Z^\epsilon(\tilde{\tau}^{\epsilon,\delta'}) \in C_\epsilon^l(\delta') \right) - \eta  \\
		& \geq c^{-1}  \sup_{z \in C_\epsilon(\rho r(\epsilon)+3\epsilon )} \mathbb{P}_{z} \left( Z^\epsilon(\tilde{\tau}^{\epsilon,\delta'}) \in C_\epsilon^l(\delta') \right) - (c^{-1}+1) \eta .
	\end{aligned}
	\end{equation*}	
	Applying \eqref{exit place sup lb} with the substitution of $\sigma^{\epsilon,\delta}$ by $\tilde{\tau}^{\epsilon,\delta'}$ (noting that both represent exit problems with neck lengths $\delta$ and $\delta'$ respectively), due to the arbitrariness of $\eta >0$, we conclude that there exists $c>0$ such that
	\begin{equation*}
		\lim_{\delta' \to 0} \lim_{\epsilon \to 0} \inf_{z \in C_\epsilon(\rho r(\epsilon)+3\epsilon )} 
		\mathbb{P}_{z} \left( Z^\epsilon(\tilde{\tau}^{\epsilon,\delta'}) \in C_\epsilon^l(\delta') \right)
		\geq c.
	\end{equation*}
\end{proof}

\section{Quasi-stationary distributions}\label{sec qsd}
This section is devoted to establishing the convergence of the discrete Markov chain $X_n^{\epsilon,\delta',\delta}$ to its quasi-stationary distribution, conditioned on survival.
As our focus remains on the local behavior near a vertex, we omit the subscript $j$ throughout this section.

Let $\delta' \ll_{\epsilon,\delta',\delta } \delta$.
We consider the sequence of stopping times
\begin{equation*}
	\tilde{\sigma}_n^{\epsilon,\delta'} \coloneqq \inf \{ t \geq \tilde{\tau}_n^{\epsilon,\delta'} : Z^\epsilon(t) \in C_\epsilon(\rho r(\epsilon)+3\epsilon) \cup C_\epsilon(2\delta') \}
\end{equation*}
and
\begin{equation*}
	\tilde{\tau}_n^{\epsilon,\delta'} \coloneqq \inf \{ t \geq \tilde{\sigma}_{n-1}^{\epsilon,\delta'} : Z^\epsilon(t) \in C_\epsilon(\delta') \}
\end{equation*}
with 
$\tilde{\tau}_0^{\epsilon,\delta'} \coloneqq 0$.
Also, let $\tilde{\sigma}^{\epsilon,\delta'}$ be the first hitting time to $C_\epsilon(\rho r(\epsilon)+3\epsilon) \cup C_\epsilon(2\delta')$.
Recall that
$\sigma^{\epsilon,\delta} $ is the first hitting time to $C_\epsilon(\delta)$ and 
$\tilde{\tau}^{\epsilon,\delta'}$ is the first hitting time to $C_\epsilon(\delta')$.

We now define a discrete time Markov process $X_n^{\epsilon,\delta',\delta}$ with state space $C_\epsilon(\delta') \cup \{ \partial \}$ by
\begin{equation*}
	X_n^{\epsilon,\delta',\delta} \coloneqq
	\begin{dcases}
		Z^\epsilon (\tilde{\tau}_n^{\epsilon,\delta'}), \ \  \text{ if } \tilde{\tau}_n^{\epsilon,\delta'} < \sigma^{\epsilon,\delta}, \\
		\partial , \ \  \text{ if } \tilde{\tau}_m^{\epsilon,\delta'} \geq \sigma^{\epsilon,\delta}, \text{ for some } m \leq n.
	\end{dcases}
\end{equation*}
That is, the process $X_n^{\epsilon,\delta',\delta}$ is killed (sent to the absorbing state $\partial$) upon the first time $Z^\epsilon$ reaches $C_\epsilon(\delta)$.

We now introduce the quasi-stationary distribution associated with the process $X_n^{\epsilon,\delta',\delta}$ .

\begin{Definition}
	A probability measure $\alpha_{\epsilon,\delta',\delta}$ is a quasi-stationary distribution for the process $X_n^{\epsilon,\delta',\delta}$ if for all Borel measurable subset $A_{\epsilon,\delta'} \subset C_\epsilon(\delta')$ and $n \geq 0$, we have
	\begin{equation*}
		\mathbb{P}_{\alpha_{\epsilon,\delta',\delta}} \left( X_n^{\epsilon,\delta',\delta} \in A_{\epsilon,\delta'} \middle|  \tilde{\tau}_n^{\epsilon,\delta'} < \sigma^{\epsilon,\delta} \right) 
		= \alpha_{\epsilon,\delta',\delta}(A_{\epsilon,\delta'}).
	\end{equation*}
\end{Definition}

\subsection{Uniform convergence to the quasi-stationary distribution}\label{sec conv qsd}
In this subsection, we aim to prove uniform convergence to the unique quasi-stationary distribution for the process $X_n^{\epsilon,\delta,\delta'}$.
\begin{Theorem}\label{conv qsd}
	There exists a unique quasi-stationary distribution $\alpha_{\epsilon,\delta',\delta}$ for the process $X_n^{\epsilon,\delta',\delta}$.
	Moreover, for every $\eta > 0$, there exists $N_{\eta} \in \mathbb{N}$ and $\delta_\eta' > 0$ such that for all $\delta' \in (0,\delta_\eta')$ there exists $\epsilon_{\eta,\delta'}>0$ such that
	\begin{equation*}
		\sup_{z \in C_\epsilon(\delta')} \left \lVert 
		\mathbb{P}_{z} \left( X_{N_\eta}^{\epsilon,\delta',\delta} \in \cdot \middle|  \tilde{\tau}_{N_\eta}^{\epsilon,\delta'} < \sigma^{\epsilon,\delta} \right) 
		-\alpha_{\epsilon,\delta',\delta}(\cdot) \right\rVert_{TV}
		< \eta, \ \ \ \epsilon \in (0,\epsilon_{\eta,\delta'}).
	\end{equation*}
\end{Theorem}

	The proof of Theorem \ref{conv qsd} is largely based on an argument similar to that of \cite{QSD}, Theorem 2.1, where the so-called Dobrushin's condition is used.
To this purpose, we introduce the following uniform Dobrushin's condition tailored to our setting.

	There exists a probability measure $\nu_{\epsilon,\delta'}$ on $C_{\epsilon}(\delta')$ such that
	\begin{claim} \label{itm: c1}
	For all $\eta>0$, there exist $c_1>0$, $n_0 \geq 0$, and $\delta_{\eta}'>0$ such that for all $\delta' \in (0,\delta_{\eta}')$, there exists $\epsilon_{\eta,\delta'} > 0$ such that 
	\begin{equation*}
		\inf_{z \in C_{\epsilon}(\delta')} \mathbb{P}_{z} \left( X_{n_0}^{\epsilon,\delta',\delta} \in \cdot \middle|  \tilde{\tau}_{n_0}^{\epsilon,\delta'} < \sigma^{\epsilon,\delta} \right) \geq c_1 \nu_{\epsilon,\delta'}(\cdot) - \eta, \ \ \ \epsilon \in (0,\epsilon_{\eta,\delta'}).
	\end{equation*}
	\end{claim}
	\begin{claim} \label{itm: c2}
	There exists $c_2>0$ such that for all $\epsilon,\delta'$ small enough, and all $n \geq 0$, we have
	\begin{equation*}
		\mathbb{P}_{\nu_{\epsilon,\delta'}} \left(   \tilde{\tau}_n^{\epsilon,\delta'} < \sigma^{\epsilon,\delta} \right) \geq c_2 \sup_{z \in C_{\epsilon}(\delta')} \mathbb{P}_{z} \left(   \tilde{\tau}_n^{\epsilon,\delta'} < \sigma^{\epsilon,\delta} \right).
	\end{equation*}
	\end{claim}

In Appendix \ref{app conv qsd}, we will show how Claim \ref{itm: c1} and \ref{itm: c2} together imply Theorem \ref{conv qsd}, by following an argument similar to that of \cite{QSD}, Theorem 2.1.
In what follows, we prove that Claims \ref{itm: c1} and \ref{itm: c2} hold with $\nu_{\epsilon,\delta'}$ being the uniform distribution on $C_\epsilon(\delta')$ and $n_0=2$. 

\begin{proof}[Proof of Claim \ref{itm: c2}]
Let $z \in C_\epsilon(\delta')$ and $p_{\epsilon,\delta',\delta}^n(z) \coloneqq \mathbb{P}_z \left(\tilde{\tau}_n^{\epsilon,\delta'} < \sigma^{\epsilon,\delta}  \right)$ for $n \geq 1$.
Since it is only related to a one dimensional Brownian motion, $p_{\epsilon,\delta',\delta}^n(z)$ is independent of $z \in C_\epsilon(\delta')$; in fact,
\begin{equation}\label{est p}
	p_{\epsilon,\delta',\delta}^1(z) = \frac{\delta'}{2\delta'-\rho r(\epsilon)-3\epsilon} + \frac{\delta'-\rho r(\epsilon)-3\epsilon}{2\delta'-\rho r(\epsilon)-3\epsilon} \cdot \frac{\delta-2 \delta'}{\delta - \delta'}
\end{equation}
and $p_{\epsilon,\delta',\delta}^n(z)=\left(p_{\epsilon,\delta',\delta}^1(z) \right)^n$.
The result then follows from the fact that $\nu_{\epsilon,\delta'}$ is supported on $C_\epsilon(\delta')$.
\end{proof}

\begin{proof}[Proof of Claim \ref{itm: c1}]
It suffices to consider two cases: $z \in C_\epsilon^k(\delta')$, $A_{\epsilon,\delta'}^l \subset C_\epsilon^l(\delta')$ for $k=l$ or $k \neq l$.
In addition, by \eqref{est p}, it suffices to show that
for all $\eta>0$, there exist $c_1'>0$, $n_0 \geq 0$, and $\delta_{\eta}'>0$ such that for all $\delta' \in (0,\delta_{\eta}')$, there exists $\epsilon_{\eta,\delta'} > 0$ such that 
\begin{equation*}
	\inf_{z \in C_{\epsilon}^k(\delta')} \mathbb{P}_{z} \left( X_{n_0}^{\epsilon,\delta',\delta} \in A_{\epsilon,\delta'}^l,  \tilde{\tau}_{n_0}^{\epsilon,\delta'} < \sigma^{\epsilon,\delta} \right) \geq c_1' \nu_{\epsilon,\delta'}(A_{\epsilon,\delta'}^l) - \eta, \ \ \ \epsilon \in (0,\epsilon_{\eta,\delta'}).
\end{equation*}

Let us first assume $k=l$ and $n_0=1$. By the strong Markov property,
\begin{equation*}
\begin{aligned}
	&\inf_{ z \in C_\epsilon^k(\delta') } \mathbb{P}_{z} \left(X_1^{\epsilon,\delta',\delta} \in A_{\epsilon,\delta'}^k, \tilde{\tau}_1^{\epsilon,\delta'} < \sigma^{\epsilon,\delta} \right) \\
	& \geq \inf_{ z \in C_\epsilon^k(\delta') } \mathbb{P}_{z} \left(X_1^{\epsilon,\delta',\delta} \in A_{\epsilon,\delta'}^k, Z^\epsilon(\tilde{\sigma}_0^{\epsilon,\delta'}) \in C_\epsilon^k(2\delta') , \tilde{\tau}_1^{\epsilon,\delta'} < \sigma^{\epsilon,\delta} \right) \\
	&= \inf_{ z \in C_\epsilon^k(\delta') } \mathbb{E}_{z} \left[ \mathbbm{1}_{ \{ Z^\epsilon(\tilde{\sigma}_0^{\epsilon,\delta'}) \in C_\epsilon^k(2\delta') \} }
	\mathbb{P}_{Z^\epsilon(\tilde{\sigma}_0^{\epsilon,\delta'})} \left( Z^\epsilon(\tilde{\tau}^{\epsilon,\delta'} ) \in A_{\epsilon,\delta'}^k ,  \tilde{\tau}^{\epsilon,\delta'} < \sigma^{\epsilon,\delta} \right)	
	 \right] \\
	 & \geq   \inf_{ z \in C_\epsilon^k(\delta') } \mathbb{P}_{z} \left( Z^\epsilon(\tilde{\sigma}_0^{\epsilon,\delta'}) \in C_\epsilon^k(2\delta') \right) \nu_{\epsilon,\delta'}^k(A_{\epsilon,\delta'}^k)  \\
	 & \ \ \ -
	 \sup_{ z \in C_\epsilon^k(\delta') } \abs{
	 \mathbb{E}_{z} \left[ \mathbbm{1}_{ \{ Z^\epsilon(\tilde{\sigma}_0^{\epsilon,\delta'}) \in C_\epsilon^k(2\delta') \} }  \left(
	 \mathbb{P}_{Z^\epsilon(\tilde{\sigma}_0^{\epsilon,\delta'})} \left( Z^\epsilon(\tilde{\tau}^{\epsilon,\delta'} ) \in A_{\epsilon,\delta'}^k , \tilde{\tau}^{\epsilon,\delta'} < \sigma^{\epsilon,\delta} \right)	
	- \nu_{\epsilon,\delta'}^k(A_{\epsilon,\delta'}^k) \right)
	 \right] },
\end{aligned}
\end{equation*}
where $\nu_{\epsilon,\delta'}^k$ is the uniform distribution of $C_\epsilon^k(\delta')$.
The term $\mathbb{P}_{z} \left( Z^\epsilon(\tilde{\sigma}_0^{\epsilon,\delta'}) \in C_\epsilon^k(2\delta') \right)$ converges to $\frac{1}{2}$, as $\epsilon \to 0$ uniformly for $z \in C_\epsilon^k(\delta')$.
As a result, it is greater than $\frac{1}{4}$, for $\epsilon$ small enough.
Moreover, since, for $A_{\epsilon,\delta'}^k \subset C_\epsilon^k(\delta')$,
\begin{equation*}
	\nu_{\epsilon,\delta'}^k(A_{\epsilon,\delta'}^k) 
	= \left( \frac{\lambda_k^{d-1}}{\sum_{l:I_l \sim O} \lambda_l^{d-1}} \right)^{-1} \nu_{\epsilon,\delta'} (A_{\epsilon,\delta'}^k),
\end{equation*}
the proof is done if we can show that the second term converges to $0$.

For all $t>0$ and $z \in C_\epsilon^k(\delta')$,
\begin{equation}\label{c1 eq 2}
\begin{aligned}
	&  \abs{ \mathbb{E}_{z} \left[ \mathbbm{1}_{ \{ Z^\epsilon(\tilde{\sigma}_0^{\epsilon,\delta'}) \in C_\epsilon^k(2\delta') \} }
	 \left(
	 \mathbb{P}_{Z^\epsilon(\tilde{\sigma}_0^{\epsilon,\delta'})} \left( Z^\epsilon(\tilde{\tau}^{\epsilon,\delta'} ) \in A_{\epsilon,\delta'}^k , \tilde{\tau}^{\epsilon,\delta'} < \sigma^{\epsilon,\delta} \right)	
	- \nu_{\epsilon,\delta'}^k(A_{\epsilon,\delta'}^k) \right)
	 \right]}  \\
	 & \leq  \abs{ \mathbb{E}_{z} \left[ \mathbbm{1}_{ \{ Z^\epsilon(\tilde{\sigma}_0^{\epsilon,\delta'}) \in C_\epsilon^k(2\delta') \} } 
	  \left(
	\mathbb{P}_{Z^\epsilon(\tilde{\sigma}_0^{\epsilon,\delta'})} \left( Z^\epsilon(\tilde{\tau}^{\epsilon,\delta'} ) \in A_{\epsilon,\delta'}^k , t < \tilde{\tau}^{\epsilon,\delta'} < \sigma^{\epsilon,\delta} \right)	
	- \nu_{\epsilon,\delta'}^k(A_{\epsilon,\delta'}^k) \right)
	 \right] } \\
	 & \ \ \ +  \abs{\mathbb{E}_{z} \left[ \mathbbm{1}_{ \{ Z^\epsilon(\tilde{\sigma}_0^{\epsilon,\delta'}) \in C_\epsilon^k(2\delta') \} }
	\mathbb{P}_{Z^\epsilon(\tilde{\sigma}_0^{\epsilon,\delta'})} \left( Z^\epsilon(\tilde{\tau}^{\epsilon,\delta'} ) \in A_{\epsilon,\delta'}^k , t \geq \tilde{\tau}^{\epsilon,\delta'} , \tilde{\tau}^{\epsilon,\delta'} < \sigma^{\epsilon,\delta} \right)	
	 \right]} \\
	 & \leq \sup_{z_1 \in C_\epsilon^k(2 \delta')}
	\abs{ \mathbb{P}_{z_1} \left(Z^\epsilon(\tilde{\tau}^{\epsilon,\delta'} ) \in A_{\epsilon,\delta'}^k, t < \tilde{\tau}^{\epsilon,\delta'} < \sigma^{\epsilon,\delta} \right) 
	-\nu_{\epsilon,\delta'}^k(A_{\epsilon,\delta'}^k) } \\
	 & \qquad +
	\sup_{z_1 \in C_\epsilon^k(2 \delta')} 
	\mathbb{P}_{z_1} \left( t \geq \tilde{\tau}^{\epsilon,\delta'}, \tilde{\tau}^{\epsilon,\delta'} < \sigma^{\epsilon,\delta} \right)  \\
	& \ \ \ \  \eqqcolon I_1^{\epsilon,\delta',\delta}(t) +I_2^{\epsilon,\delta',\delta}(t).
\end{aligned}
\end{equation}

We claim that there exists $t_{\eta,\delta'} > 0$ such that 
\begin{equation}\label{c1 eq 2-2}
\begin{aligned}
\sup_{z_1 \in C_\epsilon^k(2 \delta')} 
	\mathbb{P}_{z_1} \left( \tilde{\tau}^{\epsilon,\delta'} \leq t_{\eta,\delta'}, \tilde{\tau}^{\epsilon,\delta'} < \sigma^{\epsilon,\delta} \right)
	\leq \sup_{z_1 \in C_\epsilon^k(2 \delta')} 
	\mathbb{P}_{z_1} \left( \tilde{\tau}^{\epsilon,\delta'} \wedge \sigma^{\epsilon,\delta} \leq t_{\eta,\delta'}  \right)
	< \frac{\eta}{4} .
\end{aligned}
\end{equation}
Let $B_1(t)$ be a one dimensional Brownian motion starting at $0$, and define the hitting time
\begin{equation*}
T_{\delta'} = \inf\{ t \geq 0: B_1(t)=\delta' \}.
\end{equation*}
Let $\mathbf{P}$ denote the associated probability measure.
Since $Z^\epsilon(t)$ is the scaled Brownian motion, we have 
\begin{equation*}
\sup_{z_1 \in C_\epsilon^k(2 \delta')} 
	\mathbb{P}_{z_1} \left( \tilde{\tau}^{\epsilon,\delta'} \wedge \sigma^{\epsilon,\delta} \leq t_{\eta,\delta'}  \right)
	\leq 
\mathbf{P} \left( T_{\delta'} \wedge T_{-\delta'} \leq 2 t_{\eta,\delta'} \right).
\end{equation*}
By the reflection principle and Mill's inequality,
\begin{equation*}
\begin{aligned}
	&\mathbf{P} \left( T_{\delta'} \wedge T_{-\delta'} \leq t \right)
	\leq \mathbf{P} \left( T_{\delta'}  \leq t \right)
	+ \mathbf{P} \left( T_{-\delta'} \leq t \right)
	= 2\mathbf{P} \left( T_{\delta'}  \leq t \right) \\
	&= 2 \mathbf{P} \left( \sup_{0 \leq s \leq t} B_1(s) \geq \delta' \right)
	= 4 	\mathbf{P} \left(  B_1(t) \geq \delta' \right)
	\leq \frac{2\sqrt{2t}}{\delta'\sqrt{\pi}} e^{-\frac{(\delta')^2}{2t}}.
\end{aligned}
\end{equation*}
Since the right hand side tends to zero as $t \to 0$, we choose $t_{\eta,\delta'} > 0$ small enough so that \eqref{c1 eq 2-2} holds.
Substituting $t=t_{\eta,\delta'}$ into \eqref{c1 eq 2}, we obtain
\begin{equation}\label{c1 ineq 1}
	I_2^{\epsilon,\delta',\delta}(t_{\eta,\delta'})  < \frac{\eta}{4} .
\end{equation}

Next, we consider $I_1^{\epsilon,\delta',\delta}$. Since we are working within a cylindrical domain, we can apply separation of variables.
Specifically, we write $Z^\epsilon(t) = (X^\epsilon(t),Y^\epsilon(t))$, where $X^\epsilon(t) \in I_k$ and $Y^\epsilon(t) \in I_k^\perp$. 
Moreover, the processes $X^\epsilon(t)$ and $Y^\epsilon(t)$ are independent.
Since $A_{\epsilon,\delta'}^{k} \subset C_\epsilon(\delta')$, all points in $A_{\epsilon,\delta'}^{k} \subset C_\epsilon(\delta')$ share the same $x$-coordinate. 
We denote by $A_{\epsilon,\delta'}^{k,y}$ the projection of $A_{\epsilon,\delta'}^{k}$ onto the transverse direction $I_k^\perp$, i.e., its cross sectional component in the $y$-variables.
Let $\nu_{\epsilon,\delta'}^{k,y}$ be the uniform distribution on $I_k^\perp$, then we have
\begin{equation*}
	\nu_{\epsilon,\delta'}^k (A_{\epsilon,\delta'}^k)
	= \nu_{\epsilon,\delta'}^{k,y} (A_{\epsilon,\delta'}^{k,y}).
\end{equation*}
Finally, letting $z_1 = (x_1,y_1) \in C_\epsilon^k(2\delta')$, we have
\begin{equation*}
\begin{aligned}
	& \mathbb{P}_{z_1} \left(Z^\epsilon(\tilde{\tau}^{\epsilon,\delta'}) \in A_{\epsilon,\delta'}^k, t< \tilde{\tau}^{\epsilon,\delta'} < \sigma^{\epsilon,\delta} \right) \\
	& = \int \int  \mathbbm{1}_{ \{ Y^\epsilon(\omega, \tilde{\tau}^{\epsilon,\delta'}(\omega')) \in A_{\epsilon,\delta'}^{k,y} \}} \mathbbm{1}_{ \{ Y^\epsilon(\omega, 0) = y_1 \}}  
	\mathbbm{1}_{ \{ t < \tilde{\tau}^{\epsilon,\delta'}(\omega') < \sigma^{\epsilon,\delta}(\omega') \} }
	\mathbbm{1}_{ \{ X^\epsilon(\omega', 0) = x_1 \}}
	d \mathbb{P}^Y(\omega) d \mathbb{P}^X(\omega') \\
	& = \int \mathbb{P}_{y_1}^Y \left( Y^\epsilon( \tilde{\tau}^{\epsilon,\delta'}(\omega')) \in A_{\epsilon,\delta'}^{k,y} \right)
	\mathbbm{1}_{ \{ t < \tilde{\tau}^{\epsilon,\delta'}(\omega') < \sigma^{\epsilon,\delta}(\omega') \} }
	\mathbbm{1}_{ \{ X^\epsilon(\omega', 0) = x_1 \}}
	d \mathbb{P}^X(\omega') \\
	& = \int \left[ \mathbb{P}_{y_1}^Y \left( Y^\epsilon( \tilde{\tau}^{\epsilon,\delta'}(\omega')) \in A_{\epsilon,\delta'}^{k,y} \right) - \nu_{\epsilon,\delta'}^{k,y} (A_{\epsilon,\delta'}^{k,y}) \right] 
	 \mathbbm{1}_{ \{ t < \tilde{\tau}^{\epsilon,\delta'}(\omega') < \sigma^{\epsilon,\delta}(\omega') \} }
	\mathbbm{1}_{ \{ X^\epsilon(\omega', 0) = x_1 \}}
	d \mathbb{P}^X(\omega') \\
	& \ \ \ \ \ + \nu_{\epsilon,\delta'}^k (A_{\epsilon,\delta'}^k) \mathbb{P}_{z_1} \left( t < \tilde{\tau}^{\epsilon,\delta'} < \sigma^{\epsilon,\delta} \right).
\end{aligned}
\end{equation*}
Therefore,
\begin{equation*}
\begin{aligned}
&I_1^{\epsilon,\delta',\delta}(t_{\eta,\delta'}) = 
\sup_{z_1 \in C_\epsilon^k(2 \delta')}
	\abs{ \mathbb{P}_{z_1} \left(Z^\epsilon(\tilde{\tau}^{\epsilon,\delta'} ) \in A_{\epsilon,\delta'}^{k}, t_{\eta,\delta'} < \tilde{\tau}^{\epsilon,\delta'} < \sigma^{\epsilon,\delta} \right) 
	-\nu_{\epsilon,\delta'}^k(A_{\epsilon,\delta'}^{k}) } \\
	&\leq \sup_{z_1 \in C_\epsilon^k(2 \delta')}
	\Bigg| 
	\int \left[ \mathbb{P}_{y_1}^Y \left( Y^\epsilon( \tilde{\tau}^{\epsilon,\delta'}(\omega')) \in A_{\epsilon,\delta'}^{k,y} \right) - \nu_{\epsilon,\delta'}^{k,y} (A_{\epsilon,\delta'}^{k,y}) \right] \\
	& \qquad \times \mathbbm{1}_{ \{ t_{\eta,\delta'} < \tilde{\tau}^{\epsilon,\delta'}(\omega') < \sigma^{\epsilon,\delta}(\omega') \} }
	\mathbbm{1}_{ \{ X^\epsilon(\omega', 0) = x_1 \}}
	d \mathbb{P}^X(\omega') \Bigg| \\
	& \ \ \ \ \ +
	\sup_{z_1 \in C_\epsilon^k(2 \delta')}\abs{ \mathbb{P}_{z_1} \left( t_{\eta,\delta'} < \tilde{\tau}^{\epsilon,\delta'} < \sigma^{\epsilon,\delta} \right)  - 1 }	\nu_{\epsilon,\delta'}^k (A_{\epsilon,\delta'}^k) \\
	& \leq \sup_{z_1 \in C_\epsilon^k(2 \delta')}
	\int \sup_{y_1 \in B^{d-1}(0,\lambda_k \epsilon)} \left \lVert \mathbb{P}_{y_1}^Y \left( Y^\epsilon( \tilde{\tau}^{\epsilon,\delta'}(\omega')) \in \cdot \right) - \nu_{\epsilon,\delta'}^{k,y} (\cdot) \right \rVert_{TV} \\
	& \ \ \ \ \ \ \ \  \ \ \ \ \ \ \ \ \ \ \ \ \  \ \ \ \ \ \ \ \ \ \ \ \ \ \ \  \ \ \ \ \ \ \ \ \ \ \ \ \times
	\mathbbm{1}_{ \{ t_{\eta,\delta'} < \tilde{\tau}^{\epsilon,\delta'}(\omega') < \sigma^{\epsilon,\delta}(\omega') \} }
	\mathbbm{1}_{ \{ X^\epsilon(\omega', 0) = x_1 \}}
	d \mathbb{P}^X(\omega')  \\
	& \ \ \ \ \ +
	\sup_{z_1 \in C_\epsilon^k(2 \delta')} \mathbb{P}_{z_1} \left( t_{\eta,\delta'} \geq \tilde{\tau}^{\epsilon,\delta'} \wedge \sigma^{\epsilon,\delta} \right)  
	+ \sup_{z_1 \in C_\epsilon^k(2 \delta')} \mathbb{P}_{z_1} \left( \tilde{\tau}^{\epsilon,\delta'} > \sigma^{\epsilon,\delta} \right)    \\	
	& \ \ \ \ \ \ \ \ \ \ \ \ \ \ \ \eqqcolon J_{1,1}^{\epsilon,\delta',\delta} + J_{1,2}^{\epsilon,\delta',\delta} + J_{1,3}^{\epsilon,\delta',\delta}.
\end{aligned}
\end{equation*}

For $J_{1,2}^{\epsilon,\delta',\delta}$ and $J_{1,3}^{\epsilon,\delta',\delta}$, by \eqref{c1 eq 2-2} and the fact that for all $z_1 \in C_\epsilon^k(2 \delta')$
\begin{equation*}
\mathbb{P}_{z_1} \left( \tilde{\tau}^{\epsilon,\delta'} > \sigma^{\epsilon,\delta} \right)  
= \frac{\delta'}{\delta-\delta'},
\end{equation*} 
we obtain, for $\delta'$ small enough,
\begin{equation}\label{c1 ineq 2}
J_{1,2}^{\epsilon,\delta',\delta}+J_{1,3}^{\epsilon,\delta',\delta} <  \frac{\eta}{2}.
\end{equation}

For $J_{1,1}^{\epsilon,\delta',\delta}$, note that $Y^\epsilon$ is a scaled reflected Brownian motion on a $(d-1)$-dimensional ball with radius $\lambda_k \epsilon$.
We have the following uniform ergodicity estimate (\cite{UEB}, see also \cite{BCM})
\begin{equation*}
	\sup_{y \in B^{d-1}(0,\lambda_k \epsilon)} \left\lVert \mathbb{P}_{y}^Y \left(Y^\epsilon(t) \in \cdot \right) - \nu_{\epsilon,\delta'}^{k,y}(\cdot) \right\rVert_{TV} \leq F(\lambda_k \epsilon,8t,0),
\end{equation*}
where
\begin{equation*}
	F(r,t,0) = \sum_{n=0}^\infty e^{-\frac{\pi^2}{8r^2}(2n+1)^2t} \frac{4 (-1)^n}{\pi(2n+1)}.
\end{equation*}
As a result, under the event $\{ t_{\eta,\delta'} < \tilde{\tau}^{\epsilon,\delta'} \}$,
\begin{equation*}
	\sup_{y \in B^{d-1}(0,\lambda_k \epsilon)} \left\lVert \mathbb{P}_y^Y \left(Y^\epsilon(\tilde{\tau}^{\epsilon,\delta'}) \in \cdot \right) - \nu_{\epsilon,\delta'}^{k,y}(\cdot) \right\rVert_{TV} \leq F(\lambda_k \epsilon,8\tilde{\tau}^{\epsilon,\delta'},0)
	\leq F(\lambda_k \epsilon, 8t_{\eta,\delta'}, 0).
\end{equation*}
It is clear that, there exist some constants $c, c'>0$ such that
\begin{equation*}
	\abs{ e^{-\frac{\pi^2}{8\lambda_k^2\epsilon^2}(2n+1)^2 8t_{\eta,\delta'}} \frac{4 (-1)^n}{\pi(2n+1)} }
	\leq c e^{-c' n^2},
\end{equation*}
which is summable.
By Lebesgue dominated convergence theorem, we have
\begin{equation*}
\lim_{\epsilon \to 0} F(\lambda_k \epsilon, 8t_{\eta,\delta'}, 0) = 0,
\end{equation*}
which means that for every $\eta >0$ there exists $\epsilon_{\eta,\delta'} >0$ such that for all $\epsilon \in (0, \epsilon_{\eta,\delta'})$, we have
\begin{equation}\label{c1 ineq 3}
\begin{aligned}
	& J_{1,1}^{\epsilon,\delta',\delta} 
	=\sup_{z_1 \in C_\epsilon(2\delta')} \int \sup_{y_1 \in B^{d-1}(0,\lambda_k \epsilon)} \left \lVert \mathbb{P}_{y_1} \left( Y^\epsilon( \tilde{\tau}^{\epsilon,\delta'}(\omega')) \in \cdot \right) - \nu_{\epsilon,\delta'}^{k,y} (\cdot) \right \rVert_{TV} \\
	& \ \ \ \ \ \ \ \  \ \ \ \ \ \ \ \ \ \ \ \ \  \ \ \ \ \ \ \ \ \ \ \ \ \ \ \  \ \ \ \ \ \ \ \ \ \ \ \ \  \ \times 
	\mathbbm{1}_{ \{ t_{\eta,\delta'} < \tilde{\tau}^{\epsilon,\delta'}(\omega') < \sigma^{\epsilon,\delta}(\omega') \} }
	\mathbbm{1}_{ \{ X^\epsilon(\omega', 0) = x_1 \}}
	d \mathbb{P}^X(\omega')  < \frac{\eta}{4}.
\end{aligned}
\end{equation}
Combining \eqref{c1 ineq 1}, \eqref{c1 ineq 2} and \eqref{c1 ineq 3}, the proof is complete in the case $k=l$ and $n_0=1$.

Finally, for every $k, l$ and $n_0=2$.
By the strong Markov property,
\begin{equation*}
\begin{aligned}
	&\inf_{z \in C_\epsilon^k(\delta')} \mathbb{P}_{z} \left(X_2^{\epsilon,\delta',\delta} \in A_{\epsilon,\delta'}^l, \tilde{\tau}_2^{\epsilon,\delta'} < \sigma^{\epsilon,\delta} \right) \\
	&\geq  \inf_{z \in C_\epsilon^k(\delta')} \mathbb{P}_{z} \left(X_2^{\epsilon,\delta',\delta} \in A_{\epsilon,\delta'}^l,   \tilde{\tau}_2^{\epsilon,\delta'} < \sigma^{\epsilon,\delta}, X_1^{\epsilon,\delta',\delta} \in C_\epsilon^l(\delta') \right) \\
	 & = \inf_{z \in C_\epsilon^k(\delta')} \mathbb{E}_{z} \left( \mathbbm{1}_{ \{ X_1^{\epsilon,\delta',\delta} \in C_\epsilon^l(\delta') \} } \mathbb{P}_{X_1^{\epsilon,\delta',\delta}} \left( X_1^{\epsilon,\delta',\delta} \in A_{\epsilon,\delta'}^l,   \tilde{\tau}_1^{\epsilon,\delta'} < \sigma^{\epsilon,\delta} \right)
	 \right)  \\
	 & \geq \inf_{z' \in C_\epsilon^l(\delta')} \mathbb{P}_{z'} \left( X_1^{\epsilon,\delta',\delta} \in A_{\epsilon,\delta'}^l,   \tilde{\tau}_1^{\epsilon,\delta'} < \sigma^{\epsilon,\delta} \right) 
	 \inf_{z \in C_\epsilon^k(\delta')} \mathbb{P}_{z} \left( X_1^{\epsilon,\delta',\delta} \in C_\epsilon^l(\delta')  \right).
\end{aligned}
\end{equation*}
Note that by Lemma \ref{exit place apriori}, there exists $c>0$ such that for every $\delta'>0$ small enough there exists $\epsilon_{\delta'}>0$ such that for every $\epsilon \in (0,\epsilon_{\delta'})$, we have
\begin{equation*}
\begin{aligned}
&\inf_{z \in C_\epsilon^k(\delta')} \mathbb{P}_{z} \left( X_1^{\epsilon,\delta',\delta} \in C_\epsilon^l(\delta')  \right)
=\inf_{z \in C_\epsilon^k(\delta')} \mathbb{P}_{z} \left( X_1^{\epsilon,\delta',\delta} \in C_\epsilon^l(\delta'), Z^\epsilon(\tilde{\sigma}_0^{\epsilon,\delta'}) \in C_\epsilon^k(\rho r(\epsilon)+ 3\epsilon)  \right)\\
& = \inf_{z \in C_\epsilon^k(\delta')} \mathbb{E}_{z} \left[ \mathbbm{1}_{ \{ Z^\epsilon(\tilde{\sigma}_0^{\epsilon,\delta'}) \in C_\epsilon^k(\rho r(\epsilon)+3\epsilon) \} } 
\mathbb{P}_{Z^\epsilon(\tilde{\sigma}_0^{\epsilon,\delta'})} \left( Z^\epsilon(\tilde{\tau}^{\epsilon,\delta'}) \in C_\epsilon^l(\delta')  \right) \right] \\
& \geq \inf_{z_1 \in C_\epsilon^k(\rho r(\epsilon)+ 3\epsilon)} 
\mathbb{P}_{z_1} \left( Z^\epsilon(\tilde{\tau}^{\epsilon,\delta'}) \in C_\epsilon^l(\delta') \right)
\inf_{z \in C_\epsilon^k(\delta')}  \mathbb{P}_{z} \left(  Z^\epsilon(\tilde{\sigma}_0^{\epsilon,\delta'}) \in C_\epsilon^k(\rho r(\epsilon)+3\epsilon)  \right) \geq c.
\end{aligned}
\end{equation*}
Combining this with the result for $k=l$ and $n_0 = 1$, we complete the proof of Claim \ref{itm: c1}.
\end{proof}

\section{Proof of Lemmas \ref{exit place conti} and \ref{exit est conti}}\label{sec a prior est}

\subsection{Proof of Lemma \ref{exit place conti}}
It suffices to prove that, for every $k$ such that $I_k \sim O$,
\begin{equation*}
        \displaystyle
        \lim_{\delta' \to 0} \lim_{\epsilon \to 0} \sup_{z \in C_{\epsilon}(\rho r(\epsilon)+3\epsilon )} \abs{ \mathbb{P}_{z} \left( Z^\epsilon(\sigma^{\epsilon,\delta}) \in  C_{\epsilon}^k(\delta) \right) 
         - \mathbb{P}_{\alpha_{\epsilon,\delta',\delta}} \left( Z^\epsilon(\sigma^{\epsilon,\delta}) \in  C_{\epsilon}^k(\delta) \right)} = 0.
    \end{equation*}
    
    By Theorem \ref{conv qsd} and \eqref{est p}, for every $\eta >0$, there exists $N_{\eta} \in \mathbb{N}$ and $\delta_\eta' > 0$ such that for all $\delta' \in (0,\delta_\eta')$ there exists $\epsilon_{\eta,\delta'} >0$ with
	\begin{equation*}
		\sup_{z \in  C_\epsilon(\delta') } \left \lVert 
		\mathbb{P}_{z} \left( X_{N_\eta}^{\epsilon,\delta',\delta} \in \cdot \middle|  \tilde{\tau}_{N_\eta}^{\epsilon,\delta'} < \sigma^{\epsilon,\delta} \right) 
		-\alpha_{\epsilon,\delta',\delta}(\cdot) \right\rVert_{TV}
		< \frac{\eta}{3}
	\end{equation*}
	and
	\begin{equation*}
		\sup_{z \in  C_\epsilon(\delta') }
		\left(  \mathbb{P}_{z} \left( \tilde{\tau}_{N_\eta}^{\epsilon,\delta'} \geq \sigma^{\epsilon,\delta} \right) \vee  
         \abs{ 1- \frac{1}{\mathbb{P}_{z} \left( \tilde{\tau}_{N_\eta}^{\epsilon,\delta'} < \sigma^{\epsilon,\delta} \right)}} \right)
		< \frac{\eta}{3},
	\end{equation*}
	for all $\epsilon \in (0,\epsilon_{\eta,\delta'})$.
		Then, by the strong Markov property,
	\begin{equation*}
	\begin{aligned}
		& \sup_{z \in C_\epsilon(\rho r(\epsilon) + 3\epsilon)} \abs{\mathbb{P}_{z} \left( Z^\epsilon(\sigma^{\epsilon,\delta}) \in  C_{\epsilon}^k(\delta) \right) 
         - \mathbb{P}_{\alpha_{\epsilon,\delta',\delta}} \left( Z^\epsilon(\sigma^{\epsilon,\delta}) \in  C_{\epsilon}^k(\delta) \right) } \\
         & \leq \sup_{z \in C_\epsilon(\rho r(\epsilon) + 3\epsilon)} \abs{\mathbb{P}_{z} \left( Z^\epsilon(\sigma^{\epsilon,\delta}) \in  C_{\epsilon}^k(\delta), \tilde{\tau}_{N_\eta+1}^{\epsilon,\delta'} < \sigma^{\epsilon,\delta} \right) 
         - \mathbb{P}_{\alpha_{\epsilon,\delta',\delta}} \left( Z^\epsilon(\sigma^{\epsilon,\delta}) \in  C_{\epsilon}^k(\delta) \right) } \\
         & \ \ \ \ \ +
         \sup_{z \in C_\epsilon(\rho r(\epsilon) + 3\epsilon)} \abs{\mathbb{P}_{z} \left( Z^\epsilon(\sigma^{\epsilon,\delta}) \in  C_{\epsilon}^k(\delta), \tilde{\tau}_{N_\eta+1}^{\epsilon,\delta'} \geq \sigma^{\epsilon,\delta} \right) } \\
		& \leq \sup_{z \in C_\epsilon(\rho r(\epsilon) + 3\epsilon)} \abs{ \mathbb{E}_{z} \left[ \mathbb{P}_{Z^\epsilon(\tilde{\tau}_1^{\epsilon,\delta'})} \left( Z^\epsilon(\sigma^{\epsilon,\delta}) \in  C_{\epsilon}^k(\delta), \tilde{\tau}_{N_\eta}^{\epsilon,\delta'} < \sigma^{\epsilon,\delta} \right) 
		  - \mathbb{P}_{\alpha_{\epsilon,\delta',\delta}} \left( Z^\epsilon(\sigma^{\epsilon,\delta}) \in  C_{\epsilon}^k(\delta) \right) \right] } \\
         & \ \ \ \ \ +
         \sup_{z \in  C_\epsilon(\delta')} \mathbb{P}_{z} \left( \tilde{\tau}_{N_\eta}^{\epsilon,\delta'} \geq \sigma^{\epsilon,\delta} \right) \\
		& \leq \sup_{z \in C_\epsilon(\delta')} \abs{\mathbb{P}_{z} \left( Z^\epsilon(\sigma^{\epsilon,\delta}) \in  C_{\epsilon}^k(\delta), \tilde{\tau}_{N_\eta}^{\epsilon,\delta'} < \sigma^{\epsilon,\delta} \right) 
         - \mathbb{P}_{\alpha_{\epsilon,\delta',\delta}} \left( Z^\epsilon(\sigma^{\epsilon,\delta}) \in  C_{\epsilon}^k(\delta) \right) }
         + \frac{\eta}{3}.
	\end{aligned}
	\end{equation*}
	For the first term, we have
	\begin{equation*}
	\begin{aligned}
	&\sup_{z \in C_\epsilon(\delta')} \abs{\mathbb{P}_{z} \left( Z^\epsilon(\sigma^{\epsilon,\delta}) \in  C_{\epsilon}^k(\delta), \tilde{\tau}_{N_\eta}^{\epsilon,\delta'} < \sigma^{\epsilon,\delta} \right) 
         - \mathbb{P}_{\alpha_{\epsilon,\delta',\delta}} \left( Z^\epsilon(\sigma^{\epsilon,\delta}) \in  C_{\epsilon}^k(\delta) \right) } \\
	&\leq \sup_{z \in C_\epsilon(\delta')} \abs{\mathbb{P}_{z} \left( Z^\epsilon(\sigma^{\epsilon,\delta}) \in  C_{\epsilon}^k(\delta), \tilde{\tau}_{N_\eta}^{\epsilon,\delta'} < \sigma^{\epsilon,\delta} \right) 
         - \mathbb{P}_{z} \left( Z^\epsilon(\sigma^{\epsilon,\delta}) \in  C_{\epsilon}^k(\delta) \middle| \tilde{\tau}_{N_\eta}^{\epsilon,\delta'} < \sigma^{\epsilon,\delta} \right) } \\
         & \ \ \ \ \ +
         \sup_{z \in C_\epsilon(\delta')}
         \abs{ \mathbb{P}_{z} \left( Z^\epsilon(\sigma^{\epsilon,\delta}) \in  C_{\epsilon}^k(\delta) \middle| \tilde{\tau}_{N_\eta}^{\epsilon,\delta'} < \sigma^{\epsilon,\delta} \right)
         - \mathbb{P}_{\alpha_{\epsilon,\delta',\delta}} \left( Z^\epsilon(\sigma^{\epsilon,\delta}) \in  C_{\epsilon}^k(\delta) \right) } \\
         & \leq \sup_{z \in C_\epsilon(\delta')}
         \abs{ 1- \frac{1}{\mathbb{P}_{z} \left( \tilde{\tau}_{N_\eta}^{\epsilon,\delta'} < \sigma^{\epsilon,\delta} \right)}} \\
         & \ \ \ +
         \sup_{z \in C_\epsilon(\delta')} \abs{ \mathbb{E}_{z} \left[ \mathbb{P}_{X_{N_\eta}^{\epsilon,\delta',\delta}} \left( Z^\epsilon(\sigma^{\epsilon,\delta}) \in  C_{\epsilon}^k(\delta) \right) \middle| \tilde{\tau}_{N_\eta}^{\epsilon,\delta'} < \sigma^{\epsilon,\delta} \right]
         - \mathbb{P}_{\alpha_{\epsilon,\delta',\delta}} \left( Z^\epsilon(\sigma^{\epsilon,\delta}) \in  C_{\epsilon}^k(\delta) \right)} \\
         & \leq \frac{\eta}{3} + \sup_{z \in C_\epsilon(\delta')} \left\lVert\mathbb{P}_{z} \left( X_{N_\eta}^{\epsilon,\delta',\delta} \in \cdot \middle|  \tilde{\tau}_{N_\eta}^{\epsilon,\delta'} < \sigma^{\epsilon,\delta} \right) 
		-\alpha_{\epsilon,\delta',\delta}(\cdot) \right\rVert_{TV} < \frac{2 \eta}{3}.
	\end{aligned}
	\end{equation*}
	This allows us to conclude that 
	\begin{equation*}
        \displaystyle
        \lim_{\delta' \to 0} \lim_{\epsilon \to 0} \sup_{z \in C_{\epsilon}(\rho r(\epsilon)+3\epsilon )} \abs{ \mathbb{P}_{z} \left( Z^\epsilon(\sigma^{\epsilon,\delta}) \in  C_{\epsilon}^k(\delta) \right) 
         - \mathbb{P}_{\alpha_{\epsilon,\delta',\delta}} \left( Z^\epsilon(\sigma^{\epsilon,\delta}) \in  C_{\epsilon}^k(\delta) \right)} = 0.
    \end{equation*}

\subsection{Proof of Lemma \ref{exit est conti}}
We first consider the case of small and intermediate balls; that is, $r(\epsilon)^d \ll \epsilon^{d-1}$ or $r(\epsilon)^d \sim \epsilon^{d-1}$.
It suffices to prove that
\begin{equation*}
	\lim_{\delta' \to 0} \lim_{\epsilon \to 0} \sup_{z \in C_\epsilon(\rho r(\epsilon) + 3\epsilon)} \abs{\mathbb{E}_z \sigma^{\epsilon,\delta} - \mathbb{E}_{\alpha_{\epsilon,\delta',\delta}} \sigma^{\epsilon,\delta}} = 0.
\end{equation*}

First, we fix $\epsilon_1 > 0$ such that 
	\footnote[3]{We will keep track of the factor $\frac{r(\epsilon)^d}{\epsilon^{d-1}}$ even though in these two cases it can be easily bounded, as this makes it clearer how to adapt the argument to the case of large balls.} 
	\begin{equation}\label{asym 0}
		 \frac{r(\epsilon)^d}{\epsilon^{d-1}} \leq 2, \ \ \ \epsilon \in (0,\epsilon_1).
	\end{equation}
	Next, due to  Lemma \ref{small exit est 1.5}	\footnote[4]{Both $\sigma^{\epsilon,\delta}$ and $\tilde{\tau}^{\epsilon,\delta'}$ correspond to exit problems, differing only by the change of notation $\delta \mapsto \delta'$. Therefore, \eqref{cond 1} and \eqref{cond 2} follow directly from Lemma \ref{small exit est 1.5}.} and  \ref{inter exit est 1}, there exist $c_3 >0$ and $\epsilon_2 \in (0,\epsilon_1)$ such that for every $\epsilon \in (0,\epsilon_2)$
\begin{equation}\label{cond 1}
		\sup_{z \in B_\epsilon(\delta')} \mathbb{E}_z \tilde{\tau}^{\epsilon,\delta'} \leq c_3 \left( \frac{r(\epsilon)^d}{\epsilon^{d-1}}\delta' + (\delta')^2 \right),
	\end{equation}
	and
	\begin{equation}\label{cond 2}
		\sup_{z \in B_\epsilon(\delta)} \mathbb{E}_z \sigma^{\epsilon,\delta} \leq c_3 \left( \frac{r(\epsilon)^d}{\epsilon^{d-1}}\delta + \delta^2 \right),
	\end{equation}
	and
	\begin{equation}\label{cond 3}
		\sup_{z \in B_\epsilon(\delta)} \mathbb{E}_z  \left( \sigma^{\epsilon,\delta} \right)^2 
		\leq c_3^2 \left( \frac{r(\epsilon)^d}{\epsilon^{d-1}}\delta + \delta^2 \right)^2.
	\end{equation}

Next, according to Theorem \ref{conv qsd} and \eqref{est p}, for every $\eta >0$, we can fix $N_{\eta} \in \mathbb{N}$ and $\delta_\eta' > 0$ such that for all $\delta' \in (0,\delta_\eta')$, there exists $\epsilon_{\eta,\delta'} \in (0,\epsilon_2)$ such that
\begin{equation}\label{cond 4}
		\sup_{z \in  C_\epsilon(\delta') } \left \lVert 
		\mathbb{P}_{z} \left( X_{N_\eta}^{\epsilon,\delta',\delta} \in \cdot \middle|  \tilde{\tau}_{N_\eta}^{\epsilon,\delta'} < \sigma^{\epsilon,\delta} \right) 
		-\alpha_{\epsilon,\delta',\delta}(\cdot) \right\rVert_{TV}
		< \frac{\eta}{4c_3(2\delta+\delta^2)},
	\end{equation}
	and
	\begin{equation}\label{cond 5}
		\sup_{z \in C_\epsilon(\delta')} \left(
		\left( \mathbb{P}_{z} \left( \tilde{\tau}_{N_\eta}^{\epsilon,\delta'} \geq \sigma^{\epsilon,\delta} \right) \right)^{1/2}
		 \vee
		\abs{1 - \frac{1}{\mathbb{P}_z \left( \tilde{\tau}_{N_\eta}^{\epsilon,\delta'} < \sigma^{\epsilon,\delta}  \right)} } \right)
		< \frac{\eta}{4c_3(2\delta+\delta^2)} 
	\end{equation}
	for all $\epsilon \in (0,\epsilon_{\eta,\delta'})$.
	
Moreover, we can pick $\delta_\eta' > 0$ such that \begin{equation}\label{cond 6}
		(c_3+3)N_\eta \left( \delta+ \frac{r(\epsilon)^d}{\epsilon^{d-1}} \right) \delta' < \frac{\eta}{4}
	\end{equation}	
	holds for every $\delta' \in (0,\delta_\eta')$ and $\epsilon \in (0,\epsilon_{\delta',\eta})$.
	
	Then 
	\begin{equation*}
	\begin{aligned}
		& \sup_{z \in C_\epsilon(\rho r(\epsilon) + 3\epsilon)} \abs{\mathbb{E}_z \sigma^{\epsilon,\delta} - \mathbb{E}_{\alpha_{\epsilon,\delta,\delta'}} \sigma^{\epsilon,\delta}} 
		 = \sup_{z \in C_\epsilon(\rho r(\epsilon) + 3\epsilon)} \abs{\mathbb{E}_z \left( \sigma^{\epsilon,\delta} - \tilde{\tau}^{\epsilon,\delta'} + \tilde{\tau}^{\epsilon,\delta'} \right)
		 - \mathbb{E}_{\alpha_{\epsilon,\delta,\delta'}} \sigma^{\epsilon,\delta}} \\
		 & = \sup_{z \in C_\epsilon(\rho r(\epsilon) + 3\epsilon)} \abs{\mathbb{E}_z \left( \tilde{\tau}^{\epsilon,\delta'} + \mathbb{E}_{Z^\epsilon(\tilde{\tau}^{\epsilon,\delta'})}  \sigma^{\epsilon,\delta}  \right)
		 - \mathbb{E}_{\alpha_{\epsilon,\delta,\delta'}} \sigma^{\epsilon,\delta}} \\
		  & \leq \sup_{z \in C_\epsilon(\rho r(\epsilon) + 3\epsilon)} \mathbb{E}_z \tilde{\tau}^{\epsilon,\delta'}
		 + \sup_{z \in C_\epsilon(\delta')} 
		 \abs{  \mathbb{E}_{z} \sigma^{\epsilon,\delta}  
		 - \mathbb{E}_{\alpha_{\epsilon,\delta,\delta'}} \sigma^{\epsilon,\delta}} \\
		&\leq \sup_{z \in C_\epsilon(\rho r(\epsilon) + 3\epsilon)} \mathbb{E}_z \tilde{\tau}^{\epsilon,\delta'}
		 + \sup_{z \in C_\epsilon(\delta')}  \abs{\mathbb{E}_{z} \left( \sigma^{\epsilon,\delta} \mathbbm{1}_{\{\tilde{\tau}_{N_\eta}^{\epsilon,\delta'} < \sigma^{\epsilon,\delta}  \} } \right) 
		- \mathbb{E}_{\alpha_{\epsilon,\delta,\delta'}} \sigma^{\epsilon,\delta}}  \\
		& \quad + \sup_{z \in C_\epsilon( \delta')} \abs{\mathbb{E}_z \left( \sigma^{\epsilon,\delta} \mathbbm{1}_{\{\tilde{\tau}_{N_\eta}^{\epsilon,\delta'} \geq \sigma^{\epsilon,\delta}  \} } \right)}  \\
		& \eqqcolon I_1^{\epsilon,\delta',\delta} + I_2^{\epsilon,\delta',\delta} + I_3^{\epsilon,\delta',\delta}
	\end{aligned}
	\end{equation*}
	
	First, for $I_1^{\epsilon,\delta',\delta}$, by \eqref{cond 1},
	\begin{equation}\label{t est 1}
	\begin{aligned}
		& I_1^{\epsilon,\delta',\delta}
		\leq c_3 \left( \frac{r(\epsilon)^d}{\epsilon^{d-1}}\delta' +(\delta')^2 \right).
	\end{aligned}
	\end{equation}

	For $I_2^{\epsilon,\delta',\delta}$, we have
	\begin{equation*}
	\begin{aligned}
	&I_2^{\epsilon,\delta',\delta}
	= \sup_{z \in C_\epsilon(\delta')} \abs{\mathbb{E}_z \left( \left( \sigma^{\epsilon,\delta} - \tilde{\tau}_{N_\eta}^{\epsilon,\delta'} + \tilde{\tau}_{N_\eta}^{\epsilon,\delta'} \right)
		\mathbbm{1}_{\{ \tilde{\tau}_{N_\eta}^{\epsilon,\delta'} < \sigma^{\epsilon,\delta} \}} \right)
		- \mathbb{E}_{\alpha_{\epsilon,\delta,\delta'}} \sigma^{\epsilon,\delta}} \\
		& = \sup_{z \in C_\epsilon(\delta')} \abs{\mathbb{E}_z \left( \left( \tilde{\tau}_{N_\eta}^{\epsilon,\delta'} 
		+\mathbb{E}_{X_{N_\eta}^{\epsilon,\delta',\delta}}  \sigma^{\epsilon,\delta} \right) \mathbbm{1}_{\{ \tilde{\tau}_{N_\eta}^{\epsilon,\delta'} < \sigma^{\epsilon,\delta} \}} \right) 
		- \mathbb{E}_{\alpha_{\epsilon,\delta,\delta'}} \sigma^{\epsilon,\delta}} \\
		& \leq \sup_{z \in C_\epsilon(\delta')} \mathbb{E}_z \left( \tilde{\tau}_{N_\eta}^{\epsilon,\delta'} \mathbbm{1}_{\{ \tilde{\tau}_{N_\eta}^{\epsilon,\delta'} < \sigma^{\epsilon,\delta} \}} \right)  + \sup_{z \in C_\epsilon(\delta')} \mathbb{E}_{z}  \sigma^{\epsilon,\delta}  
		\sup_{z \in C_\epsilon(\delta')} \abs{ 1 - \frac{1}{\mathbb{P}_z \left( \tilde{\tau}_{N_\eta}^{\epsilon,\delta'} < \sigma^{\epsilon,\delta}  \right)} }  \\
		& \ \ \ \ +
		\sup_{z \in C_\epsilon(\delta')} 
		\abs{\mathbb{E}_{z} \left[ \mathbb{E}_{X_{N_\eta}^{\epsilon,\delta',\delta}}   \sigma^{\epsilon,\delta}   \middle| \tilde{\tau}_{N_\eta}^{\epsilon,\delta'} < \sigma^{\epsilon,\delta} \right] 
		- \mathbb{E}_{\alpha_{\epsilon,\delta,\delta'}} \sigma^{\epsilon,\delta} } \\
		& \ \ \ \ \ \ \ \ \ \ \ \ \eqqcolon J_{2,1}^{\epsilon,\delta',\delta} +  J_{2,2}^{\epsilon,\delta',\delta} + J_{2,3}^{\epsilon,\delta',\delta}.
	\end{aligned}
	\end{equation*}
	For $J_{2,1}^{\epsilon,\delta',\delta}$, let us first note that
	\begin{equation*}
	\begin{aligned}
		&\sup_{z \in C_\epsilon(\delta')} \mathbb{E}_z \left( \tilde{\tau}_{1}^{\epsilon,\delta'} \mathbbm{1}_{\{ \tilde{\tau}_{1}^{\epsilon,\delta'} < \sigma^{\epsilon,\delta} \}} \right) 
		 \leq \sup_{z \in C_\epsilon(\delta')} \mathbb{E}_z \left( \left( \tilde{\tau}_{1}^{\epsilon,\delta'} 
		-\tilde{\sigma}_{0}^{\epsilon,\delta'}
		+\tilde{\sigma}_{0}^{\epsilon,\delta'}  \right)
		\mathbbm{1}_{\{ \tilde{\tau}_{1}^{\epsilon,\delta'} < \sigma^{\epsilon,\delta} \}} \right) \\
		& \leq \sup_{z \in C_\epsilon(\delta')} 
		\mathbb{E}_z 
		\tilde{\sigma}_0^{\epsilon,\delta'}  
		+\sup_{z \in C_\epsilon(\rho r(\epsilon) + 3\epsilon)} \mathbb{E}_z \tilde{\tau}^{\epsilon,\delta'}
		+\sup_{z \in  C_\epsilon(2\delta')} \mathbb{E}_z \left( \tilde{\tau}^{\epsilon,\delta'} \mathbbm{1}_{\{ \tilde{\tau}^{\epsilon,\delta'} < \sigma^{\epsilon,\delta} \}} \right) \\
		&\leq  \frac{(\delta')^2}{2} +  c_3 \left( \frac{r(\epsilon)^d}{\epsilon^{d-1}}\delta' + (\delta')^2 \right) + \frac{\delta \delta'}{2}
		\leq (c_3+2) \left( \delta+ \frac{r(\epsilon)^d}{\epsilon^{d-1}} \right) \delta' ,
	\end{aligned}
	\end{equation*}
	where the first and the third terms follow from the properties of one dimensional Brownian motion, while the second term follows from \eqref{cond 1}. (Note that $Z^\epsilon(t)$ is a scaled Brownian motion.)
	By the strong Markov property, we then have
	\begin{equation}\label{t est 2}
		J_{2,1}^{\epsilon,\delta',\delta}
		\leq N_\eta \sup_{z \in C_\epsilon(\delta')} \mathbb{E}_z \left( \tilde{\tau}_{1}^{\epsilon,\delta'} \mathbbm{1}_{\{ \tilde{\tau}_{1}^{\epsilon,\delta'} < \sigma^{\epsilon,\delta} \}} \right)
		\leq (c_3+2) N_\eta \left(\delta + \frac{r(\epsilon)^d}{\epsilon^{d-1}} \right) \delta'.
	\end{equation}
	Next, if we define $\eta' \coloneqq \frac{\eta}{4c_3(2\delta+\delta^2)}$, by \eqref{cond 2} and \eqref{cond 5},
	\begin{equation}\label{t est 3}
		J_{2,2}^{\epsilon,\delta',\delta}
		= \sup_{z \in C_\epsilon(\delta')} \mathbb{E}_z \sigma^{\epsilon,\delta} 
		\sup_{z \in C_\epsilon(\delta')} \abs{ 1 - \frac{1}{\mathbb{P}_z \left( \tilde{\tau}_{N_\eta}^{\epsilon,\delta'} < \sigma^{\epsilon,\delta}  \right)} }
		< c_3 \left( \frac{r(\epsilon)^d}{\epsilon^{d-1}}\delta +\delta^2 \right) \eta'.
	\end{equation}
	For $J_{2,3}^{\epsilon,\delta',\delta}$, by \eqref{cond 2} and \eqref{cond 4},
	\begin{equation}\label{t est 4}
	\begin{aligned}
	&\sup_{z \in C_\epsilon(\delta')} \abs{\mathbb{E}_{z} \left[ \mathbb{E}_{X_{N_\eta}^{\epsilon,\delta',\delta}}   \sigma^{\epsilon,\delta}   \middle| \tilde{\tau}_{N_\eta}^{\epsilon,\delta',\delta} < \sigma^{\epsilon,\delta} \right] 
	 - \mathbb{E}_{\alpha_{\epsilon,\delta,\delta'}} \sigma^{\epsilon,\delta} } \\
	&\leq \sup_{z \in C_\epsilon(\delta')} \mathbb{E}_z \sigma^{\epsilon,\delta} 
	\sup_{z \in C_\epsilon(\delta')} \left \lVert 
		\mathbb{P}_{z} \left( X_{N_\eta}^{\epsilon,\delta',\delta} \in \cdot \middle|  \tilde{\tau}_{N_\eta}^{\epsilon,\delta',\delta} < \sigma^{\epsilon,\delta} \right) 
		-\alpha_{\epsilon,\delta',\delta}(\cdot) \right\rVert_{TV}
		< c_3 \left( \frac{r(\epsilon)^d}{\epsilon^{d-1}}\delta +\delta^2 \right) \eta'.
	\end{aligned}
	\end{equation}

	Finally, for $I_3^{\epsilon,\delta',\delta}$, by \eqref{cond 3} and \eqref{cond 5},
	\begin{equation}\label{t est 5}
	\begin{aligned}
		& I_3^{\epsilon,\delta}
		\leq \sup_{z \in C_\epsilon(\delta')}
		\left[ \mathbb{E}_z \left( \sigma^{\epsilon,\delta} \right)^2 \right]^{1/2} 
		\left( \sup_{z \in C_\epsilon(\delta')} \mathbb{P}_{z} \left( \tilde{\tau}_{N_\eta}^{\epsilon,\delta'} \geq \sigma^{\epsilon,\delta} \right) \right)^{1/2} 
		 \leq
		c_3 \left( \frac{r(\epsilon)^d}{\epsilon^{d-1}}\delta +\delta^2 \right) \eta'.
	\end{aligned}
	\end{equation}
	
	Combining \eqref{t est 1}, \eqref{t est 2}--\eqref{t est 5} and \eqref{asym 0}, \eqref{cond 6}, we conclude that
	\begin{equation*}
	\begin{aligned}
		&\sup_{z \in C_\epsilon(\rho r(\epsilon) + 3\epsilon)} \abs{\mathbb{E}_z \sigma^{\epsilon,\delta} - \mathbb{E}_{\alpha_{\epsilon,\delta,\delta'}} \sigma^{\epsilon,\delta}}
		\leq (c_3+3) N_\eta \left(\delta + \frac{r(\epsilon)^d}{\epsilon^{d-1}} \right) \delta' + 3c_3 \left( \frac{r(\epsilon)^d}{\epsilon^{d-1}}\delta +\delta^2 \right) \eta'  \\
		&\leq (c_3+3) N_\eta \left(\delta + \frac{r(\epsilon)^d}{\epsilon^{d-1}} \right) \delta' + 3c_3 \left( 2\delta+\delta^2 \right) \eta' < \eta.
	\end{aligned}
	\end{equation*}
	
	In the case of large balls, where $r(\epsilon)^d \gg \epsilon^{d-1}$, we replace condition \eqref{asym 0} with $\frac{r(\epsilon)^d}{\epsilon^{d-1}} > 2$.
	Accordingly, we  modify condition \eqref{cond 6} to
	\begin{equation*}
	(c_3+3)N_\eta (\delta+1) \delta' < \frac{\eta}{4}.
	\end{equation*}
	We get
	\begin{equation*}
	\begin{aligned}
		& \sup_{z \in C_\epsilon(\rho r(\epsilon) + 3\epsilon)} \abs{\mathbb{E}_z \sigma^{\epsilon,\delta} - \mathbb{E}_{\alpha_{\epsilon,\delta,\delta'}} \sigma^{\epsilon,\delta}}
		\left(\frac{r(\epsilon)^d}{\epsilon^{d-1}} \right)^{-1}  \leq (c_3+3) N_\eta (\delta+1) \delta' + 3c_3 \left( \delta + \delta^2 \right) \eta' \\ 
		& < \frac{\eta}{4} + 3c_3 \left( \delta + \delta^2 \right) \frac{\eta}{4c_3(2\delta+\delta^2)} < \eta
	\end{aligned}
	\end{equation*}
	and we have the desired result.

\section{Asymptotic exponential law for the exit time in the large ball regime}\label{sec exit exp}
In this section, we consider only the large balls case; that is $r_j(\epsilon)^d \gg \epsilon^{d-1}$.
Since we continue to focus on the local behavior near a vertex, we omit the subscript $j$ throughout this section.

We introduce a sequence of stopping times 
\begin{equation*}
	\hat{\sigma}_n^{\epsilon,\delta',\delta} \coloneqq \inf \{ t \geq  \hat{\tau}_n^{\epsilon,\delta',\delta}: Z^\epsilon(t) \in C_\epsilon(\delta') \}
\end{equation*}
and
\begin{equation*}
	\hat{\tau}_n^{\epsilon,\delta',\delta} \coloneqq \inf \{ t \geq  \hat{\sigma}_{n-1}^{\epsilon,\delta',\delta}: Z^\epsilon(t) \in C_\epsilon(\rho r(\epsilon)+3\epsilon) \cup C_\epsilon(\delta) \},
\end{equation*}
with $\hat{\tau}_0^{\epsilon,\delta',\delta} \coloneqq 0$.
Additionally, we define $\hat{\sigma}^{\epsilon,\delta'}$ the first hitting time of $Z^\epsilon(t)$ to $C_\epsilon(\delta')$ and
$\hat{\tau}^{\epsilon,\delta}$ the first hitting time of $Z^\epsilon(t)$ to $C_\epsilon(\rho r(\epsilon)+3\epsilon) \cup C_\epsilon(\delta)$.

Let $\mathcal{S}_{\epsilon} = C_\epsilon(\rho r(\epsilon)+3\epsilon)$.
We define the discrete time Markov chain $ \hat{X}_n^{\epsilon,\delta',\delta} \coloneqq Z^\epsilon( \hat{\tau}_n^{\epsilon,\delta',\delta} \wedge \sigma^{\epsilon,\delta})$ and the filtration $ \mathcal{F}_n^{\epsilon,\delta',\delta} \coloneqq \left\{ \sigma \left( \hat{X}^{\epsilon,\delta',\delta}_m \right): m \leq n \right\}$,
and $\xi_n^{\epsilon,\delta',\delta}$ as the non-negative random variable such that 
\begin{equation*}
\begin{aligned}
 \xi_0^{\epsilon,\delta',\delta} &= 0 \\
	 \xi_n^{\epsilon,\delta',\delta}  &=  \mathbbm{1}_{ \{ \hat{X}_{n-1}^{\epsilon,\delta',\delta} \in C_\epsilon(\delta) \}}e(\epsilon,\delta') 
	+ \mathbbm{1}_{ \{ \hat{X}_{n-1}^{\epsilon,\delta',\delta} \in C_\epsilon(\rho r(\epsilon)+3\epsilon)  \}} \left( \hat{\tau}_n^{\epsilon,\delta',\delta} - \hat{\tau}_{n-1}^{\epsilon,\delta',\delta} \right) , \quad n \geq 1
\end{aligned}
\end{equation*}
where
\begin{equation}\label{est e}
e(\epsilon,\delta') = 
\frac{\rho^d r(\epsilon)^d V_d}{\sum_{k:I_k \sim O} \lambda_k^{d-1} \epsilon^{d-1} V_{d-1}} \delta'.
\end{equation}
Moreover, we define
\begin{equation*}
\zeta_n^{\epsilon,\delta',\delta} = 
\begin{dcases}
1, \text{ if } \hat{X}_{n}^{\epsilon,\delta',\delta} \in C_\epsilon(\rho r(\epsilon)+3\epsilon), \\
0, \text{ otherwise}.
\end{dcases}
\end{equation*}
Then it is clear that for $z \in \mathcal{S}_{\epsilon}$,
\begin{equation*}
	\sum_{n=0}^\infty \xi_{n+1}^{\epsilon,\delta',\delta} \zeta_n^{\epsilon,\delta',\delta} = \sigma^{\epsilon,\delta}, \quad a.s.
\end{equation*}

The following theorem closely parallels \cite{SMP}, Lemma 6.8, with the only difference being the inclusion of the additional parameter, $\delta'$. 
The proof follows essentially the same argument as in \cite{SMP}. 
For completeness, we include the proof in Appendix \ref{sec pf exp}.

\begin{Theorem}\label{thm abs}
Let $e(\epsilon,\delta')$ be the quantity defined in \eqref{est e} and  $p(\delta',\delta) = \frac{\delta'}{\delta}$. 
	Suppose that 
	\begin{equation}\label{fir moment}
	\lim_{\delta' \to 0} \lim_{\epsilon \to 0}  \sup_{z \in \mathcal{S}_{\epsilon}} \sup_{ n \geq 0} \abs{ \frac{\mathbb{E}_z \left( \xi_{n+1}^{\epsilon,\delta',\delta} \middle| \mathcal{F}_n^{\epsilon,\delta',\delta} \right)}{e(\epsilon,\delta')} - 1} = 0,
	\end{equation}
	 and 
	\begin{equation}\label{sec moment}
		\lim_{\delta' \to 0} \lim_{\epsilon \to 0}  \sup_{z \in \mathcal{S}_{\epsilon}} \sup_{ n \geq 0} \mathbb{E}_z \left( \left( \xi_{n+1}^{\epsilon,\delta',\delta} \right)^2 \middle| \mathcal{F}^{\epsilon,\delta',\delta}_{n} \right) (e(\epsilon,\delta'))^{-2} < \infty.
	\end{equation}
	In addition, suppose
	\begin{equation}\label{trans prob}
		\lim_{\delta' \to 0} \lim_{\epsilon \to 0}  \sup_{z \in \mathcal{S}_{\epsilon}} \sup_{ n \geq 0} \abs{ \frac{\mathbb{P}_z \left( \zeta_{n+1}^{\epsilon,\delta',\delta}=0 \middle| \zeta_n^{\epsilon,\delta',\delta} = 1 \right)}{p(\delta',\delta)} - 1 }= 0.
	\end{equation}

	Let $K^{\epsilon,\delta',\delta} = \min\{ n: \zeta_n^{\epsilon,\delta',\delta} =0 \}$. Then for each $t \geq 0$,
	\begin{equation*}
		\lim_{\delta' \to 0} \lim_{\epsilon \to 0} \sup_{z \in \mathcal{S}_{\epsilon}} \abs{ \mathbb{P}_z (p(\delta',\delta)K^{\epsilon,\delta',\delta} \geq t) - e^{-t} } =0,
	\end{equation*}
	and, moreover,
	\begin{equation}\label{exp conv 2}
		\lim_{\delta' \to 0} \lim_{\epsilon \to 0} \sup_{z \in \mathcal{S}_{\epsilon}} \abs{ \mathbb{P}_z \left( \left( e(\epsilon,\delta') \right)^{-1} p(\delta',\delta) \sum_{n=0}^\infty \xi_{n+1}^{\epsilon,\delta',\delta} \zeta_n^{\epsilon,\delta',\delta}  \geq t \right) - e^{-t}} = 0.
	\end{equation}
\end{Theorem}

Now we verify \eqref{fir moment}, \eqref{sec moment}, \eqref{trans prob} hold.

\begin{Lemma}\label{lem trans prob}
We have
    \begin{equation*}
         \lim_{\epsilon \to 0}  \sup_{z \in \mathcal{S}_{\epsilon}} \sup_{ n \geq 0} \abs{ \frac{\mathbb{P}_z \left( \zeta_{n+1}^{\epsilon,\delta',\delta}=0 \middle| \zeta_n^{\epsilon,\delta',\delta} = 1 \right)}{p(\delta',\delta)} - 1 }= 0.
    \end{equation*}
\end{Lemma}
\begin{proof}
By the strong Markov property, 
\begin{equation*}
\begin{aligned}
	\mathbb{P}_z \left(\zeta_{n+1}^{\epsilon,\delta',\delta} = 0 \middle| \zeta_{n}^{\epsilon,\delta',\delta} =1 \right) 
	&= \mathbb{E}_z \left[ \mathbbm{1}_{ \{ Z^\epsilon(\hat{\tau}_n^{\epsilon,\delta',\delta} \wedge \sigma^{\epsilon,\delta}) \in C_\epsilon(\rho r(\epsilon)+3\epsilon) \} }
    \mathbb{P}_{Z^\epsilon(\hat{\sigma}_n^{\epsilon,\delta',\delta})}  \left( Z^\epsilon(\hat{\tau}^{\epsilon,\delta}) \in C_\epsilon(\delta)    \right) \right] \\
    & \qquad \times \left( \mathbb{P}_z \left( Z^\epsilon(\hat{\tau}_n^{\epsilon,\delta',\delta} \wedge \sigma^{\epsilon,\delta}) \in C_\epsilon(\rho r(\epsilon)+3\epsilon) \right) \right)^{-1}.
\end{aligned}
\end{equation*}
Since the exit problem is only related to one dimensional Brownian motion,
    \begin{equation*}
    	\mathbb{P}_{z'} \left( Z^\epsilon(\hat{\tau}^{\epsilon,\delta}) \in C_\epsilon(\delta)  \right)
    	=\frac{\delta' - \rho r(\epsilon)-3\epsilon}{\delta - \rho r(\epsilon)-3\epsilon},
    \end{equation*}
    for all $z' \in C_\epsilon(\delta')$,
	and this completes the proof of \eqref{trans prob}.
\end{proof}

\begin{Lemma}\label{lem fir moment}
    We have
    \begin{equation*}
        \lim_{\delta' \to 0} \lim_{\epsilon \to 0}  \sup_{z \in \mathcal{S}_{\epsilon}} \sup_{ n \geq 0} \abs{ \frac{\mathbb{E}_z \left( \xi_{n+1}^{\epsilon,\delta',\delta} \middle| \mathcal{F}_n^{\epsilon,\delta',\delta} \right)}{e(\epsilon,\delta')} - 1} = 0.
    \end{equation*}
\end{Lemma}
\begin{proof}
By the strong Markov property,
\begin{equation*}
	\mathbb{E}_z \left( \xi_{n+1}^{\epsilon,\delta',\delta} \middle| \mathcal{F}_n^{\epsilon,\delta',\delta} \right)
	= \mathbb{E}_z \left( \mathbb{E}_{\hat{X}_{n}^{\epsilon, \delta',\delta}}  \xi_{1}^{\epsilon,\delta',\delta}    \right).
\end{equation*}
Thanks to the definition of $\xi_{n}^{\epsilon,\delta',\delta}$, if $\hat{X}_{n}^{\epsilon, \delta',\delta} \in C_\epsilon(\delta) $, the result trivially holds.
It then suffices to consider $\hat{X}_{n}^{\epsilon, \delta',\delta} \in C_\epsilon(\rho r(\epsilon)+3\epsilon) $.
For $z' \in C_\epsilon(\rho r(\epsilon)+3\epsilon)$, by the strong Markov property, we have
\begin{equation*}
	\mathbb{E}_{z'}  \xi_{1}^{\epsilon,\delta'}   
	= \mathbb{E}_{z'} \left[ \hat{\sigma}_0^{\epsilon,\delta'} + \mathbb{E}_{Z^\epsilon(\hat{\sigma}_0^{\epsilon,\delta'})}  \hat{\tau}^{\epsilon,\delta} \right].
\end{equation*}
For $z'' \in C_\epsilon(\delta')$, the exit time $\hat{\tau}^{\epsilon,\delta}$ is only related to one dimensional Brownian motion.
Since $Z^\epsilon(t)$ is a scaled Brownian motion, we have
\begin{equation*}
	\mathbb{E}_{z''} \hat{\tau}^{\epsilon,\delta}
	=\frac{(\delta-\delta')(\delta'-\rho r(\epsilon)-3\epsilon)}{2}.
\end{equation*}
On the other hand, by Proposition \ref{small exit est 2},
\begin{equation*}
	\mathbb{E}_{z'} \hat{\sigma}^{\epsilon,\delta'} \sim_{\epsilon,\delta'} 	 \frac{\delta'^2}{2} + e(\epsilon,\delta') 
\end{equation*}
uniformly in $ z' \in C_\epsilon(\rho r(\epsilon)+3\epsilon)$.
Since $\lim_{\epsilon \to 0} e(\epsilon,\delta') = \infty$, we have the desired result.
\end{proof}

\begin{Lemma}\label{lem sec moment}
    We have
    \begin{equation*}
       \lim_{\delta' \to 0} \lim_{\epsilon \to 0} \sup_{z \in \mathcal{S}_{\epsilon}} \sup_{ n \geq 0} \mathbb{E}_z \left( \left( \xi_{n+1}^{\epsilon,\delta',\delta} \right)^2 \middle| \mathcal{F}^{\epsilon,\delta',\delta}_{n} \right) (e(\epsilon,\delta'))^{-2} < \infty.
    \end{equation*}
\end{Lemma}
\begin{proof}
Similarly, it suffices to prove that there exists $\delta'_0 >0$ such that for every $\delta' \in (0,\delta_0')$, there exists $\epsilon_{\delta'} > 0$ such that
\begin{equation*}
	\sup_{z \in C_\epsilon(\rho r(\epsilon)+3\epsilon)  } \mathbb{E}_{z} \left( \xi_{1}^{\epsilon,\delta'} \right)^2 \leq c (e(\epsilon,\delta'))^2  
\end{equation*}
for all $\epsilon \in (0,\epsilon_{\delta'})$.
	Note that, by Theorem \ref{SMFPT Main} and Proposition \ref{small exit est 2}, there exists $\delta'_0 >0$ such that for every $\delta' \in (0,\delta_0')$, there exists $\epsilon_{\delta'} > 0$ such that
	\begin{equation*}
	\begin{aligned}
		& \sup_{z \in C_\epsilon(\rho r(\epsilon)+3\epsilon) } \mathbb{E}_z \left(\xi_1^{\epsilon,\delta',\delta} \right)^2
		\leq 2 \left( \sup_{z \in C_\epsilon(\rho r(\epsilon)+3\epsilon) }\mathbb{E}_z \left( \hat{\sigma}^{\epsilon,\delta'} \right)^2+ \sup_{z' \in C_\epsilon(\delta')} \mathbb{E}_{z'} \left( \hat{\tau}^{\epsilon,\delta} \right)^2 \right) \\
		& \leq c \left( \left( \sup_{z \in C_\epsilon(\rho r(\epsilon)+3\epsilon) } \mathbb{E}_z  \hat{\sigma}^{\epsilon,\delta'} \right)^2 + \left( \sup_{z' \in C_\epsilon(\delta')} \mathbb{E}_{z'}  \hat{\tau}^{\epsilon,\delta} \right)^2 \right)
		\leq c(e(\epsilon,\delta'))^2
	\end{aligned}
	\end{equation*}
	for all $\epsilon \in (0,\epsilon_{\delta'})$.
\end{proof}

By Theorem \ref{thm abs}, the exit time $\sigma^{\epsilon,\delta}$ is, in a loose sense, asymptotically distributed as an exponential random variable with parameter $\left( \frac{\rho^d r(\epsilon)^d V_d}{\sum_{k} \lambda_k^{d-1} \epsilon^{d-1} V_{d-1}} \delta \right)^{-1}$. 
Given the assumption $r(\epsilon)^d \gg \epsilon^{d-1}$, one expects the moment generating function $\mathbb{E}_z e^{-\lambda \sigma^{\epsilon,\delta} }$ to converges to $0$, for all $\lambda>0$.
This observation is formalized in the following theorem:\begin{Theorem}\label{large exit time est}
For each $t>0$,
	\begin{equation}\label{exp est 1}
		 \lim_{\epsilon \to 0} \sup_{z \in C_\epsilon(\rho r(\epsilon)+3\epsilon)  } \abs{ \mathbb{P}_z \left( \sigma^{\epsilon,\delta} \geq t \frac{\rho^d r(\epsilon)^d V_d}{\sum_{k:I_k \sim O} \lambda_k^{d-1} \epsilon^{d-1} V_{d-1}} \delta \right) - e^{-t}} = 0.
	\end{equation}
	Moreover, for every $\lambda >0$, 
	\begin{equation}\label{exp est 2}
	\displaystyle
		\lim_{\epsilon \to 0} \sup_{z \in C_\epsilon(\rho r(\epsilon)+3\epsilon) } \mathbb{E}_z e^{-\lambda \sigma^{\epsilon,\delta} } = 0.
	\end{equation}
\end{Theorem}
\begin{proof}
The first part is a direct consequence of \eqref{exp conv 2}. 
For the second part, for every $\eta > 0$, let $t_\eta$ be a positive real number such that $e^{-t_\eta} \geq 1-\frac{\eta}{3}$.
If we define
\begin{equation*}
	\alpha(\epsilon) = \frac{\rho^d r(\epsilon)^d V_d}{\sum_{k:I_k \sim O} \lambda_k^{d-1} \epsilon^{d-1} V_{d-1}},
\end{equation*}
we have
\begin{equation*}
\begin{aligned}
\displaystyle
&\mathbb{E}_z e^{-\lambda \sigma^{\epsilon,\delta} } 
=\mathbb{E}_z \left( e^{-\lambda \sigma^{\epsilon,\delta} } \mathbbm{1}_{ \left\{ \sigma^{\epsilon,\delta} \geq t_\eta \alpha(\epsilon) \delta \right\} }  \right) 
+\mathbb{E}_z \left( e^{-\lambda \sigma^{\epsilon,\delta} } \mathbbm{1}_{ \left\{ \sigma^{\epsilon,\delta} < t_\eta \alpha(\epsilon) \delta \right\} } \right) \\
& \leq \exp {-\lambda t_\eta \alpha(\epsilon) \delta }    
+\mathbb{P}_z \left( \sigma^{\epsilon,\delta} < t_\eta \alpha(\epsilon) \delta  \right).
\end{aligned}
\end{equation*}
For the second term, thanks to \eqref{exp est 1}, there exists $\epsilon_{1,\delta, \eta} > 0$ such that for every $\epsilon \in (0,\epsilon_{1,\delta, \eta})$,
\begin{equation*}
\begin{aligned}
	&\mathbb{P}_z \left( \sigma^{\epsilon,\delta} < t_\eta \alpha(\epsilon) \delta  \right)
	= 1 - \mathbb{P}_z \left( \sigma^{\epsilon,\delta} \geq t_\eta \alpha(\epsilon) \delta  \right)
	\leq 1-e^{t_\eta}+\frac{\eta}{3}
	\leq \frac{2\eta}{3}.
\end{aligned}
\end{equation*}
For the first term, there exists $\epsilon_{2,\delta, \eta} \in (0,\epsilon_{1,\delta, \eta}) $ such that for every $\epsilon \in (0,\epsilon_{2,\delta, \eta})$,
\begin{equation*}
	\exp {-\lambda t_\eta \alpha(\epsilon) \delta }  < \frac{\eta}{3}.
\end{equation*}
This completes the proof.
\end{proof}

\section{Convergence of diffusion processes}\label{sec conv diff}
Our strategy closely follows the approach developed in \cite{SDE2}. 
First, we show that the family of probability measures $ \mathbb{P} \circ (\Pi^\epsilon)^{-1}$, induced by the projected diffusion process $(\Pi^\epsilon \circ Z^\epsilon) (\cdot)$, is weakly precompact in $C([0,\infty);\Gamma)$.
Next, we verify that any limiting measure possesses $\bar{L}$ as a generator.
Therefore, since the law of a process on $C([0,\infty);\Gamma)$ is uniquely determined by its generator $\bar{L}$ and its initial condition, we can conclude the weak convergence of the process $(\Pi^\epsilon \circ Z^\epsilon)(\cdot)$ to $\bar{Z}(\cdot)$ in $C([0,\infty);\Gamma)$.

We first show the precompactness of the probability measures $\{ \mathbb{P} \circ (\Pi^\epsilon)^{-1} \}_{\epsilon \in (0,1)}$.
\begin{Lemma}
    The family $\{ \mathbb{P} \circ (\Pi^\epsilon)^{-1} \}_{\epsilon \in (0,1)}$ is weakly precompact on $C([0,\infty);\Gamma).$
\end{Lemma}
\begin{proof}
    By \cite{SDE2}, Theorem 2.1, (see also \cite{Diffbook}), it suffices to show that for any $\eta > 0$, there exists a constant $A_\eta \geq 0$ such that for any $a \in \Gamma$, there exists a function $f^a_\eta(x): \Gamma \to [0,1]$ with $f^a_\eta(a)=1$, $f^a_\eta(x)=0$ for $d_\Gamma(x,a) \geq \eta$, and such that the process 
\begin{equation*}
	\{ f^a_\eta(\Pi^\epsilon(Z^\epsilon(t))+A_\eta t\}_{t \geq 0}
\end{equation*}    
    is a submartingale with respect to the filtration $\{ \tilde{\mathcal{F}}^\epsilon_t \}$, where $\tilde{\mathcal{F}}^\epsilon_t = \sigma( (\Pi^\epsilon \circ Z^\epsilon) (s): 0 \leq s \leq t)$.

    Let $h(x):[0,\infty) \to [0,1]$ be a smooth function such that
    \begin{equation*}
    \displaystyle
    h(x)=\begin{cases}
        1 \ \text{ if } \ x \in [0,\frac{1}{16}], \\
        0 \ \text{ if } \ x \geq \frac{1}{8}
    \end{cases}.
\end{equation*}  
    For $\epsilon > 0$ such that 
    \begin{equation*}
        \left( \max_{j=1,\cdots,\abs{V}} \rho_j r_j(\epsilon) +  3\epsilon \right) \vee  \left(  \max_{k=1,\cdots,\abs{E}} \lambda_k \epsilon \right) < \frac{\eta}{4},
    \end{equation*} 
    we define the function $f^a_\eta(x)$, for $x \in \Gamma$, in the following way:
    if $d_\Gamma(a,O_j) \leq \frac{3\eta}{8}$ for some $j$,
    \begin{equation*}
    \begin{aligned}
        f^a_\eta(x) \coloneqq \begin{dcases}
            1 \ &\text{if} \ d_\Gamma(x,O_j) \leq \frac{3\eta}{8}, \\
            h\left(\frac{d_\Gamma(x,O_j)}{\eta} - \frac{3}{8}\right) \ &\text{otherwise}.
        \end{dcases}
        \end{aligned}
    \end{equation*}
    If $d_\Gamma(a,O_j) > \frac{3\eta}{8}$ for all $j$,
    \begin{equation*}
        f^a_\eta(x) \coloneqq 
            h\left(\frac{d_\Gamma(x,a)}{\eta}\right) .
    \end{equation*}
    
    Now, let $g^a_{\eta, \epsilon}(z) \coloneqq f^a_\eta(\Pi^\epsilon(z))$, for $z \in G_\epsilon$.
    We claim that $g^a_{\eta, \epsilon}$ has zero normal derivative at the boundary of $G_\epsilon$.
    To see this, consider first the case when $z \in G_\epsilon$ satisfies $d_\Gamma(\Pi(z),O_j) \leq \rho_j r_j(\epsilon)+3\epsilon$.
    By the property of $\Pi^\epsilon$ in \eqref{proj 2}, we have $d_\Gamma(\Pi^\epsilon(z),O_j) \leq \rho_j r_j(\epsilon)+3\epsilon$ as well.
    Since $d_\Gamma(\Pi^\epsilon (z),O_j) \leq \rho_j r_j(\epsilon)+3\epsilon \leq \frac{\eta}{4}$, if $d_\Gamma(a,O_j) \leq \frac{3\eta}{8}$, for some $j$, we get $g^a_{\eta, \epsilon}(z) = 1$.
    On the other hand, if $d_\Gamma(a,O_j) > \frac{3\eta}{8}$ for all $j$, then $d_\Gamma(a,\Pi^\epsilon(z)) \geq \frac{\eta}{8}$ and hence $g^a_{\eta, \epsilon}(z) = 0$.
    Therefore, $g^a_{\eta,\epsilon}(z)$ is constant in this region and thus has zero normal derivative there.
    
    Now, suppose $z \in G_\epsilon$ and $d_\Gamma(\Pi(z),O_j) > \rho_j r_j(\epsilon)+2\epsilon$.
    In this case, the projection $\Pi^\epsilon$ is simply the nearest projection and is constant along cross sections of the tube.
    It follows that $g^a_{\eta,\epsilon}(z)$ remains constant in the normal direction, and hence the normal derivative vanishes.
	
	Moreover, $g^a_{\eta,\epsilon}$ is twice differentiable, and an application of the chain rule yields
\begin{equation*}
	\sup_{z \in G_\epsilon} \abs{ \Delta g^a_{\eta, \epsilon}(z)} \leq  \frac{c}{\eta^2} \abs{h}_{C^2} \eqqcolon A_\eta,
\end{equation*}	
where the constant $c>0$ is independent of $a \in \Gamma$.
	To be more precise, based on the preceding discussion, it suffices to consider points $z$ such that $d_\Gamma(\Pi(z),O_j) \geq \rho_j r_j(\epsilon) + 2\epsilon$, where $\Pi^\epsilon$ acts as the nearest projection. 
	In this regime, the bound on the Laplacian follows directly from the definition of $h$ and the chain rule.
	
    By the martingale formulation, we can conclude that
    \begin{equation*}
        f^a_\eta(\Pi^\epsilon(Z^\epsilon(t)))-\int_0^t  \Delta g^a_{\eta, \epsilon}(Z^\epsilon(s))ds
        = g^a_{\eta, \epsilon}(Z^\epsilon(t))-\int_0^t  \Delta g^a_{\eta, \epsilon}(Z^\epsilon(s))ds
    \end{equation*}
    is a martingale with respect to $\mathcal{F}_t$ and thus $f^a_\eta(\Pi^\epsilon(Z^\epsilon(t))+A_\eta t$ is a submartingale with respect to $\mathcal{F}_t$.
    
    Finally, since $\Pi^\epsilon$ is continuous and $\tilde{\mathcal{F}}_s^\epsilon \subset \mathcal{F}_s$,
    \begin{equation*}
    \begin{aligned}
    		& \mathbb{E} \left( f^a_\eta(\Pi^\epsilon(Z^\epsilon(t))+A_\eta t \middle| \tilde{\mathcal{F}}_s^\epsilon \right) 
    		 =  \mathbb{E}  \left( 
    		\mathbb{E} \left( f^a_\eta(\Pi^\epsilon(Z^\epsilon(t))+A_\eta t \middle| \mathcal{F}_s \right)
    		\middle| \tilde{\mathcal{F}}_s^\epsilon \right) \\
    		&\geq  \mathbb{E}  \left(
    			f^a_\eta(\Pi^\epsilon(Z^\epsilon(s))+A_\eta s
    		\middle| \tilde{\mathcal{F}}_s^\epsilon \right) 
    		=  f^a_\eta(\Pi^\epsilon(Z^\epsilon(s))+A_\eta s.
   	\end{aligned}
    \end{equation*}
    We then complete the proof.
\end{proof}

Now the only remaining part is showing that the limiting measure possesses $\bar{L}$ as its generator. 
In what follows, we will use the notation introduced in Definition \ref{def ball} for $\mathfrak{S}$, $\mathfrak{M}$, $\mathfrak{L}$.
It suffices to show the following result.
\begin{Lemma}\label{conv dif abs}
	Let $\bar{L}$ be a linear operator described in Theorem \ref{op on graph}, with $\mathcal{L}_k f =f''$, for all $k= 1,\cdots,\abs{E} $, and
\begin{equation*}
\displaystyle
p_{j,k}  = 
\begin{dcases}
\frac{\lambda_k^{d-1}}{\sum_{l:I_l \sim O_j} \lambda_l^{d-1}}, \ \ \ &\text{ if } j \notin \mathfrak{L}, \\
0, \ \ \ &\text{ if } j \in \mathfrak{L},
\end{dcases}
\end{equation*}
and 
\begin{equation*}
\alpha_j = 
\begin{dcases}
0, \ \ \ &\text{ if } j \in \mathfrak{S}, \\
\frac{\rho_j^d V_d}{\sum_{k:I_k \sim O_j} \lambda_k^{d-1} V_{d-1}}, \ \ \ &\text{ if } j \in \mathfrak{M}, \\
1, \ \ \ &\text{ if } j \in \mathfrak{L}.
\end{dcases}
\end{equation*}
	Then for every $f \in D(\bar{L})$ and $\lambda>0$, we have
	\begin{equation*}
	\begin{aligned}
	\lim_{\epsilon \to 0} \sup_{z \in G_\epsilon} \mathbb{E}_z \int_0^\infty e^{-\lambda t} (\lambda f (\Pi^\epsilon(Z^\epsilon(t)))-\bar{L}f (\Pi^\epsilon(Z^\epsilon(t)))) dt - f (\Pi^\epsilon(Z^\epsilon(0))) = 0.
	\end{aligned}	
	\end{equation*}
\end{Lemma}

The idea of the proof is similar to \cite{SDE2}, Theorem 4.1.
We consider first the following sequence of stopping times
\begin{equation*}
    \sigma_n^{\epsilon,\delta} \coloneqq \inf \{ t \geq \tau_n^{\epsilon,\delta} : Z^\epsilon(t) \in \bigcap_{j=1}^{\abs{V}} G_{\epsilon,j}(\delta) \}
\end{equation*}
and
\begin{equation*}
    \tau_n^{\epsilon,\delta} \coloneqq \inf \{ t > \sigma_{n-1}^{\epsilon,\delta} : Z^\epsilon(t) \in \bigcup_{j=1}^{\abs{V}} C_{\epsilon,j}(\rho_j r_j(\epsilon)+3\epsilon) \},
\end{equation*}
with $\tau_0^{\epsilon,\delta}=0$.
    These stopping times allow us to divide the analysis into two parts: when the process is inside a cylindrical domain, and when it is near a ball.
    The behavior in the cylindrical region is governed by standard Brownian motion, making it relatively straightforward to analyze. 
    In contrast, the behavior near the ball requires a more delicate treatment involving the gluing conditions and the narrow escape problems, which are discussed in Sections \ref{sec exit} and \ref{sec exit exp}.
    
    An important observation from \cite{SDE2} is that, by analyzing the process in the cylindrical region, one obtains the following bounds
    \begin{equation*}
		\sum_{n=0}^\infty \mathbb{E}_z e^{ - \lambda \tau_n ^{\epsilon,\delta}} 
		= O_{\epsilon,\delta} \left( \frac{1}{\delta} \right), \ \ \ \
		\sum_{n=0}^\infty \mathbb{E}_z e^{ - \lambda \sigma_n^{\epsilon,\delta} } 
		= O_{\epsilon,\delta} \left( \frac{1}{\delta} \right).
	\end{equation*}
	
	While this estimate is sufficient for the small and intermediate ball cases, it becomes too coarse when dealing with large balls. 
	The reason is that, in the large ball regime, the process spends the majority of its time near the ball, and thus the contribution from the cylindrical region becomes negligible in comparison.
To obtain the necessary precision, we must go beyond the cylindrical analysis. In particular, we will utilize the sharper estimate provided in \eqref{exp est 2} to refine our bounds in the large ball case.
	Let us define
	\begin{equation*}
    \sigma_{\mathfrak{L},n}^{\epsilon,\delta} \coloneqq \inf \{ t \geq \tau_{\mathfrak{L},n}^{\epsilon,\delta} : Z^\epsilon(t) \in \bigcap_{j \in \mathfrak{L}} G_{\epsilon,j}(\delta)  \}
\end{equation*}
and
\begin{equation*}
    \tau_{\mathfrak{L},n}^{\epsilon,\delta} \coloneqq \inf \{ t > \sigma_{\mathfrak{L},n-1}^{\epsilon,\delta} : Z^\epsilon(t) \in \bigcup_{j \in \mathfrak{L}} C_{\epsilon,j}(\rho_j r_j(\epsilon)+3\epsilon) \}
\end{equation*}
with $\tau_{\mathfrak{L},0}^{\epsilon,\delta} \coloneqq 0$.
By the strong Markov property, for $n \geq 1$,
\begin{equation*}
\begin{aligned}
	&\mathbb{E}_z e^{ - \lambda \sigma_{\mathfrak{L}, n}^{\epsilon,\delta} } 
	= \mathbb{E}_z e^{ - \lambda \left( \sigma_{\mathfrak{L}, n}^{\epsilon,\delta} - \tau_{\mathfrak{L}, n}^{\epsilon,\delta} \right) } 
	e^{ - \lambda \left( \tau_{\mathfrak{L}, n}^{\epsilon,\delta} - \sigma_{\mathfrak{L}, n-1}^{\epsilon,\delta} \right) } 
	e^{ - \lambda \sigma_{\mathfrak{L}, n-1}^{\epsilon,\delta} } \\
	 & = \mathbb{E}_z \left[ \left( \mathbb{E}_{Z^\epsilon(\tau_{\mathfrak{L}, n}^{\epsilon,\delta})}  e^{ - \lambda \sigma^{\epsilon,\delta}   }  \right)
	e^{ - \lambda \left( \tau_{\mathfrak{L}, n}^{\epsilon,\delta} - \sigma_{\mathfrak{L}, n-1}^{\epsilon,\delta} \right) }  
	e^{ - \lambda \sigma_{\mathfrak{L}, n-1}^{\epsilon,\delta} } \right] \\
	& \leq \max_{j \in \mathfrak{L}} \sup_{z' \in C_{\epsilon,j}(\rho_j r_j(\epsilon)+3\epsilon) } \left( \mathbb{E}_{z'}  e^{-\lambda \sigma^{\epsilon,\delta}} \right) 
	\mathbb{E}_z e^{ - \lambda \sigma_{\mathfrak{L}, n-1}^{\epsilon,\delta} } 
	 \leq \left( \max_{j \in \mathfrak{L}} \sup_{z' \in C_{\epsilon,j}(\rho_j r_j(\epsilon)+3\epsilon) } \mathbb{E}_{z'}  e^{-\lambda \sigma^{\epsilon,\delta}}  \right)^n.
\end{aligned}
\end{equation*}
Here, recall that $\sigma^{\epsilon,\delta}$ denotes the first hitting time of $\bigcup_{j=1}^{\abs{V}} C_{\epsilon,j}(\delta) $ and if $z \in \bigcup_{j \in \mathfrak{L}} C_{\epsilon,j}(\rho_j r_j (\epsilon) + 3\epsilon)$, then
\begin{equation*}
	\sigma_{\mathfrak{L},0}^{\epsilon,\delta} = \sigma^{\epsilon,\delta}.
\end{equation*}
Due to \eqref{exp est 2}, there exists $\epsilon_0 >0$ such that for every $\epsilon \in (0,\epsilon_0)$,
\begin{equation*}
	\max_{j \in \mathfrak{L}} \sup_{z' \in C_{\epsilon,j}(\rho_j r_j(\epsilon)+3\epsilon) } \mathbb{E}_{z'}  e^{-\lambda \sigma^{\epsilon,\delta}}  < \frac{1}{2}
\end{equation*}
and thus
\begin{equation*}
\begin{aligned}
	&\sum_{n=1}^\infty \mathbb{E}_z e^{ - \lambda \sigma_{\mathfrak{L}, n}^{\epsilon,\delta} }
	\leq \sum_{n=0}^\infty \left( \max_{j \in \mathfrak{L}} \sup_{z' \in C_{\epsilon,j}(\rho_j r_j(\epsilon)+3\epsilon) } \mathbb{E}_{z'}  e^{-\lambda \sigma^{\epsilon,\delta}}   \right)^n \\
	&= \left( 1- \max_{j \in \mathfrak{L}} \sup_{z' \in C_{\epsilon,j}(\rho_j r_j(\epsilon)+3\epsilon) } \mathbb{E}_{z'}  e^{-\lambda \sigma^{\epsilon,\delta}}  \right)^{-1}.
\end{aligned}
\end{equation*}
By \eqref{exp est 2} again, we obtain
\begin{equation*}
	\lim_{\epsilon \to 0} \sum_{n=1}^\infty \mathbb{E}_z e^{ - \lambda \sigma_{\mathfrak{L}, n}^{\epsilon,\delta} }
	\leq \left( 1- \lim_{\epsilon \to 0} \max_{j \in \mathfrak{L}} \sup_{z' \in C_{\epsilon,j}(\rho_j r_j(\epsilon)+3\epsilon) } \mathbb{E}_{z'}  e^{-\lambda \sigma^{\epsilon,\delta}}  \right)^{-1}
	=1,
\end{equation*}
which enables us to conclude that
\begin{equation}\label{thm sum exit large}
		\sum_{n=0}^\infty \mathbb{E}_z e^{ - \lambda \tau_{\mathfrak{L},n}^{\epsilon,\delta} } 
		= O_{\epsilon,\delta} \left( 1 \right), \ \ \ \
		\sum_{n=0}^\infty \mathbb{E}_z e^{ - \lambda \sigma_{\mathfrak{L},n}^{\epsilon,\delta} } 
		= O_{\epsilon,\delta} \left( 1 \right).
	\end{equation}
	
	Finally, let us define	
	\begin{equation*}
    \sigma_{\mathfrak{L}^c,n}^{\epsilon,\delta} \coloneqq \inf \{ t \geq \tau_{\mathfrak{L}^c,n}^{\epsilon,\delta} : Z^\epsilon(t) \in \bigcap_{j \notin \mathfrak{L}} G_{\epsilon,j}(\delta) \},
\end{equation*}
and
\begin{equation*}
    \tau_{\mathfrak{L}^c,n}^{\epsilon,\delta} \coloneqq \inf \{ t > \sigma_{\mathfrak{L}^c,n-1}^{\epsilon,\delta} : Z^\epsilon(t) \in \bigcup_{j  \notin \mathfrak{L}} C_{\epsilon,j}(\rho_j r_j(\epsilon)+3\epsilon) \},
\end{equation*}
with $\tau_{\mathfrak{L}^c,0}^{\epsilon,\delta} \coloneqq 0$.
We still have the following bounds
 \begin{equation}\label{thm sum exit not large}
		\sum_{n=0}^\infty \mathbb{E}_z e^{ - \lambda \tau_{\mathfrak{L}^c,n}^{\epsilon,\delta} } 
		= O_{\epsilon,\delta} \left( \frac{1}{\delta} \right), \ \ \ \
		\sum_{n=0}^\infty \mathbb{E}_z e^{ - \lambda \sigma_{\mathfrak{L}^c,n}^{\epsilon,\delta} } 
		= O_{\epsilon,\delta} \left( \frac{1}{\delta} \right).
	\end{equation}

\begin{remark}\label{proj prop}
    In the definition of $C_{\epsilon,j}(\delta)$, we use the distance between $O_j$ and $\Pi(z)$ rather than $\Pi^\epsilon(z)$.
    However, when $\delta \geq \rho_j r_j(\epsilon)+2\epsilon$, the two definitions are equivalent by \eqref{proj 1}.
    As a result, the same exit properties hold for $\Pi^\epsilon(Z^\epsilon(t))$.
    For example, if the process starts in $\bigcap_{j=1}^{\abs{V}} G_{\epsilon,j}(\delta)$, then for all $n \geq 1$,
    \begin{equation*}
        d(\Pi^\epsilon(Z^\epsilon(\tau_{\mathfrak{L}^c,n}^{\epsilon,\delta})),O_j) = \rho_j r_j(\epsilon)+3\epsilon, \quad \text{ for some } j \notin \mathfrak{L}
    \end{equation*}
    and if the process starts in $C_{\epsilon,j}(\rho_j r_j(\epsilon)+3\epsilon)$, then, 
    \begin{equation*}
        d_{\Gamma}(\Pi^\epsilon(Z^\epsilon(t)),O_j) \leq \delta, \quad t \leq \sigma^{\epsilon,\delta},
    \end{equation*}
    where $G_{\epsilon,j}(\delta) = \{ z \in G_\epsilon: d_\Gamma(\Pi(z),O_j) \geq \delta \}$ and $\sigma^{\epsilon,\delta}$ is the first hitting time of $\bigcup_{j=1}^{\abs{V}} C_{\epsilon,j}(\delta)$ as defined in Section \ref{sec exit}.
    We will use this observation several times throughout the proof.
\end{remark}

\begin{proof}[Proof of Lemma \ref{conv dif abs}]
    We show the convergence for $z \in \bigcap_{j=1}^{\abs{V}} G_{\epsilon,j}(\delta)$. 
    The proof of $z \in B_{\epsilon,j}(\delta)$ for some $j\in \{1,\cdots,\abs{V}\}$ can be adapted without substantial changes.
	Note that
    \begin{equation*}
    \begin{aligned}
        &\mathbb{E}_z \int_0^\infty e^{-\lambda t} (\lambda f (\Pi^\epsilon(Z^\epsilon(t)))-\bar{L}f (\Pi^\epsilon(Z^\epsilon(t)))) dt - f (\Pi^\epsilon(Z^\epsilon(0))) \\
        &=
        \sum_{n=0}^\infty \mathbb{E}_z \Big [ e^{-\lambda \tau_{n+1}^{\epsilon,\delta}} f (\Pi^\epsilon(Z^\epsilon(\tau_{n+1}^{\epsilon,\delta}))) - e^{-\lambda \sigma_n^{\epsilon,\delta}} f (\Pi^\epsilon(Z^\epsilon(\sigma_n^{\epsilon,\delta}))) \\
        & \ \ \ \ \ \ \ \ \ \ \ \ \ \ \ \ \ \ \ \ \ \ \ \ \ \ \ \ \ \ \ \
         + \int_{\sigma_n^{\epsilon,\delta}}^{\tau_{n+1}^{\epsilon,\delta}} e^{-\lambda t} (\lambda f (\Pi^\epsilon(Z^\epsilon(t)))-\bar{L}f (\Pi^\epsilon(Z^\epsilon(t)))) dt
        \Big ] \\
        & \ \ +
        \sum_{n=0}^\infty \mathbb{E}_z \Big [ e^{-\lambda \sigma_{\mathfrak{L}^c,n}^{\epsilon,\delta}} f (\Pi^\epsilon(Z^\epsilon(\sigma_{\mathfrak{L}^c,n}^{\epsilon,\delta}))) - e^{-\lambda \tau_{\mathfrak{L}^c,n}^{\epsilon,\delta}} f (\Pi^\epsilon(Z^\epsilon(\tau_{\mathfrak{L}^c,n}^{\epsilon,\delta}))) \\
         & \ \ \ \ \ \ \ \ \ \ \ \ \ \ \ \ \ \ \ \ \ \ \ \ \ \ \ \ \ \ \ \
         +\int_{\tau_{\mathfrak{L}^c,n}^{\epsilon,\delta}}^{\sigma_{\mathfrak{L}^c,n}^{\epsilon,\delta}} e^{-\lambda t} (\lambda f (\Pi^\epsilon(Z^\epsilon(t)))-\bar{L}f (\Pi^\epsilon(Z^\epsilon(t)))) dt
        \Big ] \\
        & \ \ +
        \sum_{n=0}^\infty \mathbb{E}_z \Big [ e^{-\lambda \sigma_{\mathfrak{L},n}^{\epsilon,\delta}} f (\Pi^\epsilon(Z^\epsilon(\sigma_{\mathfrak{L},n}^{\epsilon,\delta}))) - e^{-\lambda \tau_{\mathfrak{L},n}^{\epsilon,\delta}} f (\Pi^\epsilon(Z^\epsilon(\tau_{\mathfrak{L},n}^{\epsilon,\delta}))) \\
         & \ \ \ \ \ \ \ \ \ \ \ \ \ \ \ \ \ \ \ \ \ \ \ \ \ \ \ \ \ \ \ \
         +\int_{\tau_{\mathfrak{L},n}^{\epsilon,\delta}}^{\sigma_{\mathfrak{L},n}^{\epsilon,\delta}} e^{-\lambda t} (\lambda f (\Pi^\epsilon(Z^\epsilon(t)))-\bar{L}f (\Pi^\epsilon(Z^\epsilon(t)))) dt
        \Big ] \\
        & \eqqcolon I_1^{\epsilon,\delta} + I_2^{\epsilon,\delta} + I_3^{\epsilon,\delta}.
        \end{aligned}
    \end{equation*}
    For $I_1^{\epsilon,\delta}$, the process $\Pi^\epsilon(Z^\epsilon(t))$ is a one dimensional Brownian motion provided, we have
    \begin{equation*}
    \begin{aligned}
    		I_1^{\epsilon,\delta} = 0.
    \end{aligned}
    \end{equation*}
    
    The main difficultly is to estimate $I_2^{\epsilon,\delta}$ and $I_3^{\epsilon,\delta}$. 
    By the strong Markov property, we have
    \begin{equation*}
        I_2^{\epsilon,\delta} = \sum_{n=0}^\infty \mathbb{E}_z \left[ 
        e^{-\lambda \tau_{\mathfrak{L}^c,n}^{\epsilon,\delta}} \psi^{\epsilon,\delta}(Z^\epsilon ({\tau_{\mathfrak{L}^c,n}^{\epsilon,\delta}}))
        \right],
    \end{equation*}
    where 
    \begin{equation}\label{psi dif}
    \begin{aligned}
        & \psi^{\epsilon,\delta}(z') \\
        & = \mathbb{E}_{z'} \left[ 
        \int_0^{\sigma^{\epsilon,\delta}} e^{-\lambda t} (\lambda f(\Pi^\epsilon(Z^\epsilon(t)))-\bar{L} f(\Pi^\epsilon(Z^\epsilon(t))) ) dt
        +e^{-\lambda \sigma^{\epsilon,\delta}} f(\Pi^\epsilon(Z^\epsilon(\sigma^{\epsilon,\delta}))) \right] 
         - f(\Pi^\epsilon( z')) \\
        & = \mathbb{E}_{z'}
        \left[ \int_0^{\sigma^{\epsilon,\delta}} \lambda e^{-\lambda t} \left[ f(\Pi^\epsilon(Z^\epsilon(t))) - f(\Pi^\epsilon(Z^\epsilon(\sigma^{\epsilon,\delta}))) \right] dt \right ] \\ 
        & \ \ \ \ \ + \mathbb{E}_{z'} \left[  f(\Pi^\epsilon(Z^\epsilon(\sigma^{\epsilon,\delta}))) -f(\Pi^\epsilon( z')) \right] 
         +\mathbb{E}_{z'} \left[ \int_0^{\sigma^{\epsilon,\delta}} -e^{-\lambda t} \bar{L} f(\Pi^\epsilon(Z^\epsilon(t))) dt \right ] \\
        & \ \ \ \ \ \ \ \ \ \ \ \ \ \ \ \eqqcolon  \psi_{1}^{\epsilon,\delta}(z') + \psi_{2}^{\epsilon,\delta}(z') + \psi_{3}^{\epsilon,\delta}(z').
    \end{aligned}
    \end{equation}
Due to Remark \ref{proj prop}, $d(\Pi^\epsilon(Z^\epsilon(\tau_{\mathfrak{L}^c,n}^{\epsilon,\delta})),O_j) = \rho_j r_j(\epsilon)+3\epsilon$, for some $j \notin \mathfrak{L}$.
As a result, in the study of $I_{2,i}^{\epsilon,\delta}$ for $i=1,2,3$, we will consider $z' \in C_{\epsilon,j}(\rho_j r_j(\epsilon)+3\epsilon)$, for some $j \notin \mathfrak{L}$. 
First, by the continuity of $f$ and Proposition \ref{small exit est 2},
\begin{equation}\label{diff est 1}
     \psi_{1}^{\epsilon,\delta}(z')
    = \mathbb{E}_{z'} \sigma^{\epsilon,\delta} o_{\epsilon,\delta}(1) 
    = o_{\epsilon,\delta}(\delta).
\end{equation}
Next, thanks to Proposition \ref{exit place est}, we have
\begin{equation}
\begin{aligned}
    &\psi_{2}^{\epsilon,\delta}(z')= \mathbb{E}_{z'} \left[  f(\Pi^\epsilon(Z^\epsilon(\sigma^{\epsilon,\delta}))) -f(O_j) +f(O_j) -f(\Pi^\epsilon( z')) \right] \\
    &= \sum_{k:I_k \sim O_j} \Big( \mathbb{E}_{z'} \left[ f(\Pi^\epsilon(Z^\epsilon(\sigma^{\epsilon,\delta}))) - f(O_j) + f(O_j) - f(\Pi^\epsilon( z'))   \middle| Z^\epsilon(\sigma^{\epsilon,\delta}) \in  C_{\epsilon,j}^k(\delta) \right ] \\
     &  \qquad \times \mathbb{P}_{z'}(Z^\epsilon(\sigma^{\epsilon,\delta}) \in C_{\epsilon,j}^k(\delta)) \Big) \\
    &=  \sum_{k:I_k \sim O_j} \left[ \frac{\partial f}{ \partial x_k}(O_j) \delta + o_{\epsilon,\delta}(\delta) \right] 
    (p_{j,k}+o_{\epsilon,\delta}(1))   
    = \sum_{k:I_k \sim O_j} p_{j,k}\frac{\partial f}{ \partial x_k}(O_j) \delta + o_{\epsilon,\delta}(\delta).
\end{aligned}
\end{equation}
Finally,
\begin{equation*}
\begin{aligned}
    \psi_{3}^{\epsilon,\delta}(z') & = \mathbb{E}_{z'} \left[ 
        \int_0^{\sigma^{\epsilon,\delta}} e^{-\lambda t} (\bar{L}f(O_j) - \bar{L} f(\Pi^\epsilon(Z^\epsilon(t)))  ) dt \right ] \\
    & \ \ \ +\mathbb{E}_{z'} \left[ 
        \int_0^{\sigma^{\epsilon,\delta}} \bar{L}f(O_j) -e^{-\lambda t} \bar{L}f(O_j)  dt \right ]
         -\mathbb{E}_{z'}  (\sigma^{\epsilon,\delta} ) \bar{L}f(O_j)   \\
    & \eqqcolon  J_1^{\epsilon,\delta}(z') + J_2^{\epsilon,\delta}(z') + J_3^{\epsilon,\delta}(z').
\end{aligned}
\end{equation*}
For $J_1^{\epsilon,\delta}(z')$, similar to the previous case, by the continuity of $\bar{L}f$ and Proposition \ref{small exit est 2}, we have
\begin{equation}
    J_1^{\epsilon,\delta}(z') 
    = (\mathbb{E}_{z'} \sigma^{\epsilon,\delta}) o_{\epsilon,\delta}(1)
    = o_{\epsilon,\delta}(\delta).
\end{equation}
For $J_2^{\epsilon,\delta}(z')$, due to Lemma \ref{inter exit est 1}, 
\begin{equation}
    \abs{J_2^{\epsilon,\delta}(z') }
    \leq \abs{\bar{L}f(O_j)} \mathbb{E}_{z'} \int_0^{\sigma^{\epsilon,\delta}} (1 -e^{-\lambda t} ) dt 
    \leq c  \abs{\bar{L}f(O_j)}  \mathbb{E}_{z'} (\sigma^{\epsilon,\delta})^2 = O_{\epsilon,\delta}(\delta^2).
\end{equation}
For $J_3^{\epsilon,\delta}(z')$, by Proposition \ref{small exit est 2},
\begin{equation}\label{diff est 2}
    J_3^{\epsilon,\delta}(z') = -\left(\alpha_j +o_{\epsilon,\delta}(1) \right) \delta \bar{L}f (O_j)
\end{equation}
Combining \eqref{diff est 1}--\eqref{diff est 2}, thanks to the gluing condition \eqref{abs glu}, if $ Z^\epsilon(\tau_{\mathfrak{L}^c,n}^{\epsilon,\delta}) \in C_{\epsilon,j}(\rho_j r_j(\epsilon)+3\epsilon)$ for $j \notin \mathfrak{L}$, we have
\begin{equation*}
\displaystyle
     \psi^{\epsilon,\delta}(Z^{\epsilon}(\tau_{\mathfrak{L}^c,n}^{\epsilon,\delta})) = \delta \left( \sum_{k:I_k \sim O_j} p_{j,k}\frac{\partial f}{ \partial x_k}(O_j)   - 
    \alpha_j  \bar{L} f(O_j) \right) 
    + o_{\epsilon,\delta}(\delta)
    =o_{\epsilon,\delta}(\delta).
\end{equation*}
Now \eqref{thm sum exit not large} implies that 
\begin{equation*}
\begin{aligned}
    &I_2^{\epsilon,\delta} 
    = \sum_{n=0}^\infty \sum_{j \notin \mathfrak{L}} \mathbb{E}_z \left[  e^{-\lambda \tau_{\mathfrak{L}^c,n}^{\epsilon,\delta}}  \mathbbm{1}_{ \{ Z^{\epsilon}(\tau_{\mathfrak{L}^c,n}^{\epsilon,\delta}) \in C_{\epsilon,j}(\rho_j r_j(\epsilon)+3\epsilon) \} } \psi^{\epsilon,\delta}(Z^{\epsilon}(\tau_{\mathfrak{L}^c,n}^{\epsilon,\delta})) \right] \\
    &\leq \sum_{n=0}^\infty \mathbb{E}_z e^{-\lambda \tau_{\mathfrak{L}^c,n}^{\epsilon,\delta}} 
    \cdot o_{\epsilon,\delta}(\delta)
    = O_{\epsilon,\delta}\left(\frac{1}{\delta} \right) \cdot o_{\epsilon,\delta}(\delta) =o_{\epsilon,\delta}(1).
\end{aligned}
\end{equation*}

For $I_3^{\epsilon,\delta}$, we have similarly,
    \begin{equation*}
        I_3^{\epsilon,\delta} = \sum_{n=0}^\infty \mathbb{E}_z \left[ 
        e^{-\lambda \tau_{\mathfrak{L},n}^{\epsilon,\delta}} \psi^{\epsilon,\delta}(Z^\epsilon ({\tau_{\mathfrak{L},n}^{\epsilon,\delta}}))
        \right],
    \end{equation*}
    where $\psi^{\epsilon,\delta}$ is given by \eqref{psi dif}.
    If $z' \in C_{\epsilon,j}(\rho_j r_j(\epsilon)+3\epsilon)$, for some $j \in \mathfrak{L}$, 
due to the continuity of $f$,
\begin{equation}\label{diff est 3}
     \psi_{1}^{\epsilon,\delta}(z')
    = o_{\epsilon,\delta}(1) \mathbb{E}_{z'}
        \left[ \int_0^{\sigma^{\epsilon,\delta}} \lambda e^{-\lambda t} dt \right ]
        =o_{\epsilon,\delta}(1) \left(1-e^{-\lambda \sigma^{\epsilon,\delta}} \right)
    = o_{\epsilon,\delta}(1)
\end{equation}
and
\begin{equation}\label{diff est 4}
\begin{aligned}
    &\psi_{2}^{\epsilon,\delta}(z')= o_{\epsilon,\delta}(1).
\end{aligned}
\end{equation}
Moreover, thanks to the gluing condition \eqref{abs glu}, the continuity of $\bar{L}f$,
\begin{equation}\label{diff est 5}
\begin{aligned}
	&\psi_{3}^{\epsilon,\delta}(z') = \mathbb{E}_{z'} \left[ \int_0^{\sigma^{\epsilon,\delta}} -e^{-\lambda t} \bar{L} f(\Pi^\epsilon(Z^\epsilon(t))) dt \right ] 
	=\mathbb{E}_{z'} \left[ \int_0^{\sigma^{\epsilon,\delta}} e^{-\lambda t} \left( \bar{L}f(O_j) - \bar{L} f(\Pi^\epsilon(Z^\epsilon(t))) \right) dt \right ] \\
	& =o_{\epsilon,\delta}(1)\mathbb{E}_{z'} \left[ \int_0^{\sigma^{\epsilon,\delta}} e^{-\lambda t} dt \right ] =o_{\epsilon,\delta}(1) \frac{1}{\lambda} \left(1-e^{-\lambda \sigma^{\epsilon,\delta}} \right) = o_{\epsilon,\delta}(1).
\end{aligned}
\end{equation}
Finally, combining \eqref{diff est 3}, \eqref{diff est 4}, \eqref{diff est 5}, and \eqref{thm sum exit large}, we conclude that 
\begin{equation*}
\begin{aligned}
    &I_3^{\epsilon,\delta} 
    = \sum_{n=0}^\infty \sum_{j \in \mathfrak{L}} \mathbb{E}_z \left[  e^{-\lambda \tau_{\mathfrak{L},n}^{\epsilon,\delta}}  \mathbbm{1}_{ \{ Z^{\epsilon}(\tau_{\mathfrak{L},n}^{\epsilon,\delta}) \in C_{\epsilon,j}(\rho_j r_j(\epsilon)+3\epsilon) \} } \psi^{\epsilon,\delta}(Z^{\epsilon}(\tau_{\mathfrak{L},n}^{\epsilon,\delta})) \right] \\
    & \leq \sum_{n=0}^\infty \mathbb{E}_z e^{-\lambda \tau_{\mathfrak{L},n}^{\epsilon,\delta}} 
    \cdot o_{\epsilon,\delta}(1)
    = o_{\epsilon,\delta}(1).
\end{aligned}
\end{equation*}
\end{proof}

\section{Remarks}\label{sec rem}
A natural question is whether the graph can be perturbed in other ways.
For perturbations near the vertices, one could replace the spherical neighborhoods with more general smooth domains.
In fact, the only part of the analysis that depends on the specific geometry of these neighborhoods is in Section \ref{sec NET ball to connect}.
The work in \cite{NET4} extensively studies the narrow escape problem in general smooth domains with a single escape region.
It appears feasible to extend these results to multiple escape regions by carefully estimating the associated algebraic systems.

In contrast, perturbations along the edges pose greater challenges.
Our method crucially relies on separation of variables away from the vertices, effectively reducing the analysis to one-dimensional Brownian motion along the graph edges.
To handle more general geometries along the edges, one would need to approximate Brownian motion in narrow tubes by one-dimensional processes, as in \cite{CF}.
However, this generalization is significantly more intricate than modifying the neighborhoods around the vertices.

The restriction to low dimensions primarily arises from Section \ref{sec NET ball to connect}, which addresses the narrow escape problem in the unit ball.
The argument in \cite{NET4} relies on special functions that are explicit for $d=3$ but become less tractable in higher dimensions, and therefore does not directly extend to the higher dimensional setting.
Since Brownian motion (and the Laplace operator) behaves qualitatively similarly for  $d \geq 3$, we expect analogous results to hold in higher dimensions as well. 
It appears that the results of \cite{C+26} could be used to obtain a similar bound in Section \ref{sec NET ball to connect}. 
However, the setting considered there differs slightly from that in Section \ref{sec NET ball to connect}, and some additional work would be needed to adapt their results to our setting. 
As the narrow escape problem in smooth domains is not the focus of this paper, we do not pursue this direction further.

Finally, it would be interesting to study narrow escape problems in domains with bottlenecks having different neck lengths and different scalings of the tube widths $\lambda_k(\epsilon)$.
The case of different neck lengths is considered in our companion paper \cite{Hsu25+}. 
We plan to investigate different scalings of the tube widths, together with the higher dimensional setting, in future work.

\begin{appendices}

\section{Escape problems}\label{sec NET}
In this section, we review some fundamental results from escape problems, as presented in \cite{NET}, \cite{NETbook} and \cite{NETsecond}. 
Our approach largely follows the treatment in \cite{NETsecond}.

\subsection{First moment of the hitting time}\label{sec MFPT}

Let $\tilde{Z}(t)$ denote the scaled reflected Brownian motion on a bounded smooth domain $D$, and let $\partial D_a $ be a measurable subset of the boundary $\partial D$.
We are interested in the expected first hitting time of $\partial D_a$ from a point $z \in D$; that is,
$\tilde{\mathbb{E}}_z \left[ \tau_ {\partial D_a}\right]$, where $\tau_ {\partial D_a}$ is the first hitting time of $\partial D_a$.

For $z \in D$, we denote by $\rho(z,\cdot,t)$ the transition probability at time $t$ for the particle starting at $z$; that is, $\rho(z,\cdot,t)$ satisfies the Kolmogorov equation
\begin{equation*}
\displaystyle
    \begin{dcases}
        \partial_t \rho(z,z',t) =  \Delta \rho(z,z',t), \ \ \ (z',t) \in D \times (0,\infty), \\
        \rho(z,z',t) = 0, \ \ \ (z',t) \in \partial D_a \times (0,\infty), \\
        \frac{\partial \rho}{ \partial \nu} (z,z',t) = 0, \ \ \ (z',t) \in \partial D \setminus \partial D_a \times (0,\infty), \\
        \rho(z,z',0) = \delta(z-z'), \ \ \ z' \in D.
    \end{dcases}
\end{equation*}
Let $\lambda$ be the principal eigenvalue of $\Delta$ with mixed Neumann (on $\partial D \setminus \partial D_a$) and Dirichlet boundary condition (on $\partial D_a$), then
\begin{equation*}
	0 \leq \rho(z,z',t) \leq ce^{-\lambda t}, \ \ z \in D, \ z' \in \bar{D}, \  t \geq 1,
\end{equation*}
for some constant $c>0$ .
Since $\lambda > 0$, $\rho$ decays exponentially in time.
We then introduce the function
\begin{equation*}
	G(z,z') \coloneqq \int_0^\infty \rho(z,z',t) dt, \ \ \ (z,z') \in D \times \bar{D}.
\end{equation*}
One can show that $G$ is the Green function of the Laplace operator $\Delta$ with mixed boundary condition
\begin{equation*}
\begin{aligned}
    \begin{dcases}
    \displaystyle
        \Delta G(z,\cdot) = -\delta_z \ \ \ \text{in} \ \ D, \\
        G(z,\cdot) = 0 \ \ \ \text{on} \ \ \partial D_a, \\
        \frac{\partial G(z,\cdot)}{\partial \nu}=0 \ \ \  \text{on} \ \ \partial D \setminus \partial D_a.
    \end{dcases}
\end{aligned}
\end{equation*}
The probability that a particle starting form $z$ didn't hit $\partial D_a$ before time $t$ is
\begin{equation*}
	p(z,t) \coloneqq \tilde{\mathbb{P}}_z (\tau_{\partial D_a} > t)  = \int_D \rho(z,z',t)dz'.
\end{equation*}
We thus have
\begin{equation*}
	v(z) \coloneqq \tilde{\mathbb{E}}_z (\tau_{\partial D_a}) 
	= \int_0^\infty p(z,t) dt 
	= \int_0^\infty \int_D \rho(z,z',t)dz' dt 
	= \int_D G(z,z') dz'.
\end{equation*}
It can be shown that $v(z)$ satisfies the following mixed boundary problem
	\begin{equation*}
\begin{aligned}
    \begin{dcases}
    \displaystyle
        \Delta v(z) = -1 \ \ \ \text{for} \ \ z \in D, \\
        v(z)=0 \ \ \ \text{for} \ \ z \in \partial D_a, \\
        \frac{\partial v(z)}{\partial \nu }=0 \ \ \  \text{for} \ \ z \in \partial D \setminus \partial D_a.
    \end{dcases}
\end{aligned}
\end{equation*}

\subsection{Second moment of the hitting time}\label{sec SMFPT}
We are interested in $v_2(z) \coloneqq \tilde{\mathbb{E}}_z (\tau_{\partial D_a}^2) $. To proceed, we introduce
\begin{equation*}
	\rho_1(z,z',t) = \int_t^\infty \rho(z,z',s) ds
\end{equation*}
then for $z \in D$, $t \geq 0$, $\rho_1(z,\cdot,t)$ satisfies
\begin{equation*}
\begin{aligned}
    \begin{dcases}
    \displaystyle
    		 \partial_t \rho_{1}(z,\cdot,t) =\Delta \rho_1(z,\cdot,t) \ \ \ \text{in} \ \ D, \\
        \rho_1(z,\cdot,t) = 0 \ \ \ \text{on} \ \ \partial D_a, \\
        \frac{\partial \rho_1(z,\cdot,t)}{\partial \nu}=0 \ \ \  \text{on} \ \ \partial D \setminus \partial D_a, \\
        \rho_1(z,\cdot, 0) = G(z,\cdot) \ \ \ \text{on} \ \ D.
    \end{dcases}
\end{aligned}
\end{equation*}
The following theorem can be found in \cite{NETsecond}, the proof of Theorem 1.
\begin{Theorem}\label{SMFPT Main}
For $z \in D$,
	\begin{equation*}
		v_2(z) = \int_D 2v(z') G(z,z') dz'.
	\end{equation*}
	In particular,
	\begin{equation*}
		v_2(z) \leq 2 \left( \sup_{z \in D} v(z) \right)^2.
	\end{equation*}
\end{Theorem}
\begin{proof}
We have
\begin{equation*}
\begin{aligned}
	&v_2(z)=\int_0^\infty 2t p(z,t) dt
	=\int_0^\infty \int_D 2t \rho(z,z',t) dz' dt
	=-\int_D \int_0^\infty 2t \partial_t \rho_{1}(z,z',t) dt dz' \\
	& =\int_D \int_0^\infty 2 \rho_{1}(z,z',t) dt dz'
    =-\int_0^\infty \int_D  2 \rho_{1}(z,z',t) \Delta v(z') dz' dt \\
	&
	=-\int_0^\infty \int_D  2 \Delta_{z'} \rho_{1}(z,z',t)  v(z') dz' dt  = - \int_D \int_0^\infty 2 \partial_t \rho_{1}(z,z',t)  v(z') dt dz' \\
	& = \int_D  2 \rho_{1}(z,z',0)  v(z') dz'
	= \int_D  2 G(z,z')  v(z') dz'.
\end{aligned}
\end{equation*}
\end{proof}

\section{Proof of \texorpdfstring{$\eqref{reg per 1}$}{Lg} and \texorpdfstring{$\eqref{reg per 2}$}{Lg}}\label{App reg per}

Consider first $\Omega_{4,\epsilon}^k$ and the change of variables $z \mapsto \frac{z-z_k^*}{\epsilon}$.
Let $\tilde{\Omega}_{4,\epsilon}^k$ be the image of $\Omega_{4,\epsilon}^k$ through the transformation.
Similarly, define $\tilde{\Gamma}_{i,\epsilon}^k$ be the image of $\Gamma_{i,\epsilon}^k$ through the transformation for $i=1,2,3$. 
Moreover, define $\tilde{\Gamma}_{0,\epsilon}^k$ be the image of $C_\epsilon^k(\rho r(\epsilon)+2\epsilon) $ through the transformation.
Define 
\begin{equation*}
\tilde{Z}^\epsilon_k(t) \coloneqq \frac{1}{\epsilon}(Z^\epsilon(\epsilon^2 t)-z_k^*)
\end{equation*}
and let $\tilde{\mathbb{E}}_z^{k}$ be the corresponding expectation when the process starts at $z$.
Moreover, let $\tilde{p}^k_\epsilon$ be the probability of successfully escaping to $\tilde{\Gamma}_{3,\epsilon}^k$ (instead of $\tilde{\Gamma}_{2,\epsilon}^k$) and $\tilde{\sigma}^{2,\epsilon}_k$ be the first hitting time of $\tilde{\Gamma}_{2,\epsilon}^k \cup \tilde{\Gamma}_{3,\epsilon}^k$.

We first establish \eqref{reg per 3}, so that \eqref{reg per 1} follows.
We have known that $\tilde{p}_\epsilon^k (z)$ solves the following PDE
	\begin{equation*}
\begin{aligned}
    \begin{dcases}
    \displaystyle
        \Delta \tilde{p}_\epsilon^k (z) = 0 \ \ \ \text{for} \ \ z \in \tilde{\Omega}_{4,\epsilon}^k, \\
        \tilde{p}_\epsilon^k (z)=0 \ \ \ \text{for} \ \ z \in \tilde{\Gamma}_{2,\epsilon}^k, \\
        \tilde{p}_\epsilon^k (z)=1 \ \ \ \text{for} \ \ z \in \tilde{\Gamma}_{3,\epsilon}^k, \\
        \frac{\partial \tilde{p}_\epsilon^k (z)}{\partial \nu_\epsilon}=0 \ \ \  \text{for} \ \ z \in \partial \tilde{\Omega}_{4,\epsilon}^k \setminus \left( \tilde{\Gamma}_{2,\epsilon}^k \cup \tilde{\Gamma}_{3,\epsilon}^k \right).
    \end{dcases}
\end{aligned}
\end{equation*}
First, for $z \in \tilde{\Gamma}_{0,\epsilon}^k$, let $\Gamma_{0,\epsilon,var}^k \coloneqq C_\epsilon^k \left( \rho r(\epsilon) + \frac{7}{4}\epsilon \right)$ and let $\tilde{\Gamma}_{0,\epsilon,var}^k$ be the image of $\Gamma_{0,\epsilon,var}^k$ under the transformation $z \mapsto \frac{z-z_k^*}{\epsilon}$, we have
\begin{equation*}
	\tilde{p}_\epsilon^k(z) \geq \tilde{\mathbb{P}}^k_z \left( \tilde{Z}^\epsilon_k \text{ hitting $\tilde{\Gamma}^k_{3,\epsilon}$ before $\tilde{\Gamma}_{0,\epsilon,var}^k$} \right) = \frac{1}{5}.
\end{equation*}
We next establish a similar lower bound for $z \in \tilde{\Gamma}_{1,\epsilon}^k$ by proving the following lemma:
\begin{Lemma}\label{lem reg har}
There exists $c_{k,1} > 0$ independent of $\epsilon$ such that 
\begin{equation}\label{reg har}
	\frac{1}{5} \leq  c_{k,1} \inf_{z \in \tilde{\Gamma}_{1,\epsilon}^k} \tilde{p}_\epsilon^k(z).
\end{equation}
\end{Lemma}
\begin{proof}
In fact, we will prove 
\begin{equation*}
	\frac{1}{5} 
	\leq \sup_{z } \tilde{p}_\epsilon^k(z) 
	\leq  c_{k,1} \inf_{z } \tilde{p}_\epsilon^k(z),
\end{equation*}
where the supremum and infimum are taken over the domain
\begin{equation*}
 \tilde{\Gamma}_{1,\epsilon}^k \cup \left \{z \in \tilde{\Omega}_{4,\epsilon}^k : d(z, \partial \tilde{\Omega}_{4,\epsilon}^k ) \geq \frac{\eta'}{2} \right\}
 \end{equation*}
 for some $\eta' > 0$.
The statement is an analogue of the Harnack inequality that remains valid up to the Neumann boundary.
For $d=3$, this has been done in \cite{BH}, Theorem 3.9.
For $d=2$, the argument follows similarly to the reasoning in \cite{BH}, Remark 3.10.

More precisely, consider first the case $d=2$. 
Let $z_k^+$ and $z_k^-$ be the only two points lies in $\tilde{\Gamma}_{1,\epsilon}^k \cap \partial \tilde{\Omega}_{4,\epsilon}^k$.
Fix $\eta >0$ small enough and define $A_{\epsilon}^k(z_k^+, \eta) = B(z_k^+,\eta) \cap \tilde{\Omega}_{4,\epsilon}^k$.
After a suitable rotation and translation, we may assume coordinates $z=(z',h^+) \in \mathbb{R} \times \mathbb{R}$, where the domain lies above a smooth curve $h^+ = F_\epsilon(z')$ with $F_\epsilon : \mathbb{R} \to \mathbb{R}$ uniformly Lipschitz in $\epsilon$ (see Section \ref{sec narrow tube}).
Define a flattening map $\varphi_\epsilon(z',h^+) = (z',h^+-F_\epsilon(z')) \coloneqq (z',\tilde{h}_\epsilon^+)$, which straightens the boundary. 
Since 
\begin{equation*}
    D \varphi_\epsilon =
    \begin{bmatrix}
        1 & 0 \\
        -F_\epsilon' & 1
    \end{bmatrix},
\end{equation*}
$\tilde{p}_\epsilon^k \circ \varphi_\epsilon^{-1}$ satisfies, in the flatten coordinates,
\begin{equation*}
    div(B_\epsilon \nabla (\tilde{p}_\epsilon^k \circ \varphi_\epsilon^{-1}))(z') = 0, \quad 
    B_\epsilon(z') = 
    \begin{bmatrix}
        1 & -F_\epsilon'(z') \\
        -F_\epsilon'(z') & 1+ (F_\epsilon'(z'))^2
    \end{bmatrix},
\end{equation*}
a divergence-form equation with bounded measurable coefficients whose ellipticity constants depend only on $(F_\epsilon'(z'))^2$, and thus are bounded uniformly in $\epsilon$ by the uniform Lipschitz bound on $F_\epsilon$.
Reflecting across $\{  \tilde{h}_\epsilon^+ = 0\}$, the reflected function solves the same equation, by the Neumann boundary condition there.
By classical Harnack inequality (see e.g. \cite{PDEBook}), there exists $c_{k,2} > 0$ independent of $\epsilon$ such that
\begin{equation*}
	\sup_{z' \in B(\varphi_\epsilon(z_k^+), \frac{\eta}{2})} (\tilde{p}_\epsilon^k \circ \varphi_\epsilon^{-1}) (z') 
	\leq c_{k,2} \inf_{z' \in B(\varphi_\epsilon(z_k^+), \frac{\eta}{2})} (\tilde{p}_\epsilon^k \circ \varphi_\epsilon^{-1}) (z').
\end{equation*}
This can be done since the constant depends only on the dimension and the ellipticity.
Changing back to the original coordinates, this implies there exists $\eta'>0$, depending only on $\eta$ and the uniform Lipschitz bound on $F_\epsilon$, such that
\begin{equation}\label{Har int +}
	\sup_{z \in A_{\epsilon}^k(z_k^+, \eta')} \tilde{p}_\epsilon^k (z) \leq c_{k,2} \inf_{z \in A_{\epsilon}^k(z_k^+, \eta')} \tilde{p}_\epsilon^k (z).
\end{equation}

Apply a similar argument for $z_k^-$, there exists $c_{k,3} > 0$ such that
\begin{equation}\label{Har int -}
	\sup_{z \in A_{\epsilon}^k(z_k^-, \eta')} \tilde{p}_\epsilon^k (z) \leq c_{k,3} \inf_{z \in A_{\epsilon}^k(z_k^-, \eta')} \tilde{p}_\epsilon^k (z).
\end{equation}
Now note that the region 
\begin{equation}\label{region 1}
	\left \{z \in \tilde{\Omega}_{4,\epsilon}^k : d(z, \partial \tilde{\Omega}_{4,\epsilon}^k ) \geq \frac{\eta'}{2} \right\}
\end{equation} 
can be covered by finitely many such balls $B(z,\eta'/8)$.
Apply Harnack inequality again, there exists $c_{k,4}>0$ such that
\begin{equation}\label{Har int 1}
	\sup \tilde{p}_{\epsilon}^k(z) \leq c_{k,4} \inf  \tilde{p}_{\epsilon}^k(z),
\end{equation}
where the supremum and infimum is taken over the region in \eqref{region 1}.
Combining \eqref{Har int +}, \eqref{Har int -}, and \eqref{Har int 1}, we then have \eqref{reg har}.

For $d=3$, $\tilde{\Gamma}_{1,\epsilon}^k \cap \partial \tilde{\Omega}_{4,\epsilon}^k$ is a circle.
First, consider $A_\epsilon^k(z_k^+,2\eta')$ for some $z_k^+ \in \tilde{\Gamma}_{1,\epsilon}^k \cap \partial \tilde{\Omega}_{4,\epsilon}^k$.
The same flattening and reflection argument as in $d=2$ applies, with $z'$ now ranging over $\mathbb{R}^2$ instead of $\mathbb{R}$. This gives
\begin{equation*}
	\sup_{z \in A_\epsilon^k(z_k^+,2\eta')} \tilde{p}_\epsilon^k (z) \leq c_{k,5} \inf_{z \in A_\epsilon^k(z_k^+,2\eta')} \tilde{p}_\epsilon^k (z),
\end{equation*}
with $c_{k,5}$ depending on the Lipschitz constant of the domain $\tilde{\Gamma}_{1,\epsilon}^k \cap \partial \tilde{\Omega}_{4,\epsilon}^k$; by the construction in Section \ref{sec narrow tube} this Lipschitz constant is bounded uniformly in independently of $\epsilon$.
The region 
\begin{equation*}
	\left \{z \in \tilde{\Gamma}_{1,\epsilon}^k : d(z, \partial \tilde{\Omega}_{4,\epsilon}^k ) \leq \eta' \right\}
\end{equation*} 
can be covered by finitely many such balls $A_\epsilon^k(z_k^+,2\eta')$ for some $z_k^+ \in \tilde{\Gamma}_{1,\epsilon}^k \cap \partial \tilde{\Omega}_{4,\epsilon}^k$.
The remaining steps then follow from the previous arguments.
\end{proof}

Thanks to \eqref{reg har}, we conclude that there exists a constant $c_k>0$ such that
\begin{equation*}
	\inf_{z \in \Gamma_{1,\epsilon}^k \cup C_\epsilon^k(\rho r(\epsilon) + 2 \epsilon)} \mathbb{P}_z \left( Z^\epsilon(\sigma_0^{2,\epsilon}) \in \Gamma_{3,\epsilon} \right)
	= \inf_{z \in \tilde{\Gamma}_{0,\epsilon}^k \cup \tilde{\Gamma}_{1,\epsilon}^k } \tilde{p}_\epsilon^k (z) \geq c_k.
\end{equation*}
It then follows that
\begin{equation*}
	\sup_{z \in \Gamma_{1,\epsilon}^k \cup C_\epsilon^k(\rho r(\epsilon) + 2 \epsilon)} \mathbb{E}_z N^{1,\epsilon} \leq  \frac{1}{ \min_{k: I_k \sim O} c_k }.
\end{equation*}

Next, we prove estimate \eqref{reg per 2}. 
Let $\tilde{v}_\epsilon^k = \tilde{\mathbb{E}}_z^{k} \tilde{\sigma}^{2,\epsilon}_k$.
We know that $\tilde{v}_\epsilon^k$ solves the following PDE
	\begin{equation*}
\begin{aligned}
    \begin{dcases}
    \displaystyle
        \Delta \tilde{v}_\epsilon^k(z) = -1 \ \ \ \text{for} \ \ z \in \tilde{\Omega}_{4,\epsilon}^k, \\
        \tilde{v}_\epsilon^k(z)=0 \ \ \ \text{for} \ \ z \in \tilde{\Gamma}_{2,\epsilon}^k \cup \tilde{\Gamma}_{3,\epsilon}^k, \\
        \frac{\partial \tilde{v}_\epsilon^k}{\partial \nu_\epsilon}=0 \ \ \  \text{for} \ \ z \in \partial \tilde{\Omega}_{4,\epsilon}^k \setminus \left( \tilde{\Gamma}_{2,\epsilon}^k \cup \tilde{\Gamma}_{3,\epsilon}^k \right).
    \end{dcases}
\end{aligned}
\end{equation*}
We want to show that there exists $c>0$ independent of $\epsilon$ such that
\begin{equation}\label{har po est 0}
    \sup_{z \in \tilde{\Gamma}_{0,\epsilon}^k \cup \tilde{\Gamma}_{1,\epsilon}^k}\tilde{v}_\epsilon^k(z) \leq c.
\end{equation}
By a similar argument as above, with the Harnack inequality for Poisson equation with constant source term (see e.g., \cite{Imb26}, Theorem 2.5.1), we have
\begin{equation}\label{har po est 1}
     \sup_{z \in \tilde{\Gamma}_{0,\epsilon}^k \cup \tilde{\Gamma}_{1,\epsilon}^k}\tilde{v}_\epsilon^k(z) \leq c \left( \inf_{z \in B(z_0,\eta) }\tilde{v}_\epsilon^k(z)+1\right),
\end{equation}
where $c>0$ is independent of $\epsilon$, $z_0 \in \tilde{\Gamma}_{0,\epsilon}^k$ lies in the interior of $\tilde{\Omega}_{4,\epsilon}^k$ and $\eta > 0$ is some fixed constant such that $B(z_0,2\eta) \cap \partial \tilde{\Omega}_{4,\epsilon}^k = \emptyset$.
Furthermore, by \cite{Imb26}, Proposition 2.3.4 
\begin{equation}\label{har po est 2}
    \inf_{z \in B(z_0,\eta) }\tilde{v}_\epsilon^k(z)
    \leq \sup_{z \in B(z_0,\eta) }\tilde{v}_\epsilon^k(z)
    \leq c \left( \abs{\tilde{v}_\epsilon^k}_{L^2(B(z_0,4\eta/3))} +1 \right)
    \leq c \left( \abs{\tilde{v}_\epsilon^k}_{L^2(\tilde{\Omega}_{4,\epsilon}^k)} +1 \right),
\end{equation}
where $c > 0$ depends only on the dimension and the ellipticity and thus is independent of $\epsilon$ as demonstrated above.
We claim that
\begin{equation}\label{cl l2 est}
    \sup_{\epsilon \in(0,1)} \abs{\tilde{v}_\epsilon^k}_{L^2(\tilde{\Omega}_{4,\epsilon}^k)} \leq c.
\end{equation}

For every $\epsilon > 0$ and $z \in \tilde{\Omega}_{4,\epsilon}^k$ there exists $t_\epsilon(z) \geq 0$ such that $z-t_\epsilon(z)e_k \in \tilde{\Gamma}_{2,\epsilon}^k$, where $e_k$ is the unit vector point outward from the vertex $O$ along the edge $I_k$ (see Section \ref{sec graph}).
Moreover, by the monotonicity of the domain construction (see section \ref{sec narrow tube}), the straight line 
\begin{equation}\label{eq line}
    z-se_k \in \tilde{\Omega}_{4,\epsilon}^k, \quad s \in [0, t_\epsilon(z)].
\end{equation}
Note that, if we define $D_\epsilon$ be the diameter of $\tilde{\Omega}_{4,\epsilon}^k $, we have $t_\epsilon(z) \leq D_\epsilon \leq c$ for some $c>0$ independent of $\epsilon$.
By the fundamental theorem of calculus, for every $z \in \tilde{\Omega}_{4,\epsilon}^k$,
\begin{equation*}
\begin{aligned}
    \tilde{v}_\epsilon^k(z) = \tilde{v}_\epsilon^k(z) - \tilde{v}_\epsilon^k(z-t_\epsilon(z)e_k) 
    = \int_0^{t_\epsilon(z)}  e_k \cdot \nabla \tilde{v}_\epsilon^k(z-s  e_k) ds.
\end{aligned}
\end{equation*}
By Cauchy-Schwarz inequality,
\begin{equation*}
\begin{aligned}
    \abs{\tilde{v}_\epsilon^k(z)}^2 \leq \abs{t_\epsilon(z)} \int_0^{t_\epsilon(z)}   \abs{\nabla \tilde{v}_\epsilon^k(z-s  e_k)}^2 ds.
\end{aligned}
\end{equation*}
Consequently,
\begin{equation}\label{eq 1 poin}
\begin{aligned}
    &\abs{\tilde{v}_\epsilon^k(z)}_{L^2(\tilde{\Omega}_{4,\epsilon}^k)}^2 
    \leq  \int_{\tilde{\Omega}_{4,\epsilon}^k} \abs{t_\epsilon(z)} \int_0^{t_\epsilon(z)}   \abs{\nabla \tilde{v}_\epsilon^k(z-s  e_k)}^2 ds dz \\
    & \leq D_\epsilon \int_0^{D_\epsilon}  \int_{\tilde{\Omega}_{4,\epsilon}^k}    \abs{\nabla \tilde{v}_\epsilon^k(z-s  e_k)}^2 \mathbbm{1}_{ [0,t_\epsilon(z)]}(s) dz ds.
\end{aligned}
\end{equation}
Let 
\begin{equation*}
    \tilde{A}_\epsilon^k(s) \coloneqq \{ z \in \tilde{\Omega}_{4,\epsilon}^k: t_\epsilon(z) \geq s\}.
\end{equation*}
Thanks to \eqref{eq line}, we have
\begin{equation*}
    z \in \tilde{A}_\epsilon^k(s) \Rightarrow w \coloneqq z-se_k \in \tilde{\Omega}_{4,\epsilon}^k
\end{equation*}
and thus
\begin{equation}\label{eq 2 poin}
\begin{aligned}
     &\int_{\tilde{\Omega}_{4,\epsilon}^k}    \abs{\nabla \tilde{v}_\epsilon^k(z-s  e_k)}^2 \mathbbm{1}_{ [0,t_\epsilon(z)]}(s) dz
     = \int_{\tilde{A}_\epsilon^k(s)}    \abs{\nabla \tilde{v}_\epsilon^k(z-s  e_k)}^2 dz \\
    &= \int_{\tilde{A}_\epsilon^k(s)-se_k}    \abs{\nabla \tilde{v}_\epsilon^k(w)}^2 dw
    \leq \int_{\tilde{\Omega}_{4,\epsilon}^k}    \abs{\nabla \tilde{v}_\epsilon^k(w)}^2 dw.
\end{aligned}
\end{equation}
Combining \eqref{eq 1 poin} and \eqref{eq 2 poin}, we then conclude that
\begin{equation}\label{poin}
	\abs{\tilde{v}_\epsilon^k}_{L^2(\tilde{\Omega}_{4,\epsilon}^k)} 
	\leq D_\epsilon \abs{\nabla \tilde{v}_\epsilon^k}_{L^2(\tilde{\Omega}_{4,\epsilon}^k)}
    \leq c \abs{\nabla \tilde{v}_\epsilon^k}_{L^2(\tilde{\Omega}_{4,\epsilon}^k)}.
\end{equation}

Next, using integration by parts and Cauchy-Schwarz inequality
\begin{equation*}
\begin{aligned}
	\abs{\nabla \tilde{v}_\epsilon^k}_{L^2(\tilde{\Omega}_{4,\epsilon}^k)}^2
	= \int_{\tilde{\Omega}_{4,\epsilon}^k}  \nabla \tilde{v}_\epsilon^k(z) \nabla \tilde{v}_\epsilon^k(z) dz
	\leq \int_{\tilde{\Omega}_{4,\epsilon}^k} \abs{ \Delta \tilde{v}_\epsilon^k(z)} \abs{  \tilde{v}_\epsilon^k(z)} dz
    \leq \abs{\tilde{\Omega}_{4,\epsilon}^k}^{1/2}  \abs{\tilde{v}_\epsilon^k}_{L^2(\tilde{\Omega}_{4,\epsilon}^k)}
\end{aligned}
\end{equation*}
Since there exists a constant $c>0$ such that $\sup_{\epsilon \in(0,1)} \abs{\tilde{\Omega}_{4,\epsilon}^k} \leq c$, by \eqref{poin} , we have
\begin{equation*}
\begin{aligned}
	\abs{\nabla \tilde{v}_\epsilon^k}_{L^2(\tilde{\Omega}_{4,\epsilon}^k)}^2 \leq c\abs{\tilde{v}_\epsilon^k}_{L^2(\tilde{\Omega}_{4,\epsilon}^k)}
    \leq c \abs{\nabla \tilde{v}_\epsilon^k}_{L^2(\tilde{\Omega}_{4,\epsilon}^k)},
\end{aligned}
\end{equation*}
which implies
\begin{equation*}
\abs{\nabla \tilde{v}_\epsilon^k}_{L^2(\tilde{\Omega}_{4,\epsilon}^k)}
    \leq c.
\end{equation*}
Applying \eqref{poin} again, we obtain
\begin{equation*}
	\sup_{\epsilon \in (0,1)} \abs{\tilde{v}_\epsilon^k}_{L^2(\tilde{\Omega}_{4,\epsilon}^k)} 
	\leq c \sup_{\epsilon \in (0,1)} \abs{\nabla \tilde{v}_\epsilon^k}_{L^2(\tilde{\Omega}_{4,\epsilon}^k)}
    \leq c,
\end{equation*}
This proves \eqref{cl l2 est}. 
Combining this estimate with \eqref{har po est 1} and \eqref{har po est 2}, we obtain \eqref{har po est 0}.
Equation \eqref{reg per 2} then follows by Brownian scaling.

\begin{remark}\label{rem poin}
    Here, we prove \eqref{poin} by connecting each point $z$ to the Dirichlet boundary $\tilde{\Gamma}_{2,\epsilon}^k$ along a straight line segment, which is ensured by the monotonicity condition on the domain. 
    In principle, this condition could be weakened by allowing more general curves to connect $z$ to $\tilde{\Gamma}_{2,\epsilon}^k$. 
    Such an approach would require a more careful construction of these curves, which is beyond the scope of the present work. 
    We therefore impose the monotonicity condition for simplicity.
\end{remark}

\section{Proof of Theorem \ref{conv qsd}}\label{app conv qsd}
In this section, we show that Claim \ref{itm: c1} and \ref{itm: c2} together imply Theorem \ref{conv qsd}.
Our argument largely follows the approach in \cite{QSD}, Section 5.
In fact, we will prove that the following weaker versions of Claim \ref{itm: c1} and \ref{itm: c2} are sufficient to establish Theorem \ref{conv qsd}.

	There exists a family of probability measure $\left\{ \nu_{\epsilon,\delta',z_1,z_2} \right\}_{z_1,z_2 \in C_{\epsilon}(\delta')}$ on $C_{\epsilon}(\delta')$ such that
	\begin{claim2} \label{itm: c1'}
	For all $\eta'>0$, there exist $c_1>0$, $n_0 \geq 0$, and $\delta_{\eta'}'>0$ such that for all $\delta' \in (0,\delta_{\eta'}')$, there exists $\epsilon_{\eta',\delta'}>0$ such that for all $\epsilon \in (0,\epsilon_{\eta',\delta'})$ and $z_1,z_2 \in C_{\epsilon}(\delta')$, we have
	\begin{equation*}
		\mathbb{P}_{z_i} \left( X_{n_0}^{\epsilon,\delta',\delta} \in \cdot \middle|  \tilde{\tau}_{n_0}^{\epsilon,\delta'} < \sigma^{\epsilon,\delta} \right) \geq c_1 \nu_{\epsilon,\delta',z_1,z_2}(\cdot) - \eta', \text{ for } i=1,2.
	\end{equation*}
	\end{claim2}
	\begin{claim2} \label{itm: c2'}
	There exists $c_2>0$ such that for all $\epsilon,\delta'$ small enough, $z_1,z_2 \in C_{\epsilon}(\delta')$ and $n \geq 0$, we have
	\begin{equation*}
		\mathbb{P}_{\nu_{\epsilon,\delta',z_1,z_2}} \left(   \tilde{\tau}_n^{\epsilon,\delta'} < \sigma^{\epsilon,\delta} \right) \geq c_2 \sup_{z \in C_\epsilon(\delta')} \mathbb{P}_{z} \left(   \tilde{\tau}_n^{\epsilon,\delta'} < \sigma^{\epsilon,\delta} \right).
	\end{equation*}
	\end{claim2}
It is clear that Claim \ref{itm: c1} and \ref{itm: c2} imply the weaker versions, Claim \ref{itm: c1'} and \ref{itm: c2'}.
We now show that Claim \ref{itm: c1'} and \ref{itm: c2'} together imply Theorem \ref{conv qsd} for the case $n_0=2$.

First note that, for any fixed $\epsilon,\delta'>0$, the existence and uniqueness of a quasi-stationary distribution follow directly from \cite{QSD}.
The central question is whether the convergence to this distribution is uniform with respect to $\epsilon$ and $\delta'$.

By proceeding as \cite{QSD}, we first show that, 
for all $N \geq 2$ and $z_1,z_2 \in C_\epsilon(\delta')$, there exists a probability measure $\nu_{\epsilon,\delta',z_1,z_2}^N$ on $C_\epsilon(\delta')$ such that for all $\eta' >0$, there exists $\delta_{\eta'}'>0$ such that for all $\delta' \in (0,\delta_{\eta'}')$ there exists $\epsilon_{\eta',\delta'} > 0$ such that for all $\epsilon \in (0,\epsilon_{\eta',\delta'})$ and all measurable set $A_{\epsilon,\delta'} \subset C_\epsilon(\delta')$,
	\begin{equation}\label{app qsd eq 0}
		\mathbb{P}_{z_i} \left( X_{2}^{\epsilon,\delta',\delta} \in A_{\epsilon,\delta'} \middle|  \tilde{\tau}_N^{\epsilon,\delta'} < \sigma^{\epsilon,\delta} \right) 
		\geq c_1 c_2  \nu_{\epsilon,\delta',z_1,z_2}^N(A_{\epsilon,\delta'}) 
		- \frac{  \eta' }{\sup_{z \in C_\epsilon(\delta')} \mathbb{P}_{z} \left(  \tilde{\tau}_{N-2}^{\epsilon,\delta'} < \sigma^{\epsilon,\delta} \right)},
	\end{equation}
	for $i=1,2$.
	For $i=1,2$ and any $N \geq 2$, by the strong Markov property and Claim \ref{itm: c1'}, for all $\eta' >0$, there exists $\delta_{\eta'}'>0$ such that for all $\delta' \in (0,\delta_{\eta'}')$ there exists $\epsilon_{\eta',\delta'}>0$ such that for all $\epsilon \in (0,\epsilon_{\eta',\delta'})$ and all measurable set $A_{\epsilon,\delta'} \subset C_\epsilon(\delta')$, we have
	\begin{equation*}
	\begin{aligned}
		&\mathbb{P}_{z_i} \left(X_2^{\epsilon,\delta',\delta} \in A_{\epsilon,\delta'}, \tilde{\tau}_{N}^{\epsilon,\delta'} < \sigma^{\epsilon,\delta} \right) \\
		&= \mathbb{E}_{z_i} \left[ \mathbbm{1}_{ \{X_2^{\epsilon,\delta',\delta} \in A_{\epsilon,\delta'}\} } \mathbb{P}_{X_2^{\epsilon,\delta',\delta}} \left( \tilde{\tau}_{N-2}^{\epsilon,\delta'} < \sigma^{\epsilon,\delta} \right)
		\middle| \tilde{\tau}_{2}^{\epsilon,\delta'} < \sigma^{\epsilon,\delta} \right]
		\mathbb{P}_{z_i} \left(  \tilde{\tau}_{2}^{\epsilon,\delta'} < \sigma^{\epsilon,\delta} \right) \\
		& \geq \left( c_1 \nu_{\epsilon,\delta',z_1,z_2}\left( \mathbbm{1}_{A_{\epsilon,\delta'}}(\cdot) \mathbb{P}_{\cdot} \left( \tilde{\tau}_{N-2}^{\epsilon,\delta'} < \sigma^{\epsilon,\delta} \right) \right) - \eta' \right)
		 \mathbb{P}_{z_i} \left(  \tilde{\tau}_{2}^{\epsilon,\delta'} < \sigma^{\epsilon,\delta} \right).
	\end{aligned}
	\end{equation*}
	By the strong Markov property,
	\begin{equation*}
		\mathbb{P}_{z_i} \left(  \tilde{\tau}_{N}^{\epsilon,\delta'} < \sigma^{\epsilon,\delta} \right)
		\leq \mathbb{P}_{z_i} \left(  \tilde{\tau}_{2}^{\epsilon,\delta'} < \sigma^{\epsilon,\delta} \right) \sup_{z \in C_\epsilon(\delta')} \mathbb{P}_{z} \left(  \tilde{\tau}_{N-2}^{\epsilon,\delta'} < \sigma^{\epsilon,\delta} \right).
	\end{equation*}
	As a result,
	\begin{equation*}
		\mathbb{P}_{z_i} \left(X_2^{\epsilon,\delta',\delta} \in A_{\epsilon,\delta'} \middle| \tilde{\tau}_{N}^{\epsilon,\delta'} < \sigma^{\epsilon,\delta} \right)
		\geq \frac{ c_1 \nu_{\epsilon,\delta',z_1,z_2}\left( \mathbbm{1}_{A_{\epsilon,\delta'}}(\cdot) \mathbb{P}_{\cdot} \left( \tilde{\tau}_{N-2}^{\epsilon,\delta'} < \sigma^{\epsilon,\delta} \right) \right) - \eta' }{\sup_{z \in C_\epsilon(\delta')} \mathbb{P}_{z} \left(  \tilde{\tau}_{N-2}^{\epsilon,\delta'} < \sigma^{\epsilon,\delta} \right)}.
	\end{equation*}
	Let 
	\begin{equation*}
		\nu_{\epsilon,\delta',z_1,z_2}^N(A_{\epsilon,\delta'})
		\coloneqq \frac{ \nu_{\epsilon,\delta',z_1,z_2}\left( \mathbbm{1}_{A_{\epsilon,\delta'}}(\cdot) \mathbb{P}_{\cdot} \left( \tilde{\tau}_{N-2}^{\epsilon,\delta'} < \sigma^{\epsilon,\delta} \right) \right) }{ \mathbb{P}_{\nu_{\epsilon,\delta',z_1,z_2}} \left(  \tilde{\tau}_{N-2}^{\epsilon,\delta'} < \sigma^{\epsilon,\delta} \right)},
	\end{equation*}
	which is clearly a probability measure on $C_\epsilon(\delta')$.
	By Claim \ref{itm: c2'}, we have
	\begin{equation*}
	\begin{aligned}
		&\mathbb{P}_{z_i} \left(X_2^{\epsilon,\delta',\delta} \in A_{\epsilon,\delta'} \middle| \tilde{\tau}_{N}^{\epsilon,\delta'} < \sigma^{\epsilon,\delta} \right) \\
		&\geq c_1 \frac{  \nu_{\epsilon,\delta',z_1,z_2}\left( \mathbbm{1}_{A_{\epsilon,\delta'}}(\cdot) \mathbb{P}_{\cdot} \left( \tilde{\tau}_{N-2}^{\epsilon,\delta'} < \sigma^{\epsilon,\delta} \right) \right) }{\sup_{z \in C_\epsilon(\delta')} \mathbb{P}_{z} \left(  \tilde{\tau}_{N-2}^{\epsilon,\delta'} < \sigma^{\epsilon,\delta} \right)}
		- \frac{  \eta' }{\sup_{z \in C_\epsilon(\delta')} \mathbb{P}_{z} \left(  \tilde{\tau}_{N-2}^{\epsilon,\delta'} < \sigma^{\epsilon,\delta} \right)} \\
		& = c_1 \nu_{\epsilon,\delta',z_1,z_2}^N(A_{\epsilon,\delta'}) 
		\frac{  \mathbb{P}_{\nu_{\epsilon,\delta',z_1,z_2}} \left(  \tilde{\tau}_{N-2}^{\epsilon,\delta'} < \sigma^{\epsilon,\delta} \right) }{\sup_{z \in C_\epsilon(\delta')} \mathbb{P}_{z} \left(  \tilde{\tau}_{N-2}^{\epsilon,\delta'} < \sigma^{\epsilon,\delta} \right)}
		- \frac{  \eta' }{\sup_{z \in C_\epsilon(\delta')} \mathbb{P}_{z} \left(  \tilde{\tau}_{N-2}^{\epsilon,\delta'} < \sigma^{\epsilon,\delta} \right)} \\
		& \geq c_1 c_2 \nu_{\epsilon,\delta',z_1,z_2}^N(A_{\epsilon,\delta'}) 
		- \frac{  \eta' }{\sup_{z \in C_\epsilon(\delta')} \mathbb{P}_{z} \left(  \tilde{\tau}_{N-2}^{\epsilon,\delta'} < \sigma^{\epsilon,\delta} \right)}.
	\end{aligned}
	\end{equation*}

	Next, we show that for every $\eta>0$, there exists $N_\eta \in \mathbb{N}$ and, $\delta_{\eta}'>0$ such that for all $\delta' \in (0,\delta_{\eta}')$ there exists $\epsilon_{\eta,\delta'} > 0$ such that for all $\epsilon \in (0,\epsilon_{\eta,\delta'})$ and $z,z' \in C_\epsilon(\delta')$, we have
	\begin{equation}\label{app qsd eq 1}
		\left \lVert \mathbb{P}_z \left(X_{N_\eta}^{\epsilon,\delta',\delta} \in \cdot \middle| \tilde{\tau}_{N_\eta}^{\epsilon,\delta'} < \sigma^{\epsilon,\delta} \right) - \mathbb{P}_{z'} \left(X_{N_\eta}^{\epsilon,\delta',\delta} \in \cdot \middle| \tilde{\tau}_{N_\eta}^{\epsilon,\delta'} < \sigma^{\epsilon,\delta} \right) \right \rVert_{TV} < \eta.
	\end{equation}
	The main difference arises in the first step, where we must carry the $\eta'$ term throughout the iteration.
	To be more precise, for $0 \leq n \leq m \leq N$, we define the linear operator $S_{\epsilon,\delta',n,m}^N$ by
	\begin{equation*}
		S_{\epsilon,\delta',n,m}^N f(z) 
		\coloneqq \mathbb{E}_z \left( f \left( X_{m-n}^{\epsilon,\delta',\delta} \right) \middle| \tilde{\tau}_{N-n}^{\epsilon,\delta'} < \sigma^{\epsilon,\delta} \right)
		= \mathbb{E} \left( f \left( X_{m}^{\epsilon,\delta',\delta} \right) \middle| X_{n}^{\epsilon,\delta',\delta} = z, \tilde{\tau}_{N}^{\epsilon,\delta'} < \sigma^{\epsilon,\delta} \right).
	\end{equation*}
	One can show that the family $\{S_{\epsilon,\delta',n,m}^N\}_{0 \leq n \leq m \leq N}$ is a Markov semigroup.
	By \eqref{app qsd eq 0}, for $i=1,2$, $
		\delta_{z_i} S_{\epsilon,\delta',n,n+2}^N - c_1 c_2 \nu_{\epsilon,\delta',z_1, z_2}^{N-n} $is a signed measure whose mass is $1-c_1 c_2 $. 
		Moreover, the negative part $\left( \delta_{z_i} S_{\epsilon,\delta',n,n+2}^N - c_1 c_2 \nu_{\epsilon,\delta',z_1, z_2}^{N-n} \right)^-$ has mass less than 
		\begin{equation*}
			\frac{  \eta' }{\sup_{z \in C_\epsilon(\delta')} \mathbb{P}_{z} \left(  \tilde{\tau}_{N-n-2}^{\epsilon,\delta'} < \sigma^{\epsilon,\delta} \right)}
		\end{equation*}
		and thus the positive part $\left( \delta_{z_i} S_{\epsilon,\delta',n,n+2}^N - c_1 c_2 \nu_{\epsilon,\delta',z_1, z_2}^{N-n} \right)^+$ has mass less than 
		\begin{equation*}
			1-c_1c_2+\frac{  \eta' }{\sup_{z \in C_\epsilon(\delta')} \mathbb{P}_{z} \left(  \tilde{\tau}_{N-n-2}^{\epsilon,\delta'} < \sigma^{\epsilon,\delta} \right)}.
		\end{equation*}
		We then conclude that
		\begin{equation*}
		\begin{aligned}
			&\left \lVert \delta_{z_1} S_{\epsilon,\delta',n,n+2}^N - \delta_{z_2} S_{\epsilon,\delta',n,n+2}^N \right \rVert_{TV} \\
			 & \leq \left \lVert \delta_{z_1} S_{\epsilon,\delta',n,n+2}^N - c_1 c_2 \nu_{\epsilon,\delta',z_1, z_2}^{N-n} \right \rVert_{TV}
			+\left \lVert \delta_{z_2} S_{\epsilon,\delta',n,n+2}^N - c_1 c_2 \nu_{\epsilon,\delta',z_1, z_2}^{N-n} \right \rVert_{TV} \\
			& \leq \sum_{i=1}^2  \left \lVert \left( \delta_{z_i} S_{\epsilon,\delta',n,n+2}^N - c_1 c_2 \nu_{\epsilon,\delta',z_1, z_2}^{N-n} \right)^+ \right \rVert_{TV}
			+\left \lVert \left( \delta_{z_i} S_{\epsilon,\delta',n,n+2}^N - c_1 c_2 \nu_{\epsilon,\delta',z_1, z_2}^{N-n} \right)^- \right \rVert_{TV}  \\
			& \leq 2 \left( 1-c_1 c_2 + \frac{  2\eta' }{\sup_{z \in C_\epsilon(\delta')} \mathbb{P}_{z} \left(  \tilde{\tau}_{N-n-2}^{\epsilon,\delta'} < \sigma^{\epsilon,\delta} \right)} \right).
		\end{aligned}
		\end{equation*}
		Arguing the same as \cite{QSD}, we have, for any two probability measures $\mu_{1}, \mu_{2}$ on $C_\epsilon(\delta')$,
		\begin{equation*}
		\begin{aligned}
			&\left \lVert \mu_{1} S_{\epsilon,\delta',n,n+2}^N - \mu_{2} S_{\epsilon,\delta',n,n+2}^N \right \rVert_{TV} \\
			 &\leq \left( 1-c_1 c_2 + \frac{  2\eta' }{\sup_{z \in C_\epsilon(\delta')} \mathbb{P}_{z} \left(  \tilde{\tau}_{N-n-2}^{\epsilon,\delta'} < \sigma^{\epsilon,\delta} \right)} \right) \left \lVert \mu_{1} - \mu_{2}  \right \rVert_{TV}.
		\end{aligned}
		\end{equation*}
	
	By the semigroup property, we have
	\begin{equation}\label{app qsd eq 2}
	\begin{aligned}
		&\left \lVert \mathbb{P}_z \left(X_N^{\epsilon,\delta',\delta} \in \cdot \middle| \tilde{\tau}_{N}^{\epsilon,\delta'} < \sigma^{\epsilon,\delta} \right) - \mathbb{P}_{z'} \left(X_{N}^{\epsilon,\delta',\delta} \in \cdot \middle| \tilde{\tau}_{N}^{\epsilon,\delta'} < \sigma^{\epsilon,\delta} \right) \right \rVert_{TV} \\
		&= \left \lVert \delta_{z} S_{\epsilon,\delta',0,N}^N - \delta_{z'} S_{\epsilon,\delta',0,N}^N \right \rVert_{TV} \\
		&= \left \lVert \delta_{z} S_{\epsilon,\delta',0,N-2}^N S_{\epsilon,\delta',N-2,N}^N 
		- \delta_{z'} S_{\epsilon,\delta',0,N-2}^N S_{\epsilon,\delta',N-2,N}^N \right \rVert_{TV} \\
		& \leq \left( 1-c_1 c_2 + \frac{  2\eta' }{\sup_{z \in C_\epsilon(\delta')} \mathbb{P}_{z} \left(  \tilde{\tau}_{0}^{\epsilon,\delta'} < \sigma^{\epsilon,\delta} \right)} \right) 
		\left \lVert \delta_{z} S_{\epsilon,\delta',0,N-2}^N 
		- \delta_{z'} S_{\epsilon,\delta',0,N-2}^N   \right \rVert_{TV} \\
		&\leq 2 \prod_{m=0}^{\left \lfloor{\frac{N}{2}}\right \rfloor -1 } \left( 1-c_1c_2+ \frac{ 2\eta'}{\sup_{z \in C_\epsilon(\delta')} \mathbb{P}_{z} (\tilde{\tau}_{2m}^{\epsilon,\delta'} < \sigma^{\epsilon,\delta}) }\right).
	\end{aligned}
	\end{equation}
	Now pick $N_\eta > 0$ such that 
	\begin{equation}\label{app qsd eq 3}
		2\left(1- \frac{1}{2}c_1c_2 \right)^{\left \lfloor{\frac{N_\eta}{2}}\right \rfloor} < \eta
	\end{equation}
	and let $\delta_\eta' > 0$ small enough such that for all $\delta' \in (0,\delta_\eta')$, we can pick $\epsilon_{\eta,\delta'}> 0$ small enough such that for all $\epsilon \in (0,\epsilon_{\eta,\delta'})$, we have
	\begin{equation}\label{app qsd eq 4}
		\max_{m=0}^{\left \lfloor{\frac{N_\eta}{2}}\right \rfloor -1} \frac{ 2\eta'}{\sup_{z \in C_\epsilon(\delta')} \mathbb{P}_{z} (\tilde{\tau}_{2m}^{\epsilon,\delta'} < \sigma^{\epsilon,\delta}) } < \frac{c_1 c_2}{2}.
	\end{equation}
	This is possible because, by \eqref{est p}, for all $z \in C_\epsilon(\delta')$,
	\begin{equation*}
		\mathbb{P}_{z} (\tilde{\tau}_{m}^{\epsilon,\delta'} < \sigma^{\epsilon,\delta})
		= \left( p_{\epsilon,\delta',\delta}^1 \right)^m,
	\end{equation*}
	where $\lim_{\delta' \to 0} \lim_{\epsilon \to 0} p_{\epsilon,\delta',\delta}^1 = 1$.	
	Combining this with Claim \ref{itm: c1'} yields \eqref{app qsd eq 4}.
	Then, \eqref{app qsd eq 1} follows from \eqref{app qsd eq 2}--\eqref{app qsd eq 4}.
	
	For the final step, to extend \eqref{app qsd eq 1} to general initial distributions, we follow the same argument as in \cite{QSD}, as the bound established above holds uniformly with respect to the initial condition.

\section{Proof of Theorem \ref{thm abs}}\label{sec pf exp}

	Let 
	\begin{equation*}
		p_+(\delta',\delta) \coloneqq \sup_{z \in \mathcal{S}_{\epsilon}, n \geq 0} \mathbb{P}_z \left( \zeta_{n+1}^{\epsilon,\delta',\delta}=0 \middle| \zeta_n^{\epsilon,\delta',\delta} = 1 \right)
	\end{equation*}		
	and
	\begin{equation*}
		p_-(\delta',\delta) \coloneqq \inf_{z \in \mathcal{S}_{\epsilon}, n \geq 0} \mathbb{P}_z \left( \zeta_{n+1}^{\epsilon,\delta',\delta}=0 \middle| \zeta_n^{\epsilon,\delta',\delta} = 1 \right).
	\end{equation*}	
    Recall that $K^{\epsilon,\delta',\delta} = \min\{ n: \zeta_n^{\epsilon,\delta',\delta} =0 \}$.
	For each $n \geq 0,$ $ z \in \mathcal{S}_{\epsilon}$, $\epsilon >0$ and $\delta' > 0$
	\begin{equation*}
		\left( 1-p_+(\delta',\delta) \right)^n
		\leq \mathbb{P}_z \left( K^{\epsilon,\delta',\delta} > n \right) 
		\leq \left( 1-p_-(\delta',\delta) \right)^n
	\end{equation*}
	By equation \eqref{trans prob}, 
	\begin{equation*}
	\begin{aligned}
		&\limsup_{\delta' \to 0} \limsup_{\epsilon \to 0} \sup_{z \in \mathcal{S}_{\epsilon}} \mathbb{P}_z \left( p(\delta',\delta) K^{\epsilon,\delta',\delta} > t \right) 
		=\limsup_{\delta' \to 0}\limsup_{\epsilon \to 0} \sup_{z \in \mathcal{S}_{\epsilon}} \mathbb{P}_z \left(  K^{\epsilon,\delta',\delta} > \left \lfloor \frac{t}{p(\delta',\delta)} \right \rfloor \right) \\
		&\leq \limsup_{\delta' \to 0} \limsup_{\epsilon \to 0} \left( 1-p_-(\delta',\delta) \right)^{\left \lfloor \frac{t}{p(\delta',\delta)} \right \rfloor}
		=e^{-t}.
	\end{aligned}
	\end{equation*}
	Using $p_+(\delta',\delta)$ one can derive the case for limit infimum. We thus obtain the result for the first part
    \begin{equation}\label{pf abs 1}
		\lim_{\delta' \to 0} \lim_{\epsilon \to 0} \sup_{z \in \mathcal{S}_{\epsilon}} \abs{ \mathbb{P}_z (p(\delta',\delta)K^{\epsilon,\delta',\delta} \geq t) - e^{-t} } =0.
	\end{equation}
	
	Let $t>0$ and $\eta >0$. We have
	\begin{equation}\label{exp est eq 1}
	\begin{aligned}
		&\mathbb{P}_z \left( \left( e(\epsilon,\delta') \right)^{-1} p(\delta',\delta) \sum_{n=0}^\infty \xi_{n+1}^{\epsilon,\delta',\delta}\zeta_n^{\epsilon,\delta',\delta}   > t+\eta, \ \ p(\delta',\delta)K^{\epsilon,\delta',\delta} < t \right) \\
		&\leq \mathbb{P}_z \left( \left( e(\epsilon,\delta') \right)^{-1} p(\delta',\delta) \sum_{n=0}^{K^{\epsilon,\delta',\delta}-1}  \xi_{n+1}^{\epsilon,\delta',\delta} > t+\eta, \ \ K^{\epsilon,\delta',\delta} \leq \left \lfloor \frac{t}{p(\delta',\delta)} \right \rfloor  \right) \\
		& \leq \mathbb{P}_z \left( \left( e(\epsilon,\delta') \right)^{-1} p(\delta',\delta) \sum_{n=0}^{\left \lfloor \frac{t}{p(\delta',\delta)} \right \rfloor -1}  \xi_{n+1}^{\epsilon,\delta',\delta}  > t+ \eta \right) \\
        & \leq \mathbb{P}_z \left(  \sum_{n=0}^{\left \lfloor \frac{t}{p(\delta',\delta)} \right \rfloor -1} \left( \frac{\xi_{n+1}^{\epsilon,\delta',\delta}}{e(\epsilon,\delta')} - 1 \right)> \frac{\eta}{p(\delta',\delta)} \right) \\
		&\leq \frac{p(\delta',\delta)^2}{\eta^2} \mathbb{E}_z \left( \left( \sum_{n=0}^{\left \lfloor \frac{t}{p(\delta',\delta)} \right \rfloor -1} \left( \frac{\xi_{n+1}^{\epsilon,\delta',\delta}}{e(\epsilon,\delta')} - 1 \right) \right)^2 \right) \\
		& \leq \frac{p(\delta',\delta)^2}{\eta^2} \sum_{n=0}^{\left \lfloor \frac{t}{p(\delta',\delta)} \right \rfloor -1} \mathbb{E}_z \left(   \left( \frac{\xi_{n+1}^{\epsilon,\delta',\delta}}{e(\epsilon,\delta')} - 1 \right)^2  \right) \\ 
		& \ \ \ \ \ \ \  + \frac{2p(\delta',\delta)^2}{\eta^2} \sum_{n=0}^{\left \lfloor \frac{t}{p(\delta',\delta)} \right \rfloor -1} \sum_{m=n+1}^{\left \lfloor \frac{t}{p(\delta',\delta)} \right \rfloor -1}
		\mathbb{E}_z \left(    \frac{\xi_{n+1}^{\epsilon,\delta',\delta}}{e(\epsilon,\delta')} - 1 \right)     \left( \frac{\xi_{m+1}^{\epsilon,\delta',\delta}}{e(\epsilon,\delta')} - 1 \right)  .
	\end{aligned}
	\end{equation}
	For the first term, note that, by equation \eqref{sec moment},
	\begin{equation*}
		\mathbb{E}_z \left(\left( \xi_{n+1}^{\epsilon,\delta',\delta} \right)^2  \right)
		=\mathbb{E}_z \left(  \mathbb{E} \left(  \left( \xi_{n+1}^{\epsilon,\delta',\delta} \right)^2 \middle | \mathcal{F}^{\epsilon,\delta',\delta}_{n} \right) \right)  \leq c(e(\epsilon,\delta'))^2
\end{equation*}	 
for $\epsilon,\delta'$ small enough.
As a result, we have
\begin{equation}\label{exp est eq 2}
	\frac{p(\delta',\delta)^2}{\eta^2} \sum_{n=0}^{\left \lfloor \frac{t}{p(\delta',\delta)} \right \rfloor -1} \mathbb{E}_z \left(   \left( \frac{\xi_{n+1}^{\epsilon,\delta',\delta}}{e(\epsilon,\delta')} - 1 \right)^2  \right)
	\leq c t \frac{p(\delta',\delta)}{\eta^2}.
\end{equation}
For the second term, by equation \eqref{fir moment}, we have, for each $m \geq 0$
\begin{equation*}
	  \lim_{\delta' \to 0} \lim_{\epsilon \to 0} \mathbb{E}_z \left(   \left( \frac{\xi_{m+1}^{\epsilon,\delta',\delta}}{e(\epsilon,\delta')} - 1 \right)   \middle | \mathcal{F}^{\epsilon,\delta',\delta}_{m}  \right)  =0.
\end{equation*}
Therefore, there exists a positive function $\alpha(\epsilon,\delta')$ such that $\lim_{\delta' \to 0} \lim_{\epsilon \to 0} \alpha(\epsilon,\delta') =0$ and
\begin{equation*}
	-\alpha(\epsilon,\delta') \leq \mathbb{E}_z \left(   \left( \frac{\xi_{m+1}^{\epsilon,\delta',\delta}}{e(\epsilon,\delta')} - 1 \right)   \middle | \mathcal{F}^{\epsilon,\delta',\delta}_{m}  \right) \leq \alpha(\epsilon,\delta').
\end{equation*}
Moreover, we have
\begin{equation}\label{exp est eq 3}
\begin{aligned}
&\frac{2p(\delta',\delta)^2}{\eta^2} \sum_{n=0}^{\left \lfloor \frac{t}{p(\delta',\delta)} \right \rfloor -1} \sum_{m=n+1}^{\left \lfloor \frac{t}{p(\delta',\delta)} \right \rfloor -1}
		\mathbb{E}_z \left(    \frac{\xi_{n+1}^{\epsilon,\delta',\delta}}{e(\epsilon,\delta')} - 1 \right)     \left( \frac{\xi_{m+1}^{\epsilon,\delta',\delta}}{e(\epsilon,\delta')} - 1 \right) \\
		&\leq \alpha(\epsilon,\delta') \frac{2p(\delta',\delta)^2}{\eta^2} \sum_{n=0}^{\left \lfloor \frac{t}{p(\delta',\delta)} \right \rfloor -1} \sum_{m=n+1}^{\left \lfloor \frac{t}{p(\delta',\delta)} \right \rfloor -1}
		\mathbb{E}_z \abs{    \frac{\xi_{n+1}^{\epsilon,\delta',\delta}}{e(\epsilon,\delta')} - 1 } \\
		&\leq \alpha(\epsilon,\delta') \frac{2p(\delta',\delta)^2}{\eta^2} \sum_{n=0}^{\left \lfloor \frac{t}{p(\delta',\delta)} \right \rfloor -1} \sum_{m=n+1}^{\left \lfloor \frac{t}{p(\delta',\delta)} \right \rfloor -1}
		\left( \mathbb{E}_z \left( \mathbb{E} \left( \left(    \frac{\xi_{n+1}^{\epsilon,\delta',\delta}}{e(\epsilon,\delta')} - 1 \right)^2 \middle| \mathcal{F}_n^{\epsilon,\delta',\delta} \right) \right) \right)^{1/2} \\
        & \leq c t^2 \alpha(\epsilon,\delta') \frac{1}{\eta^2}
\end{aligned}
\end{equation}
	Combining \eqref{exp est eq 1}, \eqref{exp est eq 2} and \eqref{exp est eq 3}, for every $\nu > 0$, $\eta >0$, $t >0$, there exists $\delta'_{1,\nu,\eta,t} >0$ such that for every $\delta' \in (0,\delta'_{1,\nu,\eta,t})$ there exists $\epsilon_{1,\nu,\eta,t,\delta'}>0$ such that for every $\epsilon \in (0,\epsilon_{1,\nu,\eta,t,\delta'})$, we have
	\begin{equation}\label{pf abs 2}
		\sup_{z \in \mathcal{S}_{\epsilon}} \mathbb{P}_z \left( \left( e(\epsilon,\delta') \right)^{-1} p(\delta',\delta) \sum_{n=0}^\infty \xi_{n+1}^{\epsilon,\delta',\delta}\zeta_n^{\epsilon,\delta',\delta}   > t+\eta, \ \ p(\epsilon,\delta')K^{\epsilon,\delta',\delta} < t \right) < \nu.
	\end{equation}
	
	Similarly,
	\begin{equation*}
	\begin{aligned}
	&\mathbb{P}_z \left( \left( e(\epsilon,\delta') \right)^{-1} p(\delta',\delta) \sum_{n=0}^\infty \xi_{n+1}^{\epsilon,\delta',\delta}\zeta_n^{\epsilon,\delta',\delta}   <t, \ \ p(\delta',\delta)K^{\epsilon,\delta',\delta} > t+\eta \right) \\ 
	&\leq \mathbb{P}_z \left(   \sum_{n=0}^{\left \lfloor \frac{t+\eta}{p(\delta',\delta)} \right \rfloor  } \frac{\xi_{n+1}^{\epsilon,\delta',\delta}}{e(\epsilon,\delta')}  <  \frac{t}{p(\delta',\delta)} \right) 
    \leq \mathbb{P}_z \left(  \sum_{n=0}^{\left \lfloor \frac{t+\eta}{p(\delta',\delta)} \right \rfloor } \left( 1- \frac{\xi_{n+1}^{\epsilon,\delta',\delta}}{e(\epsilon,\delta')} \right) > \frac{\eta}{p(\delta',\delta)} \right).
	\end{aligned}
	\end{equation*}
	Arguing as above, we also have
	for every $\nu > 0$, $\eta >0$, $t >0$, there exists $\delta'_{2,\nu,\eta,t} >0$ such that for every $\delta' \in (0,\delta'_{2,\nu,\eta,t})$ there exists $\epsilon_{2,\nu,\eta,t,\delta'}>0$ such that for every $\epsilon \in (0,\epsilon_{2,\nu,\eta,t,\delta'})$, we have
	\begin{equation}\label{pf abs 3}
		 \sup_{z \in \mathcal{S}_{\epsilon}} \mathbb{P}_z \left( \left( e(\epsilon,\delta') \right)^{-1} p(\delta',\delta) \sum_{n=0}^\infty \xi_{n+1}^{\epsilon,\delta',\delta}\zeta_n^{\epsilon,\delta',\delta}   <t, \ \ p(\delta',\delta)K^{\epsilon,\delta',\delta} > t+\eta \right) < \nu.
	\end{equation}

Finally, let
	\begin{equation*}
	\begin{aligned}
		A_1(\epsilon,\delta',\delta) \coloneqq \{ p(\delta',\delta)K^{\epsilon,\delta',\delta} \geq t \} \setminus \{ \left( e(\epsilon,\delta') \right)^{-1} p(\delta',\delta) \sum_{n=0}^\infty \xi_{n+1}^{\epsilon,\delta',\delta}\zeta_n^{\epsilon,\delta',\delta}    \geq t \} \\
		A_2(\epsilon,\delta',\delta) \coloneqq\{ \left( e(\epsilon,\delta') \right)^{-1} p(\delta',\delta) \sum_{n=0}^\infty \xi_{n+1}^{\epsilon,\delta',\delta}\zeta_n^{\epsilon,\delta',\delta}    \geq t \} \setminus  \{ p(\delta',\delta)K^{\epsilon,\delta',\delta} \geq t \} 
	\end{aligned}
	\end{equation*}
    It is clear that
\begin{equation*}
\begin{aligned}
    &\abs{ \mathbb{P}_z\left(\left( e(\epsilon,\delta') \right)^{-1} p(\delta',\delta) \sum_{n=0}^\infty \xi_{n+1}^{\epsilon,\delta',\delta}\zeta_n^{\epsilon,\delta',\delta}    \geq t \right) - \mathbb{P}_z \left(p(\delta',\delta)K^{\epsilon,\delta',\delta} \geq t\right)} \\
    &\leq \mathbb{P}_z(A_1(\epsilon,\delta',\delta)) + \mathbb{P}(A_2(\epsilon,\delta',\delta))
\end{aligned}
\end{equation*}
and thus due to \eqref{pf abs 1}, it suffices to estimate $\mathbb{P}_z(A_1(\epsilon,\delta',\delta))$ and $\mathbb{P}_z(A_2(\epsilon,\delta',\delta))$.
For $A_1(\epsilon,\delta',\delta)$,
\begin{equation*}
\begin{aligned}
    &\mathbb{P}_z(A_1(\epsilon,\delta',\delta))
    \leq \mathbb{P}_z \left( \left( e(\epsilon,\delta') \right)^{-1} p(\delta',\delta) \sum_{n=0}^\infty \xi_{n+1}^{\epsilon,\delta',\delta}\zeta_n^{\epsilon,\delta',\delta}   <t, \ \ p(\delta',\delta)K^{\epsilon,\delta',\delta} > t+\eta \right) \\
    & \ \ \ \ \ \ \ \ + \mathbb{P}_z \left( t \leq  p(\delta',\delta)K^{\epsilon,\delta',\delta} \leq t+\eta \right)
\end{aligned}
\end{equation*}
Thanks to \eqref{pf abs 1}, for every $\nu > 0$ there exists $\eta_\nu > 0$, $\delta'_{1,\nu,t} > 0$ such that for every $\delta' \in (0, \delta'_{1,\nu,t})$ there exists $\epsilon_{1,\nu,\delta',t} > 0$ such that
\begin{equation*}
    \sup_{z \in \mathcal{S}_{\epsilon}} \mathbb{P}_z \left( t \leq  p(\delta',\delta)K^{\epsilon,\delta',\delta} \leq t+\eta_\nu \right) \leq \frac{\nu}{4}, \quad \epsilon \in (0,\epsilon_{1,\nu,\delta',t}).
\end{equation*}
By \eqref{pf abs 3}, there exists $\delta'_{2,\nu,t} \in (0,\delta'_{1,\nu,t})$ such that for every $\delta' \in (0, \delta'_{2,\nu,t})$ there exists $\epsilon_{2,\nu,\delta',t} \in (0,\epsilon_{1,\nu,\delta',t})$ such that
\begin{equation*}
    \sup_{z \in \mathcal{S}_{\epsilon}} \mathbb{P}_z \left( \left( e(\epsilon,\delta') \right)^{-1} p(\delta',\delta) \sum_{n=0}^\infty \xi_{n+1}^{\epsilon,\delta',\delta}\zeta_n^{\epsilon,\delta',\delta}   <t, \ \ p(\delta',\delta)K^{\epsilon,\delta',\delta} > t+\eta_\nu \right) \leq \frac{\nu}{4}, \ \ \epsilon \in (0,\epsilon_{2,\nu,\delta',t}).
\end{equation*}
It is similar for $A_2(\epsilon,\delta',\delta)$ with the help of \eqref{pf abs 1} and \eqref{pf abs 2}.
We then complete the proof.

\section{List of domain notation used in Sections \ref{sec exit}--\ref{sec exit exp}}\label{app list dom}
We provide here a summary of the domain notation frequently used in Sections \ref{sec exit}--\ref{sec exit exp}.
We begin with some preliminary notation introduced in Section \ref{sec prelim}:
\begin{center}
\begin{tabular}{ | l | l | l | } 
	\hline
  Notation & Definition & Introduced in Section\\ 
  \hline
  $d_\Gamma$ & the metric on the graph & \ref{sec graph} \\ 
  \hline
  $d$ & the Euclidean distance & \ref{sec graph} \\ 
  \hline
  $B(O,r)$ & the Euclidean ball in $\mathbb{R}^d$ centered at $O$ with radius $r$ & \ref{sec narrow tube} \\ 
  \hline
  $z^*_{j,k}$ & $B(O_j,\rho_j r_j(\epsilon)) \cap I_k$ &  \ref{sec narrow tube} \\ 
  \hline
\end{tabular}
\end{center}

Next, we summarize the domains used from Section \ref{sec exit} to \ref{sec exit exp}.
Since our focus in these sections (except Section \ref{sec proof lem place}) is on the local behavior near a vertex, we omit the subscript $j$ for clarity.
All definitions are introduced in the section where they first appear, and their precise formulations can be found therein.

\begin{center}
\begin{tabular}{ | r | l |l | } 
	\hline
  Domain & Definition & Used in Section \\ 
  \hline
  $C_\epsilon(\delta)$ & $\{ z \in G_\epsilon: d_\Gamma(\Pi(z),O) =\delta \}$ & \ref{sec exit}--\ref{sec exit exp}  \\ 
  \hline
  $C_\epsilon^k(\delta)$ & $C_\epsilon(\delta) \cap \Pi^{-1}(I_k)$ & \ref{sec exit}--\ref{sec exit exp}  \\ 
  \hline
  $B_\epsilon(\delta)$ & $\{ z \in G_\epsilon: d_\Gamma(\Pi(z),O) \leq \delta \}$ & \ref{sec exit}, \ref{sec exit place} \\ 
  \hline
  $\Omega_{1,\epsilon}$ & $B(O,\rho r(\epsilon))$ & \ref{sec proof lem}--\ref{sec exit place} \\ 
  \hline
  $\Omega_{2,\epsilon,\delta}^k$ & the region enclosed by $C_\epsilon^k(\rho r(\epsilon)+3\epsilon)$, $C_\epsilon^k(\delta)$, $\partial G_\epsilon$ & \ref{sec proof lem} \\ 
  \hline
  $\Omega_{2,\epsilon,\delta}$ & $\bigcup_k \Omega_{2,\epsilon,\delta}^k$ & \ref{sec proof lem} \\ 
  \hline
  $\Omega_{3,\epsilon,\delta}$ & $B_\epsilon(\delta) \setminus (\Omega_{1,\epsilon} \cup \Omega_{2,\epsilon,\delta})$ & \ref{sec proof lem} \\ 
  \hline
  $\Omega_{4,\epsilon}^k$ & the region enclosed by  $\Gamma_{2,\epsilon}^k$, $\Gamma_{3,\epsilon}^k$, $\partial G_\epsilon$ & \ref{sec exit time apriori} \\ 
  \hline
  $\Omega_{4,\epsilon}$ & $\bigcup_k \Omega_{4,\epsilon}^k$ & \ref{sec exit time apriori} \\ 
  \hline
  $\Omega_{5,\epsilon,\delta}^k$ & the region enclosed by $C_\epsilon^k(\rho r(\epsilon)+2\epsilon)$, $C_\epsilon^k(\delta)$, $\partial G_\epsilon$ & \ref{sec exit time apriori} \\ 
  \hline
  $\Omega_{5,\epsilon,\delta}$ & $\bigcup_k \Omega_{5,\epsilon,\delta}^k$ & \ref{sec exit time apriori} \\ 
  \hline
  $\Omega_{6,\epsilon}^k$ & $\left\{ z \in \Omega_{1,\epsilon}: d(z,z_k^*) \geq \tilde{\rho} r(\epsilon) \right\}$ & \ref{sec exit place} \\ 
  \hline
  $\Omega_{6,\epsilon}$ & $\bigcap_k \Omega_{6,\epsilon}^k$ & \ref{sec exit place} \\ 
  \hline
  $\Omega_{7,\epsilon,\delta}^k$ & $\left( B_\epsilon(\delta+1) \cap \Pi^{-1}(I_k) \right) \cup \Omega_{1,\epsilon}$ & \ref{sec exit place} \\ 
  \hline
  $\Omega_{8,\epsilon,\delta}^k$ & the region enclosed by $\Gamma_{4,\epsilon}^k$, $C_\epsilon^k(\delta+1)$, $\partial G_\epsilon$ & \ref{sec exit place} \\ 
  \hline
  $\Gamma_{1,\epsilon}$ & $\partial B(O,\rho r(\epsilon)) \setminus \partial G_\epsilon$ & \ref{sec exit time apriori}, \ref{sec exit place} \\ 
  \hline
  $\Gamma_{1,\epsilon}^k$ & $\Gamma_{1,\epsilon} \cap \Pi^{-1}(I_k)$ & \ref{sec exit time apriori}, \ref{sec exit place} \\ 
  \hline
  $\Gamma_{2,\epsilon}^k$ &  $\{ z \in \Omega_{1,\epsilon} : d(z,z_k^*)= 4 \lambda_k \epsilon \}$ & \ref{sec exit time apriori} \\ 
  \hline
  $\Gamma_{2,\epsilon}$ & $\bigcup_k \Gamma_{2,\epsilon}^k$ & \ref{sec exit time apriori} \\ 
  \hline
  $\Gamma_{3,\epsilon}^k$ &  $C_\epsilon^k(3\epsilon + \rho r(\epsilon))$ & \ref{sec exit time apriori} \\ 
  \hline
  $\Gamma_{3,\epsilon}$ & $\bigcup_k \Gamma_{3,\epsilon}^k$ & \ref{sec exit time apriori} \\ 
  \hline
  $\Gamma_{4,\epsilon}^k$ & $\left\{ z \in \Omega_{1,\epsilon} : d(z,z_k^*)= \tilde{\rho}  r(\epsilon) \right\}$ & \ref{sec exit place} \\ 
  \hline
  $\Gamma_{4,\epsilon}$ & $\bigcup_k \Gamma_{4,\epsilon}^k$ & \ref{sec exit place} \\ 
  \hline
\end{tabular}
\end{center}

\end{appendices}

%%%%%%%%%%%%%%%%%%%%%%%%%%%%%%%%%%%%%%%%%%%%%%
%% Support information, if any,             %%
%% should be provided in the                %%
%% Acknowledgements section.                %%
%%%%%%%%%%%%%%%%%%%%%%%%%%%%%%%%%%%%%%%%%%%%%%
\section*{Acknowledgment}
The author would like to thank Sandra Cerrai for bringing this problem to his attention and for many helpful discussions.
He is also grateful to Leonid Koralov for drawing his attention to the paper \cite{SMP} and for valuable discussions.
Finally, he thanks Mark Freidlin for insightful conversations regarding the martingale problem.


\begin{thebibliography}{99}

\bibitem {DPT}
S.~Albeverio, S.~Kusuoka,
{\em Diffusion processes in thin tubes and their limits on graphs}, Annals of Probability 40 (2012), pp. 2131--2167.

\bibitem {LPNET}
H.~Ammari, K.~Kalimeris, H.~Kang, H.~Lee,
{\em Layer potential techniques for the narrow escape problem}, Journal de Math\'ematiques Pures et Appliqu\'ees 97 (2012), pp. 66--84.

\bibitem{BH}
R.~F.~Bass, P.~Hsu,
{\em Some potential theory for reflecting Brownian motion in H\"older and Lipschitz domains}, Annals of Probability 19 (1991), pp. 486--508.

\bibitem{QGbook}
G.~Berkolaiko, P.~Kuchment,
{\sc Introduction to quantum graphs}, American Mathematical Society, 2013.

\bibitem{nlpde}
 I.~Birindelli, A.~Briani, H.~Ishii,
{\em Fully nonlinear elliptic PDEs in thin domains with oblique boundary condition}, SIAM Journal on Mathematical Analysis 58 (2026), pp. 1232--1256.

\bibitem{sticky}
S.~Bonaccorsi, M.~D'Ovidio,
{\em Sticky Brownian motion on star graphs}, Fractional Calculus and Applied Analysis 27 (2024), pp. 2859--2891.


\bibitem{BCM}
K.~Burdzy, Z.~Q.~Chen, D.~E.~Marshall,
{\em Traps for reflected Brownian motion}, Mathematische Zeitschrift 252 (2006), pp. 103--132.

\bibitem{NETsecond}
 C.~Caginalp, X.~Chen,
{\em Analytical and numerical results for an escape problem}, Archive for Rational Mechanics and Analysis 203 (2011), pp. 329--342.

\bibitem{CLNV26}
 L.~Carillo, T.~Leli\`evre, T.~Normand, U.~Vaes,
{\em A spectral approach to the narrow escape problem in two-dimensional domains}, arXiv:2608.05007.

\bibitem{C+26}
 L.~Carillo, T.~Leli\`evre, T.~Normand, G.~Stoltz, U.~Vaes,
{\em The narrow escape problem in arbitrary dimension}, arXiv:2608.02212.

\bibitem{CF}
 S.~Cerrai, M.~I.~Freidlin,
{\em SPDEs on narrow domains and on graphs: an asymptotic result}, Annales de l'Institut Henri Poincar\'{e} 53 (2017), pp. 865--899.

\bibitem{QSD}
N.~Champagnat, D.~Villemonais,
{\em Exponential convergence to quasi-stationary distribution and Q-process}, Probability Theory and Related Fields 164 (2016), pp. 243--283.

\bibitem{NET4}
X.~Chen, A.~Friedman,
{\em Asymptotic analysis for the narrow escape problem}, SIAM Journal on Mathematical Analysis 43 (2011), pp. 2542--2563.

\bibitem{C+20}
 S.~Cohn, G.~Iyer, J.~Nolen, R.~L.~Pego,
{\em Anomalous diffusion in comb-shaped domains and graphs}, Communications in Mathematical Sciences 18 (2020), pp. 1815--1862.

\bibitem{D+03}
 L.~Dagdug, A.~M.~Berezhkovskii, S.~Y.~Shvartsman, G.~H.~Weiss,
{\em Equilibration in two chambers connected by a capillary}, Journal of Chemical Physics 119 (2003), pp. 12473--12478.

\bibitem{spectral}
 P.~Exner, O.~Post,
{\em Convergence of spectra of graph-like thin manifolds}, Journal of Geometry and Physics 54 (2005), pp. 77--115.

\bibitem{Feller54}
 W.~Feller,
{\em Diffusion processes in one dimension}, Transactions of the American Mathematical Society 77 (1954), pp. 1--31.

\bibitem{SDE1}
M.~I.~Freidlin,
{\sc Markov processes and differential equations: asymptotic problems}, Lectures in Mathematics ETH Z\"{u}rich. Birkh\"{a}user Basel, 1996.

\bibitem{FH13} 
M.~I.~Freidlin, W.~Hu, {\em On diffusion in narrow random channels}, Journal of Statistical Physics 152 (2013), pp. 136--158.

 
\bibitem{SDE2} 
M.~I.~Freidlin, A.~D.~Wentzell, {\em Diffusion processes on graphs and the averaging principle}, Annals of Probability 21 (1993), pp. 2215--2245.

\bibitem{FW12} 
M.~I.~Freidlin, A.~D.~Wentzell, {\em On the Neumann problem for PDEs with a small parameter and the corresponding diffusion processes}, Probability Theory and Related Fields 152 (2012), pp. 101--140.


\bibitem{PDEBook} D.~Gilbarg, N.~S.~Trudinger, {\sc Elliptic partial differential equations of second order}, Springer Berlin, Second Edition, 2001.

\bibitem{NETthird}
 D.~Holcman, Z.~Schuss,
{\em Diffusion laws in dendritic spines}, Journal of Mathematical Neuroscience 1 (2011), pp. 1--14.

\bibitem{NET}
 D.~Holcman, Z.~Schuss,
{\em The narrow escape problem}, SIAM Review 56 (2014), pp. 213--257.

\bibitem{Hsu25+}
 W.~T.~Hsu,
{\em Metastability of diffusion processes in narrow tubes}, arXiv:2511.21548.

\bibitem{Imb26}
 C.~Imbert,
{\em De Giorgi's regularity theory for elliptic, parabolic and kinetic equations}, arXiv:2601.15238.

\bibitem{SMP}
 I.~Imtiyas, L.~Koralov,
{\em Metastable distributions of semi-Markov processes}, arXiv:2411.04795.

\bibitem{Ito} K.~It\^{o}, H.~P.~McKean, {\sc Diffusion processes and their sample paths}, Springer Berlin, 1996.

\bibitem{EPD}
 R.~Z.~Khaminskii,
{\em Ergodic properties of recurrent diffusion processes and stabilization of the solution to the Cauchy problem for parabolic equations}, Theory of Probability and Its Application 5 (1960), pp. 179--196.

\bibitem{KZ01}
 P.~Kuchment, H.~Zeng
{\em Convergence of spectra of mesoscopic systems collapsing onto a graph}, Journal of Mathematical Analysis and Applications 258 (2001), pp. 671--700.

\bibitem{NETsp}
 T.~Lelievre, M.~Rachid, G.~Stoltz,
{\em A spectral approach to the narrow escape problem in the disk}, arXiv:2401.06903.

\bibitem{Li14}
X.~Li, 
{\em Matched asymptotic analysis to solve the narrow escape problem in a domain with a long neck}, Journal of Physics A: Mathematical and Theoretical 47 (2014) 505202.

\bibitem{LL24}
X.~Li, S.~Lin,
{\em Asymptotic analysis of the narrow escape problem in general shaped domain with several absorbing necks}, Journal of Physics A: Mathematical and Theoretical 57 (2024) 015206.

\bibitem{NETBN1}
X.~Li, H.~Lee, Y.~Wang,
{\em Asymptotic analysis of the narrow escape problem in dendritic spine shaped domain: three dimensions}, Journal of Physics A: Mathematical and Theoretical 50 (2017) 325203.

\bibitem{LS84}
P.~L.~Lions, A.~S.~Sznitman,
{\em Stochastic differential equations with reflecting boundary conditions}, Communications on Pure and Applied Mathematics 37 (1984), pp. 511--537.

\bibitem{UEB}
J.~Loper,
{\em Uniform ergodicity for Brownian motion in a bounded convex set}, Journal of Theoretical Probability 33 (2020) , pp. 22--35.

\bibitem{NETconj}
 S.~Louca,
{\em On the narrow escape problem}, Lecture Notes.


\bibitem{NETnetwork}
F.~Paquin-Lefebvre, K.~Basnayake, D.~Holcman,
{\em Narrow escape in composite domains forming heterogeneous networks}, Physica D: Nonlinear Phenomena 454 (2023) 133837.

\bibitem{Convbook} O.~Post, {\sc Spectral analysis on graph-like spaces}, Springer Berlin, 2012.

\bibitem{DST} G.~Raugel, {\sc Dynamics of partial differential equations on thin domains}, Dynamical systems. Lectures given at the 2nd session of the Centro Internazionale Matematico Estivo (CIME) held in Montecatini Terme, Italy, June 13--22, 1994.

\bibitem{MPD}  M.~Salins, K.~Spiliopoulos, {\em Markov process with spatial delay: path space characterization, occupation time and properties}, Stochastics and Dynamics 17 (2017) 1750042.

\bibitem{NETbook}
Z.~Schuss, {\sc Brownian dynamics at boundaries and interfaces}, Springer New York, 2013.

\bibitem{Spi09}
K.~Spiliopoulos,
{\em Wiener process with reflection in non-smooth narrow tubes}, Electronic Journal of Probability 14 (2009), pp. 2011--2037.

\bibitem{Diffbook}
D.~W.~Stroock, S.~R.~S.~Varadhan, {\sc Multidimensional diffusion processes}, Springer Berlin, 2006.



\end{thebibliography}
\end{document}